\documentclass[11pt,a4paper,reqno]{amsart} 
\usepackage{amsmath}
\usepackage{amssymb}
\usepackage{euscript}
\usepackage{multirow}
\usepackage{graphicx}
\usepackage{comment}
\usepackage{xcolor}
\usepackage[colorlinks=true,linktocpage=true,pagebackref=false, citecolor=black,linkcolor=black]{hyperref}
\usepackage{pifont}
\usepackage{tikz,pgfplots}
\pgfplotsset{compat=1.10}
\usepgfplotslibrary{fillbetween}
\usetikzlibrary{cd,arrows,decorations.pathmorphing,backgrounds,automata,positioning,fit,matrix}
\usepackage{caption}
\pgfplotsset{soldot/.style={color=black,only marks,mark=*}} \pgfplotsset{holdot/.style={color=black,fill=white,only marks,mark=*}}

\newtheorem{thm}{Theorem}[section]

\newtheorem{lem}[thm]{Lemma}

\newtheorem{cor}[thm]{Corollary}

\theoremstyle{definition}
\newtheorem{defn}[thm]{Definition}

\theoremstyle{remark}
\newtheorem{remark}[thm]{Remark}

\newtheorem{example}[thm]{Example}
\newtheorem{examples}[thm]{Examples}

\newtheorem{prob}[thm]{Problem}
\newtheorem{quest2}[subsection]{Question}

\numberwithin{equation}{section}
\numberwithin{figure}{section}

 \newcommand{\R}{{\mathbb R}}
 \newcommand{\C}{{\mathbb C}}
 
\newcommand{\HH}{{\mathbb H}} 
\newcommand{\sph}{{\mathbb S}}

\newcommand{\psd}{{\mathcal P}}

 \newcommand{\Cont}{{\mathcal C}}

\newcommand{\Ff}{{\EuScript F}}
\newcommand{\Pp}{{\EuScript P}}
\newcommand{\Ss}{{\EuScript S}}
\newcommand{\Tt}{{\EuScript T}}
\newcommand{\Ee}{{\EuScript E}}

\newcommand{\Bb}{{\EuScript B}}
\newcommand{\Cc}{{\EuScript C}}
\newcommand{\pol}{{\EuScript K}}
\newcommand{\Qq}{{\EuScript Q}}
\newcommand{\Dd}{{\EuScript D}}

\newcommand{\Hh}{{\EuScript H}}
\newcommand{\Ll}{{\EuScript L}}
\newcommand{\Rr}{{\EuScript R}}
\newcommand{\Ii}{{\EuScript I}}

\newcommand{\im}{\operatorname{im}}

\newcommand{\Int}{\operatorname{Int}}
\newcommand{\cl}{\operatorname{Cl}}

\newcommand{\dist}{\operatorname{dist}}

\newcommand{\id}{\operatorname{id}}

\newcommand{\zar}{\operatorname{zar}}

\newcommand{\x}{{\tt x}} \newcommand{\y}{{\tt y}} 
\newcommand{\z}{{\tt z}} \renewcommand{\t}{{\tt t}}

\newcommand{\veps}{\varepsilon}

\newcommand{\ol}{\overline}

\newcommand{\sector}[1]{\Dd^e_{#1}}
\newcommand{\seg}[1]{\Ss^e_{#1}}
\newcommand{\tri}[1]{\Tt^e_{#1}}
\newcommand{\hsector}[1]{\Dd^h_{#1}}
\newcommand{\htri}[1]{\Tt^h_{#1}}
\newcommand{\hseg}[1]{\Ss^h_{#1}}
\newcommand{\pseg}[1]{\Ss^p_{#1}}
\newcommand{\ptri}[1]{\Tt^p_{#1}}

\newcommand{\sstar}[1]{{\rm St}(#1)}

\newcommand{\xmark}{\text{\ding{56}}}

\begin{document}

\title[On polynomial images of a closed ball]{On polynomial images of a closed ball}

\author{Jos\'e F. Fernando}
\address{Departamento de \'Algebra, Geometr\'\i a y Topolog\'\i a, Facultad de Ciencias Matem\'aticas, Universidad Complutense de Madrid, Plaza de Ciencias 3, 28040 MADRID (SPAIN)}
\email{josefer@mat.ucm.es}


\author{Carlos Ueno}
\address{Departamento de Matem\'aticas, CEAD Profesor F\'elix P\'erez Parrilla, C/ Dr. Garc\'\i a Castrillo 22, 35005 Las Palmas de Gran Canaria (SPAIN)}
\email{cuenjac@gmail.com}

\date{04/04/2022}
\subjclass[2010]{Primary: 14P10, 52B99; Secondary: 52A20, 41A10, 26D15, 13P25.}
\keywords{Closed ball, polynomial maps and images, polynomial paths inside semialgebraic sets, $n$-dimensional bricks, PL semialgebraic sets, semialgebraic sets connected by analytic paths.}

\begin{abstract}
In this work we approach the problem of determining which (compact) semialgebraic subsets of $\R^n$ are images under polynomial maps $f:\R^m\to\R^n$ of the closed unit ball $\ol{\Bb}_m$ centered at the origin of some Euclidean space $\R^m$ and that of estimating (when possible) which is the smallest $m$ with this property. Contrary to what happens with the images of $\R^m$ under polynomial maps, it is quite straightforward to provide basic examples of semialgebraic sets that are polynomial images of the closed unit ball. For instance, simplices, cylinders, hypercubes, elliptic, parabolic or hyperbolic segments (of dimension $n$) are polynomial images of the closed unit ball in $\R^n$. 

The previous examples (and other basic ones proposed in the article) provide a large family of `$n$-bricks' and we find necessary and sufficient conditions to guarantee that a finite union of `$n$-bricks' is again a polynomial image of the closed unit ball either of dimension $n$ or $n+1$. In this direction, we prove: {\em A finite union $\Ss$ of $n$-dimensional convex polyhedra is the image of the $n$-dimensional closed unit ball $\ol{\Bb}_n$ if and only if $\Ss$ is connected by analytic paths}.

The previous result can be generalized using the `$n$-bricks' mentioned before and we show: {\em If $\Ss_1,\ldots,\Ss_\ell\subset\R^n$ are `$n$-bricks', the union $\Ss:=\bigcup_{i=1}^\ell\Ss_i$ is the image of the closed unit ball $\ol{\Bb}_{n+1}$ of $\R^{n+1}$ under a polynomial map $f:\R^{n+1}\to\R^n$ if and only if $\Ss$ is connected by analytic paths}. 
\end{abstract}

\dedicatory{Dedicated to Prof. J.M. Gamboa on the occasion of his 65th birthday}

\maketitle

\section{Introduction}\label{s1}

A map $f:=(f_1,\ldots,f_n):\R^m\to\R^n$ is \em polynomial \em if its components $f_k\in\R[\x_1,\ldots,\x_m]$ are polynomials. Analogously, $f$ is \em regular \em if its components can be represented as quotients $f_k:=\frac{g_k}{h_k}$ of two polynomials $g_k,h_k\in\R[\x_1,\ldots,\x_m]$ such that $h_k$ never vanishes on $\R^m$. A subset $\Ss\subset\R^n$ is \em semialgebraic \em when it has a description by a finite boolean combination of polynomial equalities and inequalities. The category of semialgebraic sets is closed under basic boolean operations but also under usual topological operations: taking closures (denoted by $\cl(\cdot)$), interiors (denoted by $\Int(\cdot)$), connected components, etc. If $\Ss\subset\R^m$ and $\Tt\subset\R^n$ are semialgebraic sets, a map $f:\Ss\to\Tt$ is \em semialgebraic \em if its graph is a semialgebraic set. By Tarski-Seidenberg's principle \cite[\S1.4]{bcr} the image of a semialgebraic map (and in particular of a polynomial or a regular map) is a semialgebraic set. 

In \cite{kps} the authors develop a computational study of images under polynomial maps $f:\R^3\to\R^2$ (and the corresponding convex hulls) of compact (principal) semialgebraic subsets $\{h\geq0\}\subset\R^3$, where $h\in\R[\x_1,\x_2,\x_3]$. This includes for example the case of a $3$-dimensional closed unit ball $\ol{\Bb}_3$ centered at the origin. In \cite[\S5.Prob.1]{kps} it is proposed the following concrete problem:

\begin{prob}\label{probl2}
Let $\Pp$ be an arbitrary (compact) convex polygon in $\R^2$. Construct explicit polynomials $f$ and $g$ in $\R[{\tt u},{\tt v},{\tt w}]$ such that $\Pp=(f,g)(\ol{\Bb}_3)$.
\end{prob}

Sturmfels suggested us during the 2018 Santal\'o Conference (presented by him) in Universidad Complutense de Madrid to confront the previous problem taking into account our knowledge in the subject of polynomial images of affine spaces. This suggestion was the starting point of the present article, where we make an extended study of the $n$-dimensional semialgebraic subsets of $\R^n$ that are images under a polynomial map $f:\R^m\to\R^n$ of the $m$-dimensional closed unit ball $\ol{\Bb}_m$ for some $m\geq n$. We will be mainly concerned with the cases $m=n$ and $m=n+1$, and the first main result is Theorem \ref{main1-i}, which solves Problem \ref{probl2} (as a particular case) in its natural generalization to arbitrary dimension. Recall that a semialgebraic set $\Ss\subset\R^n$ is \em connected by analytic paths \em if for each pair of points $p_1,p_2\in\Ss$ there exists an analytic path $\alpha:[0,1]\to\Ss$ such that $\alpha(0)=p_1$ and $\alpha(1)=p_2$. Contrary to what happens when dealing with continuous paths, even if $p_1$ and $p_2$ are connected by an analytic path, and $p_2$ and $p_3$ are also so connected, these may not imply that $p_1$ and $p_3$ are connected by an analytic path. We borrow the following enlightening example from \cite[Ex.7.12]{f2}.

\begin{example}
There exist (path) connected semialgebraic sets that are not connected by analytic paths. Let $
\Ss:=\{(4\x^2-\y^2)(4\y^2-\x^2)\geq0,\y\geq0\}\subset\R^2$, which is a (path) connected semialgebraic set, see Figure \ref{fig11}. We claim: {\em The semialgebraic set $\Ss$ is not connected by analytic paths.} 

\begin{figure}[!ht]
\begin{center}
\begin{tikzpicture}[scale=0.75]

\draw[fill=gray!60,opacity=0.75,dashed,draw] (4.5,1) -- (0.5,3) arc (153.43494882292201:116.56505117707798:4.47213595499958cm) -- (4.5,1) -- (8.5,3) arc (26.56505117707799:63.43494882292202:4.47213595499958cm) -- (4.5,1);

\draw[->,thick=1.5pt] (4.5,1) -- (0.5,3);
\draw[->,thick=1.5pt] (4.5,1) -- (6.5,5);
\draw[->,thick=1.5pt] (4.5,1) -- (8.5,3);
\draw[->,thick=1.5pt] (4.5,1) -- (2.5,5);

\draw[->] (4.5,0) -- (4.5,6);
\draw[->] (0,1) -- (9,1);

\draw[fill=black] (4.5,1) circle (0.75mm);
\draw[fill=black] (5.5,2) circle (0.75mm);
\draw[fill=black] (3.5,2) circle (0.75mm);

\draw (6.1,2.4) node{\footnotesize$(1,1)$};
\draw (2.9,2.4) node{\footnotesize$(-1,1)$};

\draw (2,4.25) node{\small$\Cc_1$};
\draw (7,4.25) node{\small$\Cc_2$};
\draw (5,5) node{\small$\Ss$};

\end{tikzpicture}
\end{center}
\caption{Semialgebraic set $\Ss:=\{(4\x^2-\y^2)(4\y^2-\x^2)\geq0,\y\geq0\}\subset\R^2$\label{fig11}}
\end{figure}
\end{example}
\begin{proof}
Pick the points $p_1:=(-1,1),p_2:=(1,1)\in\Ss$ and assume that there exists an analytic path $\alpha:[0,1]\to\Ss$ such that $\alpha(0)=p_1$ and $\alpha(1)=p_2$. Consider the closed semialgebraic sets $\Cc_1:=\Ss\cap\{\x\leq0\}$ and $\Cc_2:=\Ss\cap\{\x\geq0\}$, which satisfy $\Ss=\Cc_1\cup\Cc_2$. Both $\Cc_1$ and $\Cc_2$ are convex, so they are connected by analytic paths (in fact, they are connected by segments) and $\Cc_1\cap\Cc_2=\{(0,0)\}$. Define $\Cc_i^*:=\{\lambda w:\ w\in\Cc_i,\ \lambda\in\R\}$ for $i=1,2$. Note that $\Ss\cap\{\x<0\}=\Cc_1\setminus\{(0,0)\}$ and $\Ss\cap\{\x>0\}=\Cc_2\setminus\{(0,0)\}$ are pairwise disjoint open subsets of $\Ss$. We have $0\in\alpha^{-1}(\Cc_1\setminus\{(0,0)\})$ and $1\in\alpha^{-1}(\Cc_2\setminus\{(0,0)\})$, so $t_0:=\inf(\alpha^{-1}(\Cc_2\setminus\{(0,0)\}))>0$. As $\alpha$ is a (non-constant) analytic path, $t_0\in\cl(\alpha^{-1}(\Cc_1\setminus\{(0,0)\}))\cap\cl(\alpha^{-1}(\Cc_2\setminus\{(0,0)\}))$ and $\alpha(t_0)=(0,0)$. As $\alpha^{-1}(\Cc_1\setminus\{(0,0)\})$ and $\alpha^{-1}(\Cc_2\setminus\{(0,0)\})$ are pairwise disjoint open subsets of $[0,1]$, there exists $\veps>0$ such that 
$$
\alpha((t_0-\veps,t_0))\subset\Cc_1\setminus\{(0,0)\}\quad\text{and}\quad\alpha((t_0,t_0+\veps))\subset\Cc_2\setminus\{(0,0)\}.
$$
The tangent direction to $\im(\alpha|_{(t_0-\veps,t_0+\veps)})$ at $\alpha(t_0)=(0,0)$ is the line generated by the vector 
$$
w=\lim_{t\to t_0}\frac{\alpha(t)-\alpha(t_0)}{(t-t_0)^k}=\begin{cases}
\lim_{t\to t_0^{+}}\frac{\alpha(t)-(0,0)}{(t-t_0)^k}\in\Cc_1^*\setminus\{(0,0)\},\\
\lim_{t\to t_0^{-}}\frac{\alpha(t)-(0,0)}{(t-t_0)^k}\in\Cc_2^*\setminus\{(0,0)\},
\end{cases}
$$
where $k$ is the multiplicity of $t_0$ as a root of $\|\alpha\|$. This is a contradiction (because $\Cc_1^*\cap\Cc_2^*=\{(0,0)\}$), so $\Ss$ is not connected by analytic paths.
\end{proof}

\begin{thm}\label{main1-i}
Let $\Ss\subset\R^n$ be the union of a finite family of $n$-dimensional convex (compact) polyhedra. The following assertions are equivalent:
\begin{itemize}
\item[(i)] $\Ss$ is connected by analytic paths.
\item[(ii)] There exists a polynomial map $f:\R^{n+1}\to\R^n$ such that $f(\ol{\Bb}_{n+1})=\Ss$.
\end{itemize}
\end{thm}

Squeezing the arguments used to prove Theorem \ref{main1-i} (and increasing the complexity of the involved constructions) we can go further and prove the following sharp result.

\begin{thm}\label{main1-i2}
Let $\Ss\subset\R^n$ be the union of a finite family of $n$-dimensional convex (compact) polyhedra. The following assertions are equivalent:
\begin{itemize}
\item[(i)] $\Ss$ is connected by analytic paths.
\item[(ii)] There exists a polynomial map $f:\R^n\to\R^n$ such that $f(\ol{\Bb}_n)=\Ss$.
\end{itemize}
\end{thm}

We included an independent proof of Theorem \ref{main1-i} because it is enlightening to show the techniques we develop in this article and it is less demanding from the point of view of complexity than the one of Theorem \ref{main1-i2}. In Figure \ref{fig12} we present two examples of polygons to illustrate Theorems \ref{main1-i} and \ref{main1-i2}: the first one is not connected by analytic paths, whereas the second one is connected by analytic paths.

\begin{center}
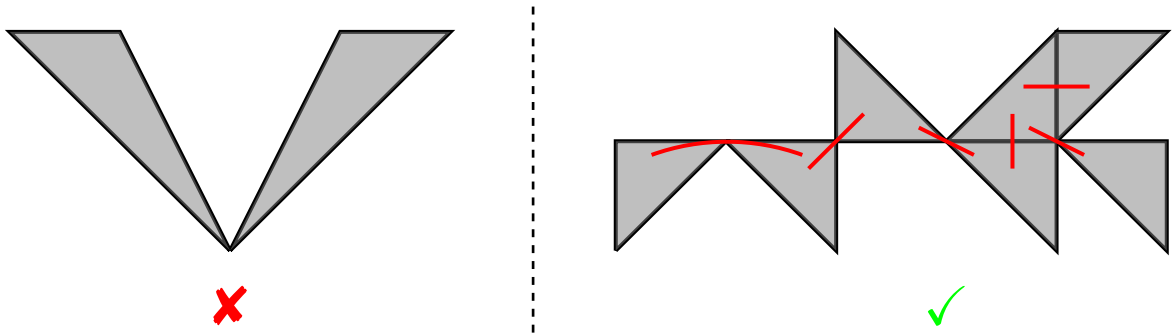
\begin{figure}[ht]
\begin{tikzpicture}[scale=1.45]
\draw[line width=1.5pt] (-0.5,0) -- (-2.5,2) -- (-1.5,2) -- (-0.5,0); 
\draw[line width=1.5pt] (-0.5,0) -- (1.5,2) -- (0.5,2) -- (-0.5,0); 

\draw[dashed,line width=1pt] (2.25,-0.75) -- (2.25,2.25);

\draw[line width=1.5pt] (3,0) -- (3,1) -- (4,1) -- (3,0);
\draw[line width=1.5pt] (4,1) -- (5,1) -- (5,0) -- (4,1);
\draw[line width=1.5pt] (5,1) -- (5,2) -- (6,1) -- (5,1);
\draw[line width=1.5pt] (6,1) -- (7,1) -- (7,2) -- (6,1);
\draw[line width=1.5pt] (6,1) -- (7,1) -- (7,0) -- (6,1);
\draw[line width=1.5pt] (7,1) -- (7,2) -- (8,2) -- (7,1);
\draw[line width=1.5pt] (7,1) -- (8,1) -- (8,0) -- (7,1);

\draw[fill=gray!100,opacity=0.5,draw=none] (-0.5,0) -- (-2.5,2) -- (-1.5,2) -- (-0.5,0); 
\draw[fill=gray!100,opacity=0.5,draw=none] (-0.5,0) -- (1.5,2) -- (0.5,2) -- (-0.5,0);

\draw[fill=gray!100,opacity=0.5,draw=none] (3,0) -- (3,1) -- (4,1) -- (3,0);
\draw[fill=gray!100,opacity=0.5,draw=none] (4,1) -- (5,1) -- (5,0) -- (4,1);
\draw[fill=gray!100,opacity=0.5,draw=none] (5,1) -- (5,2) -- (6,1) -- (5,1);
\draw[fill=gray!100,opacity=0.5,draw=none] (6,1) -- (7,1) -- (7,2) -- (6,1);
\draw[fill=gray!100,opacity=0.5,draw=none] (6,1) -- (7,1) -- (7,0) -- (6,1);
\draw[fill=gray!100,opacity=0.5,draw=none] (7,1) -- (7,2) -- (8,2) -- (7,1);
\draw[fill=gray!100,opacity=0.5,draw=none] (7,1) -- (8,1) -- (8,0) -- (7,1);

\draw (6,-0.5) node{\huge\color{green}$\checkmark$};
\draw (-0.5,-0.5) node{\huge\color{red}$\xmark$};

\draw[color=red,line width=1.5pt] (3.325,0.875) arc (110:70:2cm);
\draw[color=red,line width=1.5pt] (4.75,0.75) -- (5.25,1.25);
\draw[color=red,line width=1.5pt] (5.75,1.125) -- (6.25,0.875);
\draw[color=red,line width=1.5pt] (6.6,0.75) -- (6.6,1.25);
\draw[color=red,line width=1.5pt] (6.7,1.5) -- (7.3,1.5);
\draw[color=red,line width=1.5pt] (6.75,1.125) -- (7.25,0.875);

\end{tikzpicture}
\caption{Illustration of Theorems \ref{main1-i} and \ref{main1-i2}\label{fig12}}
\end{figure}
\end{center}

It seems natural to make abstraction of the techniques developed in the proof of Theorem \ref{main1-i} and use them to provide further examples of polynomial images of a closed unit ball. We will be concerned about the representativity as polynomial images of $(m+1)$-dimensional closed balls of finite unions of polynomial images $\Ss_k\subset\R^n$ of $m$-dimensional closed balls. As one can expect, we need to ask some mild additional conditions concerning the semialgebraic sets $\Ss_k$. Let ${\mathcal C}^0([0,1])$ denote the ring of continuous functions on the interval $[0,1]$.

\begin{defn}\label{mbrickd}
We say that an $n$-dimensional semialgebraic set $\Ss\subset\R^n$ is an \em $m$-brick \em if there exists a homotopy $H:=(H_1,\ldots,H_n):[0,1]\times\ol{\Bb}_m\to\Ss$ such that:
\begin{itemize}
\item[(i)] Each $H_i$ is the restriction to $[0,1]\times\ol{\Bb}_m$ of a polynomial of the ring ${\mathcal C}^0([0,1])[\x_1,\ldots,\x_m]$.
\item[(ii)] $H(\{0\}\times\ol{\Bb}_m)=\Ss$ and $H(1,\cdot)$ is a constant map.
\item[(iii)] $H(\{t\}\times\ol{\Bb}_m)\subset\Int(\Ss)$ for each $t\in(0,1)$.
\end{itemize}
\end{defn}
Roughly speaking, we have a family of (restrictions of) polynomial maps $H_\lambda:\ol{\Bb}_m\to\Ss$ that deforms $\Ss$ to a point $p\in\Ss$ and the intermediate sets are contained in $\Int(\Ss)$. With this definition in mind, we state the following result. Figures \ref{fig2} and \ref{fig3} illustrates some applications of Theorem \ref{main2-i}.
\begin{center}
\begin{figure}[ht]
\includegraphics[width=15cm]{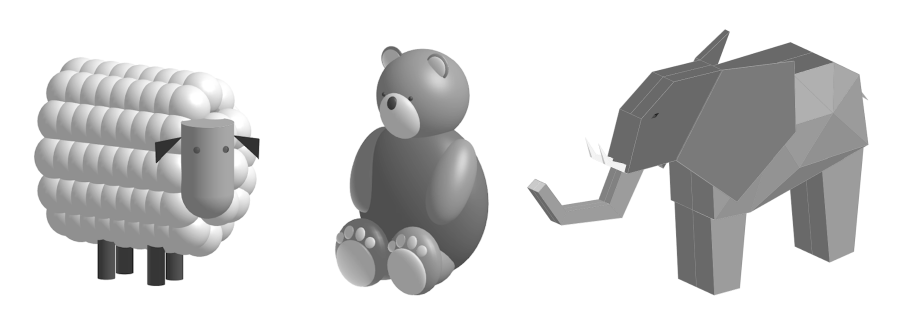}
\caption{Examples of polynomial images of the $4$-dimensional closed unit ball (application of Theorem \ref{main2-i}): little sheep (Sch\"afchen), Teddy bear, elephant.\label{fig2}}
\end{figure}
\end{center}
\begin{thm}[Sch\"afchen's theorem]\label{main2-i}
Let $\Ss\subset\R^n$ be the union of a finite family of $m$-bricks. The following assertions are equivalent:
\begin{itemize}
\item[(i)] $\Ss$ is connected by analytic paths.
\item[(ii)] There exists a polynomial map $f:\R^{m+1}\to\R^n$ such that $f(\ol{\Bb}_{m+1})=\Ss$.
\end{itemize}
\end{thm}

\begin{examples}[$m$-bricks]\label{bricks}
In Section \ref{s2} we present many examples of $m$-bricks for some $m\geq1$. We will prove: \em if $\Ss\subset\R^n$ is a convex $n$-dimensional semialgebraic set and it is the image of the closed unit ball $\ol{\Bb}_m$, then it is an $m$-brick\em. Having this in mind, we find large families of $m$-bricks:
\begin{itemize}
\item[(i)] $n$-dimensional ellipsoids ($\Ee:=\{\frac{\x_1^2}{a_1^2}+\cdots+\frac{\x_n^2}{a_n^2}\leq1\}$ where $a_i>0$) are $n$-bricks.
\item[(ii)] $n$-dimensional simplices (Lemma \ref{tetra}), $n$-dimensional hypercubes (Corollary \ref{bq}) and more generally $n$-dimensional products of closed unit balls (Corollary \ref{pb}) are $n$-bricks.
\item[(iii)] The finite product of $m_i$-bricks is an $m$-brick where $m:=\sum_im_i$ (Corollary \ref{pbr}).
\item[(iv)] If $f:\R^m\to\R^n$ is a polynomial map such that its image is not contained in a hyperplane, the convex hull of $f(\ol{\Bb}_m)$ is an $(m(n+1)+n)$-brick (Corollary \ref{convexhull}). In particular, if $m=1$, the convex hull of $f([-1,1])$ is a $(2n+1)$-brick. 
\item[(v)] If $k_1,\ldots,k_n\geq1$, the image of $\ol{\Bb}_n$ under $f:\R^n\to\R^n,\ (x_1,\ldots,x_n)\mapsto(x_1^{k_1},\ldots,x_n^{k_n})$ is an $n$-brick (Corollary \ref{ss}). We call this $n$-bricks \em spherical stars\em.
\item[(vi)] Truncated $n$-dimensional cones are $n$-bricks (Corollary \ref{trun}). 
\item[(vii)] $n$-dimensional parabolic segments are $n$-bricks (Lemma \ref{parseg1}), whereas more general $2$-dimensional parabolic segments (Lemma \ref{parab}) are $2$-bricks.
\item[(viii)] $n$-dimensional elliptic sectors and segments are $n$-bricks (Theorem \ref{circdef}). 
\item[(ix)] $n$-dimensional hyperbolic sectors and segments are $n$-bricks (Theorem \ref{hiperdef}). 
\end{itemize}
\end{examples} 

\begin{center}
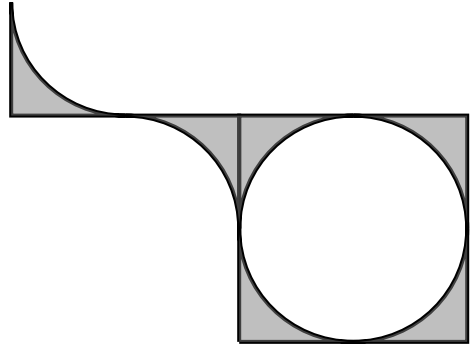
\begin{figure}[ht]
\begin{tikzpicture}[scale=0.75]

\draw[line width=1.5pt] (0,0) -- (2,0) arc (270:180:2cm) -- (0,0);
\draw[line width=1.5pt] (2,0) -- (4,0) -- (4,2) arc (360:270:2cm);
\draw[line width=1.5pt] (4,2) -- (4,4) -- (2,4) arc (90:0:2cm);
\draw[line width=1.5pt] (2,4) -- (0,4) -- (0,2) arc (180:90:2cm);
\draw[line width=1.5pt] (0,2) -- (0,4) -- (-2,4) arc (90:0:2cm);
\draw[line width=1.5pt] (-2,4) -- (-4,4) -- (-4,6) arc (180:270:2cm);

\draw[fill=gray!100,opacity=0.5,draw=none] (0,0) -- (2,0) arc (270:180:2cm) -- (0,0);
\draw[fill=gray!100,opacity=0.5,draw=none] (2,0) -- (4,0) -- (4,2) arc (360:270:2cm);
\draw[fill=gray!100,opacity=0.5,draw=none] (4,2) -- (4,4) -- (2,4) arc (90:0:2cm);
\draw[fill=gray!100,opacity=0.5,draw=none] (2,4) -- (0,4) -- (0,2) arc (180:90:2cm);
\draw[fill=gray!100,opacity=0.5,draw=none] (0,2) -- (0,4) -- (-2,4) arc (90:0:2cm);
\draw[fill=gray!100,opacity=0.5,draw=none] (-2,4) -- (-4,4) -- (-4,6) arc (180:270:2cm);
\end{tikzpicture}
\caption{Example of application of Theorem \ref{main2-i}\label{fig3}}
\end{figure}
\end{center}

By Theorem \ref{main2-i} it is possible to construct many polynomial images of the closed unit ball combining the previous $m$-bricks (Figures \ref{fig2} and \ref{fig3}). 

\subsection{State of the art}

The problem of characterizing the polynomial images of closed unit balls is related to the following one concerning polynomial images of affine spaces:

\begin{prob}\label{probl}
Characterize which (semialgebraic) subsets $\Ss\subset\R^n$ are polynomial or regular images of $\R^m$. 
\end{prob}

The first proposal for studying this problem and related ones, like the famous `quadrant problem' \cite{fg1}, goes back to \cite{g} (see also \cite[\S3.IV, p.69]{ei}). Other types of maps (like Nash, continuous rational, etc.) have been already considered to represent semialgebraic sets as images of affine spaces \cite{f2,ffqu}. A complete solution to Problem \ref{probl} seems far, but we have developed significant progresses:

\vspace*{1mm}
\noindent{\em General properties.} We have found conditions \cite{f1,fg2,fu1,u1} that a semialgebraic subset must satisfy to be either a polynomial or regular image of $\R^m$. The most remarkable one states that the set of points at infinity of a polynomial image of $\R^m$ is connected \cite{fu1}. The $1$-dimensional case was described in \cite{f1}. In \cite{ffqu} we proved the equality between the family of regular images of $\R^2$ and the family of continuous rational images of $\R^2$. 

\vspace*{1mm}
\noindent{\em Nash images.} 
In \cite{f2} it is provided a full characterization of the semialgebraic subsets $\Ss\subset\R^m$ that are Nash images of $\R^n$: \em A $d$-dimensional semialgebraic set $\Ss\subset\R^n$ is the image of $\R^d$ under a Nash map $f:\R^d\to\R^n$ if and only if $\Ss$ is connected by analytic paths\em. Parallel to the elaboration of the current work, Carbone--Fernando have obtained a full characterization of the Nash images of the closed unit ball \cite{cf}: \em A compact $d$-dimensional semialgebraic set $\Ss\subset\R^n$ is the image of the closed unit ball $\ol{\Bb}_d$ under a Nash map $f:\R^d\to\R^n$ if and only if $\Ss$ is connected by analytic paths\em. As a straightforward consequence, we show there: {\em a $d$-dimensional compact semialgebraic set $\Ss\subset\R^n$ connected by analytic paths is the image of any $d$-dimensional semialgebraic set $\Tt\subset\R^m$ under a suitable Nash map $f:\R^m\to\R^n$}, whereas {\em a general $d$-dimensional semialgebraic set $\Ss\subset\R^n$ connected by analytic paths is the image of any non-compact $d$-dimensional semialgebraic set $\Tt\subset\R^m$ under a suitable Nash map $f:\R^m\to\R^n$}.

\vspace*{1mm}
\noindent{\em Representation of semialgebraic sets as polynomial or regular images of $\R^n$.} We have proposed constructions to represent as either polynomial or regular images of $\R^n$ semialgebraic sets that can be described by linear equalities and inequalities. In \cite{f1,fg1,fgu1,fgu2,fgu3,fgu4,fu2,fu3,fu4,fu5,u2} we have analyzed the cases of convex polyhedra and their interiors, together with their respective complements and we have provided a full answer \cite[Table 1]{fu5}.

\vspace*{1mm}
\noindent{\em Optimization of the `open quadrant problem'.} 
It is difficult to determine which is the minimum degree for a polynomial map that has a given semialgebraic set as image. We have tried to find the least degree of the components of a polynomial map $f:\R^2\to\R^2$ whose image is the open quadrant $\Qq:=\{\x>0,\y>0\}$. We have shown that it is bounded above by 16 (see \cite{fg1,fgu2,fu4}) and we know that it can be lowered down until $8$, which seems to be the sharpest bound. 

\subsection{Alternative models}\label{am}

The image of a polynomial map $f:\R^m\to\R^n$ is either a point or an unbounded semialgebraic set. Thus, it is not posible to compare the polynomial images of closed unit balls with the polynomial images of affine spaces. However, we can compare polynomial images of closed unit balls with regular images of affine spaces and we will see that the first family is a subfamily of the second (Corollary \ref{images}). Of course, such inclusion is the maximum one can aspire because regular images of affine spaces contain polynomial images of affine spaces.

We have chosen the closed unit ball to represent semialgebraic sets as polynomial images because of its good properties. The closed unit ball is a compact manifold with boundary and its boundary is a homogeneous manifold. There exist other spaces like an $n$-sphere (which is a compact manifold without boundary) to represent as polynomial images semialgebraic sets. We claim: \em There exists no polynomial map from an $m$-dimensional closed unit ball onto the $n$-sphere $\sph^n:=\{\x_1^2+\cdots+\x_{n+1}^2=1\}$\em. 

If $f:=(f_1,\ldots,f_{n+1}):\R^m\to\R^{n+1}$ is a polynomial map such that $f(\ol{\Bb}_m)=\sph^n$, then $f_1^2+\cdots+f_{n+1}^2=1$ on the open ball $\Bb_m$. Using the Taylor's expansion of $f_1^2+\cdots+f_{n+1}^2$ at the origin, we deduce that $f_1^2+\cdots+f_{n+1}^2=1$ on $\R^m$, so $\deg(f_1),\ldots,\deg(f_{n+1})\leq0$, that is, $f_1,\ldots,f_{n+1}$ are constant polynomials, which is a contradiction. 

An $n$-sphere projects onto an $n$-dimensional closed unit ball. This means that family of polynomial images of the $n$-dimensional closed unit ball is a subfamily of the family of polynomial images of the $n$-sphere. However, if we consider regular images instead of polynomial images, both families coincide (Corollary \ref{images2}).

\subsection{Related problems and open questions}

The effective representation of a semialgebraic subset $\Ss\subset\R^n$ as a polynomial or a regular image of a closed unit ball $\ol{\Bb}_m$ may help the handling of classical problems in Real Geometry by reducing them to its study in $\ol{\Bb}_m$. 

\noindent{\em Positivstellens\"atze}. A widespread studied problem is the algebraic characterization of those polynomial or regular functions $g:\R^n\to\R$ that are either strictly positive or positive semidefinite on a semialgebraic set $\Ss\subset\R^n$. When $\Ss$ is a basic closed semialgebraic set, these problems were solved in \cite{s} (see also \cite[Cor.4.4.3]{bcr}). For convex (compact) polyhedra we refer the reader to \cite{h}, where stronger Positivstellens\"atze are obtained, specially for strictly positive polynomials. The obtained certificate of positiveness is the best possible one. 

Let $f:\R^m\to\R^n$ be a polynomial map and denote $\Ss:=f(\ol{\Bb}_m)$. Note that $g$ is strictly positive (respectively positive semidefinite) on $\Ss$ if and only if $g\circ f$ is strictly positive (respectively positive semidefinite) on $\ol{\Bb}_m$ and both questions are decidable by \cite{s}. This provides an algebraic characterization of positiveness for polynomial and regular functions on semialgebraic sets that are polynomial images of $\ol{\Bb}_m$. As a consequence of Theorem \ref{main1-i}, we deduce a Positivstellensatz certificate for finite unions (connected by analytic paths) of $n$-dimensional convex (compact) polyhedra. By Theorem \ref{main2-i} the same happens with finite unions (connected by analytic paths) of $n$-dimensional convex semialgebraic sets $\Ss$ that are polynomial images of the closed unit ball. 

\noindent{\em Optimization}. Suppose that $f:\R^m\to\R^n$ is either a polynomial or a regular map and let $\Ss:=f(\ol{\Bb}_m)$. Then the optimization of a given regular or polynomial function $g:\Ss\to\R$ is equivalent to the optimization of the composition $g\circ f$ on $\ol{\Bb}_m$. In this way one can strongly simplify contour conditions, but one has to evaluate whether the increase of the complexity (concentrated mainly in one variable, see the proofs of Theorems \ref{main1-i}, \ref{main1-i2} and \ref{main2-i}) of the composition $g\circ f$ is preferable to the existence of contour conditions.

Alternatively, let $\Tt\subset\R^n$ be a compact semialgebraic set and let $h:\Tt\to\R$ be a continuous semialgebraic function. Compact semialgebraic sets are triangulable and by \cite[Thm.9.4.1]{bcr} also continuous semialgebraic functions on compact semialgebraic sets can be `triangulated'. Thus, a continuous semialgebraic function on a compact semialgebraic set could be assumed (up to a suitable triangulation) as a continuous function on a finite simplicial complex that is affine on each simplex of the complex. Optimization problems for this type of functions are `straightforwardly' approached. 

The usual algorithms to triangulate a compact semialgebraic set $\Ss\subset\R^n$ (and continuous semialgebraic maps) \cite[Ch.9]{bcr} are based on the use of cylindrical decompositions, which have doubly exponential complexity in the number $n$ of variables involved in describing $\Ss$. More precisely, its complexity is in general $(\ell d)^{O(1)^n}$ where $O(1)$ represents a constant, $\ell$ is a bound on the number of polynomials need to describe $\Ss$ and $d$ is a bound on the degrees of a family of polynomials describing $\Ss$, see \cite[Ch.11]{bpr}. If $\Ss$ has piecewise linear boundary, the complexity of cylindrical decomposition is $\ell^{O(1)^n}$, which is still doubly exponential in the number $n$ of variables involved in describing $\Ss$. 

It would be interesting to compare both complexities and decide which procedure is more effective.

\noindent{\it Questions regarding complexity.} 
The algorithms developed in Theorems \ref{main1-i}, \ref{main1-i2} and \ref{main2-i} are constructive. Natural questions arise when considering the issue of complexity and we refer the reader to \cite[Rem.4.8]{f3} for some estimations concerning the piecewise linear case: 

\begin{quest2}\label{q1}
Fix $m$ to be either $n$ or $n+1$. {\em Which is the minimum degree of a polynomial map $f:\R^m\to\R^n$ such that $f(\ol{\Bb}_m)$ is a prescribed finite union $\Ss$ (connected by analytic paths) of $n$-dimensional convex polyhedra?}
\end{quest2}

\begin{quest2}\label{q3}
Fix $m\geq1$ and let $\Ss_i$ be an $m$-brick for $i=1,\ldots,r$ such that $\Ss=\bigcup_{i=1}^r\Ss_i$ is connected by analytic paths. {\em Which is the minimum degree of a polynomial map $f:\R^{m+1}\to\R^n$ such that $f(\ol{\Bb}_{m+1})=\Ss$?}
\end{quest2}

\subsection*{Structure of the article}

The article is organized as follows. Main basic examples of polynomial maps of the closed unit ball are collected in Section \ref{s2}. Some of these examples will be useful for our subsequent constructions. In Section \ref{s3} we present a `Smart curve selection lemma', which is the key for the main results of this work (see also \cite{f3}). In Section \ref{s4} we prove Theorems \ref{main1-i} and \ref{main1-i2}, whereas in Section \ref{s5} we prove Theorem \ref{main2-i}. The article ends with three appendices. In the first one we approach the questions concerning alternative models for regular images proposed in \S\ref{am}. In the second one we provide proofs to some of the results in Section \ref{s2}. We have postponed them in order to make the exposition smoother. In the third appendix we present an explicit example concerning Theorem \ref{main1-i}.

\subsection*{Acknowledgements}

The authors are deeply indebted with Prof. Safey el Din for very helpful and enlightening comments and inspiring talks during the preparation of this work. The authors would also like to point out that some of the figures included in this article have been made using Geogebra.

\section{Relevant families of bricks}\label{s2}

In this section we approach the key examples of $m$-bricks already presented in Examples \ref{bricks} and some additional ones. Before that we prove the following preliminary result, which shows that convex polynomial images of an $m$-dimensional closed unit ball are $m$-bricks. A set $\Ss\subset\R^n$ is \em strictly radially convex (with respect to a point $p\in\Int(\Ss)$) \em if for each ray $\ell$ with origin at $p$, the intersection $\ell\cap\Ss$ is a segment whose relative interior is contained in $\Int(\Ss)$. Convex sets are particular examples of strictly radially convex sets (with respect to any of its interior points \cite[Lem.11.2.4]{ber1}).

\begin{lem}\label{propbrick}
Let $\Ss\subset\R^n$ be a strictly radially convex set with respect to a point $p\in\Int(\Ss)$ that is in addition a polynomial image of the closed unit ball $\ol{\Bb}_m$. Then $\Ss$ is an $m$-brick.
\end{lem}
\begin{proof}
Let $F:\R^m\to\R^n$ be a polynomial map such that $F(\ol{\Bb}_m)=\Ss$. The continuous semialgebraic map $H:[0,1]\times\ol{\Bb}_m\to\Ss,\ (t,x)\mapsto tp+(1-t)F(x)$ satisfies the conditions in Definition \ref{mbrickd}, so $\Ss$ is an $m$-brick, as required.
\end{proof}

As an immediate application, one shows that $n$-ellipsoids are $n$-bricks. In general a convex semialgebraic set $\Ss\subset\R^n$ is an $m$-brick if and only if it is the image of $\ol{\Bb}_m$ under a polynomial map $f:\R^m\to\R^n$.

If $\Ss\subset\R^n$, denote by $\partial\Ss:=\cl(\Ss)\setminus\Int(\Ss)$ the {\em boundary of $\Ss$}. The reader can produce other examples of $m$-bricks as an application of the following result.

\begin{lem}\label{propbrick2}
Let $F:\R^m\to\R^n$ be a polynomial map and let $\Ss:=F(\ol{\Bb}_m)$. Suppose $F^{-1}(\partial\Ss)\cap\ol{\Bb}_m\subset\partial\ol{\Bb}_m$. Then $\Ss$ is an $m$-brick. 
\end{lem}
\begin{proof}
The continuous semialgebraic map $H:[0,1]\times\ol{\Bb}_m\to\Ss,\ (t,x)\mapsto F((1-t)x)$ satisfies the conditions in Definition \ref{mbrickd}, so $\Ss$ is an $m$-brick, as required.
\end{proof}

The following result lighten the conditions to be an $m$-brick and show that what we need is to be able to deform polynomially $\Ss$ (as a polynomial image of $\ol{\Bb}_m$) to put it inside $\Int(\Ss)$, see also Lemma \ref{basico}.

\begin{lem}[Characterization of $m$-bricks]
Let $\Ss\subset\R^n$ be a semialgebraic set. The following assertions are equivalent:
\begin{itemize}
\item[(i)] $\Ss$ is an $m$-brick.
\item[(ii)] There exists a map $H:=(H_1,\ldots,H_n):[0,1]\times\ol{\Bb}_m\to\R^n$ and $\veps>0$ such that:
\begin{itemize}
\item[$\bullet$] $H_i\in{\mathcal C}^0([0,1])[\x_1,\ldots,\x_m]$ for $i=1,\ldots,n$.
\item[$\bullet$] $H(\{0\}\times\ol{\Bb}_m)=\Ss$ and $H(\{t\}\times\ol{\Bb}_m)\subset\Int(\Ss)$ for each $t\in(0,\veps)$.
\end{itemize}
\end{itemize}
\end{lem}
\begin{proof}
The implication (i) $\Longrightarrow$ (ii) is clear. Let us prove the converse implication. We modify $H$ to have a homotopy $H:=(H_1,\ldots,H_n):[0,1]\times\ol{\Bb}_m\to\R^n$ that satisfies the conditions in Definition \ref{mbrickd}. Consider the homotopy
$$
H':[0,1]\times\ol{\Bb}_m\to\Ss,\ (t,x)\mapsto\begin{cases}
H(t,x)&\text{if $t\in[0,\frac{\veps}{2}]$,}\\
H(\frac{\veps}{2},(\frac{2}{2-\veps})(1-t)x)&\text{if $t\in[\frac{\veps}{2},1]$.}
\end{cases}
$$
The reader can check that $H'$ satisfies all the conditions in Definition \ref{mbrickd}.
\end{proof}

Using the fact that $m$-bricks are compact, one deduces the following.

\begin{cor}\label{homeo}
Let $\Ss\subset\R^n$ be an $m$-brick and let $F:\R^n\to\R^n$ be a polynomial map such that $F|_{\Ss}$ is injective. Then $F(\Ss)$ is an $m$-brick.
\end{cor}

\subsection{Simplices}

We begin with the case of an {\em $n$-simplex}, that is, a simplex of dimension $n$. 

\begin{lem}[$n$-simplex]\label{tetra}
Let $\Delta_n:=\{0\leq\y_1,\ldots,0\leq\y_n,\y_1+\cdots+\y_n\leq1\}$ be the $n$-simplex of vertices the origin and the points ${\tt e}_i:=(0,\ldots,0,\overset{(i)}{1},0,\ldots,0)$. Then $\Delta_n$ is an $n$-brick.
\end{lem}
\begin{proof}
The images of the closed unit ball $\ol{\Bb}_n$ under $f:\R^n\to\R^n,\ (x_1,\ldots,x_n)\mapsto(x_1^2,\ldots,x_n^2)$ is $\Delta_n$, as required.
\end{proof}

We prove next that each $n$-dimensional convex (compact) polyhedron $\pol\subset\R^n$ is an $m$-brick where $m+1$ is the number of vertices on $\pol$. In Theorem \ref{main1-i2} we will prove that $\pol$ is an $n$-brick. 

\begin{cor}
Let $\pol\subset\R^n$ be the $n$-dimensional convex (compact) polyhedron of vertices $v_0,\ldots,v_m$. Then $\pol$ is an $m$-brick.
\end{cor}
\begin{proof}
As $v_0,\ldots,v_m$ are the vertices of an $n$-dimensional convex (compact) polyhedron, $n\leq m$ and we may assume $v_0,\ldots,v_n$ are affinely independent. Let 
$$
\pi:\R^m\to\R^n,\ (x_1,\ldots,x_m)\to(x_1,\ldots,x_n)
$$ 
be the projection onto the first $n$ coordinates. Denote ${\tt e}_j:=(0,\ldots,0,\overset{(j)}{1},0,\ldots,0)\in\R^{m-n}$ for $j=1,\ldots,m-n$. Write 
$$
u_i:=\begin{cases}
(v_i,0)&\text{if $i=0,\ldots,n$},\\
(v_i,{\tt e}_{i-n})&\text{if $i=n+1,\ldots,m$.}
\end{cases}
$$
The points $u_0,\ldots,u_m\in\R^m$ are affinely independent and $\pi(u_i)=v_i$ for $i=0,\ldots,m$. If $\Delta:=\{\sum_{k=0}^m\lambda_iu_i:\,\lambda_i\geq0,\ \sum_{k=0}^m\lambda_i=1\}$ is the $m$-simplex of vertices $u_0,\ldots,u_m$, then $\pi(\Tt)=\pol$. By Lemma \ref{tetra} we conclude $\pol$ is an $m$-brick, as required.
\end{proof}

\subsection{Products of closed balls.}

All the semialgebraic sets $\Ss\subset\R^n$ treated in this subsection are convex and $n$-dimensional, so to prove that they are $m$-bricks, it is enough to prove that they are polynomial images of $\ol{\Bb}_m$. We prove first that $n$-dimensional cylinders and simplicial prisms are $n$-bricks (Figure \ref{rel}). This latter case is particularly interesting for the proofs of Theorems \ref{main1-i}, \ref{main1-i2} and \ref{main2-i}.

\begin{center}
\begin{figure}[ht]
\begin{tikzpicture}[scale=0.5]

\draw[line width=1.5pt] (-1,0) circle (2cm);
\draw[line width=1.5pt] (-3,0) arc (180:360:2cm and 0.5cm);
\draw[dashed,line width=1.5pt] (1,0) arc (0:180:2cm and 0.5cm);

\draw[line width=1.5pt] (6,1.75) -- (6,-1.75) arc (180:360:2cm and 0.5cm) -- (10,1.75) arc (0:180:2cm and 0.5cm);
\draw[line width=1.5pt] (6,1.75) arc (180:360:2cm and 0.5cm);

\draw[fill=gray!100,opacity=0.5,draw=none] (-1,0) circle (2cm);
\draw[fill=gray!100,opacity=0.5,draw=none] (6,1.75) -- (6,-1.75) arc (180:360:2cm and 0.5cm) -- (10,1.75) arc (0:180:2cm and 0.5cm);

\draw[dashed,line width=1.5pt] (10,-1.75) arc (0:180:2cm and 0.5cm);

\draw[->, line width=1pt] (2,0) -- (5,0);
\draw[->, line width=1pt] (11,0) -- (14,3.5);
\draw[->, line width=1pt] (11,0) -- (14,-3.5);

\draw[line width=1.5pt] (15,5.5) -- (15,1.5) -- (19,1) -- (19,5) -- (17.5,6.5) -- (15,5.5);
\draw[line width=1.5pt] (15,5.5) -- (19,5);
\draw[dashed,line width=1.5pt] (15,1.5) -- (17.5,2.5) -- (19,1);
\draw[dashed,line width=1.5pt] (17.5,2.5) -- (17.5,6.5);
\draw[fill=gray!100,opacity=0.5,draw=none] (15,5.5) -- (15,1.5) -- (19,1) -- (19,5) -- (17.5,6.5) -- (15,5.5);

\draw[line width=1.5pt] (15,-5.5) -- (15,-1.5) -- (16,-0.5) -- (20,-0.5) -- (20,-4.5) -- (19,-5.5) -- (15,-5.5);
\draw[fill=gray!100,opacity=0.5,draw=none] (15,-5.5) -- (15,-1.5) -- (16,-0.5) -- (20,-0.5) -- (20,-4.5) -- (19,-5.5) -- (15,-5.5);

\draw[line width=1.5pt] (15,-1.5) -- (19,-1.5) -- (20,-0.5);
\draw[line width=1.5pt] (19,-1.5) -- (19,-5.5);
\draw[dashed,line width=1.5pt] (15,-5.5) -- (16,-4.5) -- (16,-0.5);
\draw[dashed,line width=1.5pt] (16,-4.5) -- (20,-4.5);

\draw (-1,-3) node{Closed ball};
\draw (8,-3) node{Cylinder};
\draw (17,-6.25) node{Hypercube};
\draw (17,0.5) node{Prism};

\draw (3.4,0.6) node{\small Lemma \ref{imagebc}};
\draw (12,2) node[rotate=50]{\small Corollary \ref{tetra1}};
\draw (12.1,-2) node[rotate=-50]{\small Corollary \ref{bq}};

\end{tikzpicture}
\caption{Examples of $n$-bricks and some relations between them\label{rel}}
\end{figure}
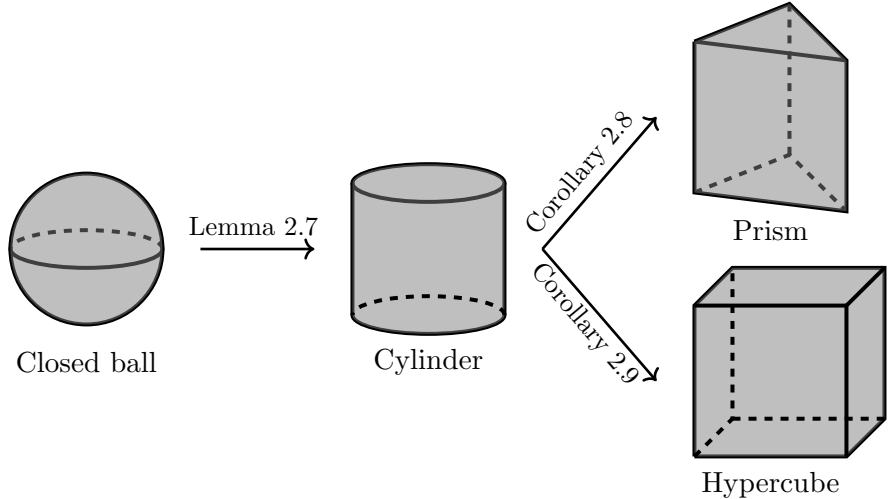
\end{center}

\begin{lem}[Cylinder]\label{imagebc}
Let $\Cc_n:=\ol{\Bb}_{n-1}\times[-1,1]\subset\R^n$. Then $\Cc_n$ is an $n$-brick. 
\end{lem}
\begin{proof}
Consider the polynomial functions $g(\t):=\t(3-4\t^2)$ and $h(\t):=\sqrt{3}(1-\frac{4}{9}\t^2)$. Observe that $g(\pm1)=\mp1$, $g(\pm\frac{1}{2})=\pm1$, $g|_{[-1,1]}$ has a global maximum at $t=\frac{1}{2}$, a global minimum at $t=-\frac{1}{2}$ and it is strictly increasing on $[-\tfrac{1}{2},\tfrac{1}{2}]$ (Figure \ref{fig21}). In addition, $h_1:=\t h$ satisfies $h_1(0)=0$, $h_1(\frac{\sqrt{3}}{2})=1$ and $h_1|_{[0,1]}$ has a global maximum at $t=\frac{\sqrt{3}}{2}$ (Figure \ref{fig21}). The polynomial function $g^*:=g(\sqrt{1-\t^2})^2=(1-t^2)(4t^2-1)^2$ satisfies $g^*([0,1])=[0,1]$ (Figure \ref{fig21}).

Denote $x:=(x_1,\ldots,x_n)$ and $x':=(x_1,\ldots,x_{n-1})$. Define the polynomial maps
\begin{align*}
&G:\R^n\to\R^n,\ x:=(x',x_n)\mapsto(x',g(x_n)),\\
&H:\R^n\to\R^n,\ x:=(x',x_n)\mapsto(x'h(\|x'\|),x_n).
\end{align*}
We claim: 
$$
\Ss_n:=G(\ol{\Bb}_n)=\{\|\x'\|^2\leq\tfrac{3}{4},\ |\x_n|\leq1\}\cup\{\tfrac{3}{4}\leq\|\x'\|^2\leq1,\x_n^2\leq g^*(\|\x'\|)\}\subset\Cc_n.
$$

Write $\ol{\Bb}_n=\bigcup_{x'\in\ol{\Bb}_{n-1}}\Ii_{x'}$ where $\Ii_{x'}:=\{x'\}\times\{\x_n^2\leq1-\|\x'\|^2\}$ and observe that 
$$
G(\ol{\Bb}_n)=\bigcup_{x'\in\ol{\Bb}_{n-1}}G(\Ii_{x'})=\bigcup_{x'\in\ol{\Bb}_{n-1}}\{x'\}\times g(\{\x_n^2\leq1-\|\x'\|^2\}).
$$
We distinguish two cases:

\noindent{\sc Case 1.} If $\|x'\|^2\leq\frac{3}{4}$, then $\{x'\}\times[-\frac{1}{2},\frac{1}{2}]\subset\Ii_{x'}\subset\{x'\}\times[-1,1]$, so $g(\Ii_{x'})=\{x'\}\times[-1,1]$. 

\noindent{\sc Case 2.} If $\frac{3}{4}\leq\|x'\|^2\leq1$, then $\Ii_{x'}\subset\{x'\}\times[-\frac{1}{2},\frac{1}{2}]$. As $g$ is odd and strictly increasing on $[-\tfrac{1}{2},\tfrac{1}{2}]$, we have 
$$
G(\Ii_{x'})=\{x'\}\times\{\x_n^2\leq g(\sqrt{1-\|x'\|^2})^2=g^*(\|x'\|)\}
$$
and the claim follows.

\begin{center}
\begin{figure}[ht]
\begin{minipage}{0.48\textwidth}
\begin{center}
{\scriptsize\begin{tikzpicture}[scale=0.9]
\begin{axis}[unit vector ratio=1.25 1.25 1.25,axis x line=middle,axis y line=middle]	
\addplot[domain=-1:1,blue,line width=1.5pt,smooth] {x*(3-4*x^2)};
\draw[dotted] (axis cs:0.5,0) -- (axis cs:0.5,1);
\draw[dotted] (axis cs:-0.5,0) -- (axis cs:-0.5,-1);
\draw[dotted] (axis cs:-1,1) -- (axis cs:1,1);
\draw[dotted] (axis cs:-1,-1) -- (axis cs:1,-1);
\addplot[holdot] coordinates{(0.5,0)(-0.5,0)};
\addplot[soldot] coordinates{(0,0)(0.5,1)(-0.5,-1)};
\end{axis}
\end{tikzpicture}}
\end{center}
\end{minipage}
\hfil
\begin{minipage}{0.48 \textwidth}
\begin{center}
{\scriptsize\begin{tikzpicture}[scale=1]
\begin{axis}[unit vector ratio=1 1 1,axis x line=middle,axis y line=middle]
\addplot[domain=0:1.5,red,line width=1.5pt,smooth] {x*(3^(0.5)-4*3^(0.5)/9*x^2)};
\draw[dotted] (axis cs:0.866025403784439,0) -- (axis cs:0.866025403784439,1);
\draw[dotted] (axis cs:1,0) -- (axis cs:1,0.9622504485);
\draw[dotted] (axis cs:0,1) -- (axis cs:1.5,1);
\addplot[holdot] coordinates{(0.866025403784439,0)(1,0)};
\addplot[soldot] coordinates{(0,0)(0.866025403784439,1)(1,0.9622504485)(1.5,0)};
\end{axis}
\end{tikzpicture}}
\end{center}
\end{minipage}
\begin{center}
{\scriptsize\begin{tikzpicture}[scale=0.9]
\begin{axis}[unit vector ratio=1.25 1.25 1.25, axis x line=middle,axis y line=middle]	
\addplot[domain=0:1.003,black,line width=1.5pt,smooth] {(1-x^2)*(4*x^2-1)^2};
\draw[dotted] (axis cs:0,1) -- (axis cs:1,1);
\draw[dotted] (axis cs:0.866025403784439,0) -- (axis cs:0.866025403784439,1);
\addplot[holdot] coordinates{(0.5,0)(0.866025403784439,0)(1,0)};
\addplot[soldot] coordinates{(0,1)(0.866025403784439,1)};
\end{axis}
\end{tikzpicture}}
\end{center}

\caption{Graphs of $g$, $h_1$ and $g^*$\label{fig21}}
\end{figure}
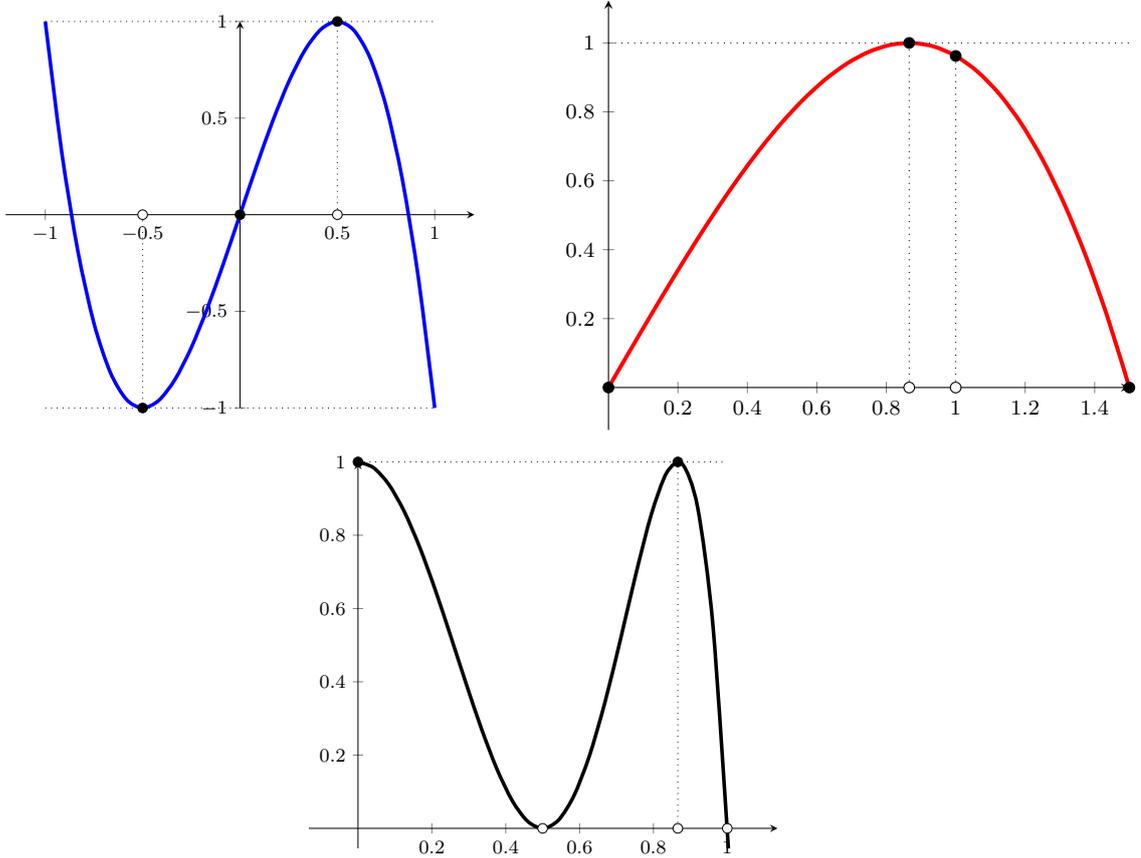
\end{center}

Let us check: {\em $H(\Cc_n)\subset\Cc_n$ and $H(\{\|\x'\|^2\leq\tfrac{3}{4},\ |\x_n|\leq1\})=\Cc_n$}. 

If $(x',x_n)\in\Cc_n$, then $\|x'\|\leq1$, so we deduce $\|x'h(\|x'\|)\|=|h_1(\|\x'\|)|\leq1$ and $H(x',x_n)=(x'h(\|x'\|),x_n)\in\Cc_n$. If $x'\in\partial\ol{\Bb}_{n-1}$, then 
\begin{multline*}
H(\{\lambda x':\ \lambda\in[0,\tfrac{\sqrt{3}}{2}]\}\times[-1,1])\\
=\{h_1(\lambda)x':\ \lambda\in[0,\tfrac{\sqrt{3}}{2}]\}\times[-1,1]=\{\mu x':\ \mu\in[0,1]\}\times[-1,1],
\end{multline*}
so $H(\{\|\x'\|^2\leq\tfrac{3}{4},\ |\x_n|\leq1\})=\Cc_n$.

Consequently, $H(G(\ol{\Bb}_n))=\Cc_n$ and $f:=H\circ G:\R^n\to\R^n$ is a polynomial map such that $f(\ol{\Bb}_n)=\Cc_n$ (Figure \ref{fig22}), as required.
\end{proof}

\begin{center}
\begin{figure}[ht]
\begin{tikzpicture}[scale=0.75]
\draw[line width=1.5pt] (0,0) circle (2cm);
\draw[line width=1.5pt] (-2,0) arc (180:360:2cm and 0.5cm);
\draw[dashed,line width=1.5pt] (2,0) arc (0:180:2cm and 0.5cm);

\draw[line width=1.5pt] (6.1,-2.125) .. controls (5,-1.5) and (5,1.5) .. (6.1,2.125); 
\draw[line width=1.5pt] (6,2) arc (180:0:1.5cm and 0.35cm);
\draw[line width=1.5pt] (8.9,2.125) .. controls (10,1.5) and (10,-1.5) .. (8.9,-2.125); 
\draw[line width=1.5pt] (9,-2) arc (360:180:1.5cm and 0.35cm);
\draw[dashed,line width=1.5pt] (6,-2) -- (6,2);
\draw[dashed,line width=1.5pt] (9,-2) -- (9,2);
\draw[line width=1.5pt] (13,2) -- (13,-2) arc (180:360:2cm and 0.5cm) -- (17,2) arc (0:180:2cm and 0.5cm);
\draw[line width=1.5pt] (13,2) arc (180:360:2cm and 0.5cm);
\draw[dashed,line width=1.5pt] (17,-2) arc (0:180:2cm and 0.5cm);

\draw[line width=1.5pt] (6,2) arc (180:360:1.5cm and 0.35cm);
\draw[line width=1.5pt] (5.275,0) arc (180:360:2.225cm and 0.35cm);
\draw[dashed,line width=1.5pt] (9.725,0) arc (0:180:2.225cm and 0.35cm);
\draw[dashed,line width=1.5pt] (9,-2) arc (0:180:1.5cm and 0.35cm);

\draw[fill=gray!100,opacity=0.5,draw=none] (0,0) circle (2cm);
\draw[fill=gray!100,opacity=0.5,draw=none] (6,-2) .. controls (5,-1.5) and (5,1.5) .. (6,2) arc (180:0:1.5cm and 0.35cm) .. controls (10,1.5) and (10,-1.5) .. (9,-2) arc (360:180:1.5cm and 0.35cm);
\draw[fill=gray!100,opacity=0.5,draw=none] (13,2) -- (13,-2) arc (180:360:2cm and 0.5cm) -- (17,2) arc (0:180:2cm and 0.5cm);

\draw[->, line width=1pt] (2.75,0) -- (4.5,0);
\draw[->, line width=1pt] (10.5,0) -- (12.25,0);

\draw (3.625,0.5) node{$G$};
\draw (11.375,0.5) node{$H$};
\draw (0,0) node{$\ol{\Bb}_n$};
\draw (7.5,0) node{$\Ss_n$};
\draw (15,0) node{\small$\ol{\Bb}_{n-1}\times[-1,1]$};

\end{tikzpicture}
\caption{The cylinder $\ol{\Bb}_{n-1}\times[-1,1]$ as a polynomial image of $\ol{\Bb}_n$\label{fig22}}
\end{figure}
\end{center}

\begin{cor}[Simplicial prism]\label{tetra1}
Let $\Delta_n\subset\R^n$ be an $n$-simplex of vertices $v_0,\ldots,v_n$. Then the product $\Delta_n\times[-1,1]$ is an $(n+1)$-brick. 
\end{cor}
\begin{proof}
After an affine change of coordinates we assume $\Delta_n$ is the $n$-simplex of vertices $v_0={\tt0}$ and $v_i={\tt e}_i$ for $i=1,\ldots,n$. By Lemma \ref{imagebc} there exists a polynomial map $f_0:\R^{n+1}\to\R^{n+1}$ such that $f_0(\ol{\Bb}_{n+1})=\ol{\Bb}_n\times[-1,1]$. By Lemma \ref{tetra} there exists a polynomial map $f_1:\R^n\to\R^n$ such that $f_1(\ol{\Bb}_n)=\Delta_n$. Thus, the polynomial map $f:=(f_1,\x_{n+1})\circ f_0:\R^{n+1}\to\R^{n+1}$ satisfies $f(\ol{\Bb}_{n+1})=(f_1,\x_{n+1})(f_0(\ol{\Bb}_{n+1}))=(f_1,\x_{n+1})(\ol{\Bb}_n\times[-1,1])=\Delta_n\times[-1,1]$, as required.
\end{proof}

Let us check next that $n$-dimensional hypercubes and more generally $n$-dimensional products of balls are $n$-bricks.

\begin{cor}[Hypercube]\label{bq}
The $n$-dimensional hypercube $\Qq_n:=[-1,1]^n\subset\R^n$ is an $n$-brick.
\end{cor}
\begin{proof}
Write $\Cc_n:=\ol{\Bb}_{n-1}\times[-1,1]$. By Lemma \ref{imagebc} there exists a polynomial map $f_0:\R^n\to\R^n$ such that $f_0(\ol{\Bb}_n)=\ol{\Bb}_{n-1}\times[-1,1]$. In addition, $\ol{\Bb}_1=[-1,1]=\Qq_1$, so the first step of the induction process holds. By induction hypothesis there exists a polynomial map $f_1:\R^{n-1}\to\R^{n-1}$ such that $f_1(\ol{\Bb}_{n-1})=\Qq_{n-1}$. Consider the polynomial map $f:=(f_1,\x_n)\circ f_0:\R^n\to\R^n$, which satisfies 
$$
f(\ol{\Bb}_n)=(f_1,\x_n)\circ f_0(\ol{\Bb}_n)=(f_1,\x_n)(\ol{\Bb}_{n-1}\times[-1,1])=f_1(\ol{\Bb}_{n-1})\times[-1,1]=\Qq_{n-1}\times[-1,1]=\Qq_n,
$$
as required.
\end{proof}

For each $p\in\R^n$ and each $\rho>0$ denote by $\ol{\Bb}_n(p,\rho)$ the closed ball of center $p$ and radius $\rho$ and $\Bb_n(p,\rho)$ the open ball of center $p$ and radius $\rho$.

\begin{lem}\label{inverse}
There exists a polynomial map $f:\R^n\to\R^n$ such that $f([-1,1]^n)=\ol{\Bb}_n$ and $f(\Bb_n(0,1+\veps))=\ol{\Bb}_n$ for each $\veps>0$ small enough.
\end{lem}
\begin{proof}
If $n=1$, consider the polynomial function $f:\R\to\R,\ t\mapsto\frac{1}{2}t(3-t^2)$. Assume $n\geq2$ and consider the univariate polynomial $h:=\t^2\frac{(\t-n)^{2(n-1)}}{(n-1)^{2(n-1)}}$. Observe that $h(0)=0$, $h(n)=0$, $h(1)=1$ and
$$
h'=\frac{1}{(n-1)^{2(n-1)}}(2\t(\t-n)^{2(n-1)}+2(n-1)\t^2(\t-n)^{2(n-1)-1})=\frac{2n\t(\t-n)^{2(n-1)-1}}{(n-1)^{2(n-1)}}(\t-1).
$$
Thus, $h'$ is positive on $(0,1)$ and it is negative on $(1,n)$. Consequently, $1$ is an absolute maximum of $h$ on the interval $[0,n]$, so $h(\t)\leq 1$ on $[0,n]$.

Recall that $\ol{\Bb}_n\subset[-1,1]^n\subset\ol{\Bb}_n(0,\sqrt{n})$ and consider the polynomial map
$$
f:\R^n\to\R^n,\ x\mapsto h(\|x\|^2)x.
$$
Observe that $f(\ol{\Bb}_n)=\ol{\Bb}_n$ and $f(\ol{\Bb}_n(0,\sqrt{n}))=\ol{\Bb}_n$. Thus, $f([-1,1]^n)=\ol{\Bb}_n$ and $f(\Bb_n(0,1+\veps))=\ol{\Bb}_n$ for each $0<\veps<\sqrt{n}-1$, as required.
\end{proof}

\begin{cor}[Products of closed balls]\label{pb}
Let $n_1,\ldots,n_\ell$ be positive integers and denote $n:=\sum_{i=1}^\ell n_i$. Then the product $\Ss:=\prod_{i=1}^\ell\ol{\Bb}_{n_i}\subset\R^n$ has dimension $n$ and it is an $n$-brick.
\end{cor}
\begin{proof}
By Lemma \ref{inverse} there exist polynomial maps $f_i:\R^{n_i}\to\R^{n_i}$ such that $f_i([-1,1]^{n_i})=\ol{\Bb}_{n_i}$. Consequently, the image of $[-1,1]^n=\prod_{i=1}^\ell[-1,1]^{n_i}$ under the polynomial map
$$
f:\prod_{i=1}^\ell\R^{n_i}\to\prod_{i=1}^\ell\R^{n_i},\ (x_1,\ldots,x_\ell)\mapsto(f_1(x_1),\ldots,f_\ell(x_\ell))
$$
is $\Ss$. By Corollary \ref{bq} there exists a polynomial map $h:\R^n\to\R^n$ such that $h(\ol{\Bb}_n)=[-1,1]^n$. The image of $\ol{\Bb}_n$ under the composition $F:=f\circ h$ is $\Ss$, as required.
\end{proof}

\subsection{Products of bricks}

As a straightforward consequence of the previous result, we show that finite products of $m_i$-bricks is an $m$-brick where $m:=\sum_im_i$.

\begin{cor}[Products of bricks]\label{pbr}
Let $m_1,\ldots,m_\ell$ be positive integers and $m:=\sum_{i=1}^\ell m_i$. Let $\Ss_i\subset\R^{n_i}$ be an $n_i$-dimensional $m_i$-brick for $i=1,\ldots,\ell$ and denote $n:=\sum_{i=1}^\ell n_i$. Then the product $\Ss:=\prod_{i=1}^\ell\Ss_i\subset\R^n$ has dimension $n$ and it is an $m$-brick.
\end{cor}
\begin{proof}
For each $i=1,\ldots,\ell$ let $H_i:=(H_{i1},\ldots,H_{in_i}):[0,1]\times\ol{\Bb}_{m_i}\to\Ss_i$ be a homotopy such that
\begin{itemize}
\item[(i)] $H_{ij}\in{\mathcal C}^0([0,1])[\x_{i1},\ldots,\x_{im_i}]$ for each $i,j$. 
\item[(ii)] $H_i(\{0\}\times\ol{\Bb}_{m_i})=\Ss_i$ and $H_i(1,\cdot)$ is a constant map.
\item[(iii)] $H_i(\{t\}\times\ol{\Bb}_{m_i})\subset\Int(\Ss_i)$ for each $t\in(0,1)$.
\end{itemize}
As $\Int(\Ss)=\prod_{i=1}^\ell\Int(\Ss_i)$, one can check that the continuous semialgebraic map
\begin{multline*}
H:=(H_1,\ldots,H_\ell):=(H_{ij}, 1\leq i\leq\ell,\ 1\leq j\leq n_i):[0,1]\times\prod_{i=1}^\ell\ol{\Bb}_{m_i}\to\Ss,\\ 
(t,x_1,\ldots,x_\ell)\mapsto (H_1(t,x_1),\ldots,H_\ell(t,x_\ell))
\end{multline*}
satisfies:
\begin{itemize}
\item[(i)] $H_{ij}\in{\mathcal C}^0([0,1])[\x_{i1},\ldots,\x_{im_i}]$ for each $i,j$. 
\item[(ii)] $H(\{0\}\times\prod_{i=1}^\ell\ol{\Bb}_{m_i})=\prod_{i=1}^\ell\Ss_i$ and $H(1,\cdot)$ is a constant map.
\item[(iii)] $H(\{t\}\times\prod_{i=1}^\ell\ol{\Bb}_{m_i})\subset\prod_{i=1}^\ell\Int(\Ss_i)=\Int(\Ss)$ for each $t\in(0,1)$.
\end{itemize}
To finish one applies Corollary \ref{pb}.
\end{proof}

\subsection{Convex hulls of semialgebraic sets}

Let $X\subset\R^n$ and denote by ${\rm conv}(X)$ the \em convex hull of $X$\em, that is, the smallest convex set that contains $X$. We have
$$
{\rm conv}(X):=\Big\{\sum_{k=0}^r\lambda_kx_k:\ r\geq1, x_0,\ldots,x_r\in X,\ \lambda_0,\ldots,\lambda_r\geq0,\ \sum_{k=0}^r\lambda_k=1\Big\}.
$$

Caratheodory's theorem provides a finitary description of the convex hull of a subset of $\R^n$.

\begin{thm}[Caratheodory, {\cite[Theorem 11.1.8.6]{ber1}}]
Let $X\subset\R^n$ be a subset. A point $x\in{\rm conv}(X)$ if and only if there exist $n+1$ points $x_0,\ldots,x_n\in X$ and non-negative real numbers $\lambda_0,\ldots,\lambda_n$ such that $\sum_{k=0}^n\lambda_k=1$ and $\sum_{k=0}^n\lambda_kx_k=x$. Consequently,
$$
{\rm conv}(X)=\Big\{\sum_{k=0}^n\lambda_kx_k:\ x_0,\ldots,x_n\in X,\ \lambda_0,\ldots,\lambda_n\geq0,\ \sum_{k=0}^n\lambda_k=1\Big\}.
$$
\end{thm}

Consider the $n$-simplex 
$$
\Delta_n:=\Big\{(\lambda_1,\ldots,\lambda_n)\in\R^n:\ \lambda_1\geq0,\ldots,\lambda_n\geq0,\sum_{k=1}^n\lambda_k\leq1\Big\}.
$$

\begin{cor}\label{convexhull0}
Let $\Ss\subset\R^n$ be a semialgebraic set. Then ${\rm conv}(\Ss)$ is the image of $\Ss\times\overset{n+1}{\cdots}\times\Ss\times\Delta_n$ under the polynomial map
$$
\varphi:\R^n\times\overset{n+1}{\cdots}\times\R^n\times\R^n\to\R^n,\ (x_0,x_1,\ldots,x_n,\lambda:=(\lambda_1,\ldots,\lambda_n))\mapsto\sum_{k=0}^n\lambda_kx_k
$$
where $\lambda_0:=1-\sum_{k=1}^n\lambda_k$.
\end{cor}

Consequently, the convex hull of a polynomial image of the $m$-dimensional closed unit ball is an $(m(n+1)+n)$-brick.

\begin{cor}\label{convexhull}
Let $\Ss\subset\R^n$ be a polynomial image of the $m$-dimensional closed unit ball. Then ${\rm conv}(\Ss)$ is an $(m(n+1)+n)$-brick.
\end{cor}
\begin{proof}
Let $f:\R^m\to\R^n$ be a polynomial map such that $f(\ol{\Bb}_m)=\Ss$. Then ${\rm conv}(\Ss)$ is the image of $\ol{\Bb}_m^{n+1}\times\Delta_n$ under the polynomial map
$$
F:\R^{m(n+1)+n}\to\R^n,\ (y_0,y_1\ldots,y_n,\lambda:=(\lambda_1,\ldots,\lambda_n))\mapsto\Big(1-\sum_{k=1}^n\lambda_k\Big)f(y_0)+\sum_{k=1}^n\lambda_kf(y_k).
$$
As $\ol{\Bb}_m^{n+1}\times\Delta_n$ is by Corollary \ref{pbr} an $(m(n+1)+n)$-brick, we deduce $\Ss$ is an $(m(n+1)+n)$-brick, as required.
\end{proof}

\begin{examples}[Toeplitz's spectrahedra and Caratheodory's orbitopes]\label{so}
There are some spectrahedra that are convex hulls of semialgebraic sets that are either polynomial or regular images of some closed unit ball.

(i) The \em Toeplitz spectrahedron \em
$$
\Ss:=\Big\{(x,y,z)\in\R^3:\ \begin{pmatrix}
1&x&y&z\\
x&1&x&y\\
y&x&1&x\\
z&y&x&1
\end{pmatrix}\ \text{is positive semidefinite}
\Big\}
$$
is the convex hull of the \em cosine moment curve\em
$$
\Cc:=\{(\cos(\theta),\cos(2\theta),\cos(3\theta)):\ \theta\in[0,\pi]\}=\{(t,2t^2-1,4t^3-3t):\ t\in[-1,1]\}.
$$
As $\Cc$ is a polynomial image of $[-1,1]$, we deduce by Corollary \ref{convexhull} that $\Ss$ is the image of $\ol{\Bb}_7$ under a polynomial map $f:\R^7\to\R^3$.

(ii) The \em Caratheodory orbitopes $\Cc$ \em are the convex hulls in $\R^{2d}$ of \em trigonometric moment curves \em 
$$
\alpha:[0,2\pi]\to\R^{2d},\ \theta\mapsto(\cos(n_1\theta),\sin(n_1\theta),\ldots,\cos(n_d\theta),\sin(n_d\theta))
$$
where $n_1,\ldots,n_d$ are non-negative integers \cite{sss}. Recall that $\cos(n_k\theta)$ is a polynomial $F_k$ in $\cos(\theta)$ of degree $n_k$, whereas $\sin(n_k\theta)$ equals $\sin(\theta)$ times a polynomial $G_k$ in $\cos(\theta)$ of degree $n_k-1$. Thus, there exists a polynomial map
$$
H:=(F_1,G_1,\ldots,F_d,G_d):\sph^1\to\R^{2d},\ (u,v)\mapsto(F_1(u),G_1(u)v,\ldots,F_d(u),G_d(u)v)
$$
such that $H(\sph^1)=\alpha([0,2\pi])$. Let $\beta:[-1,1]\to\sph^1$ be a surjective regular map (Lemma \ref{nsph}). Then $P:=H\circ\beta:[-1,1]\to\R^{2d}$ is a regular map such that $P([-1,1])=\alpha([0,2\pi])$. By Corollaries \ref{pbr} and \ref{convexhull0} the convex hull $\Cc$ is the image of $\ol{\Bb}_{4d+1}$ under a regular map $f:\R^{4d+1}\to\R^{2d}$.
\end{examples}

\subsection{Spherical stars}\label{tiles}

Let $k_1,\ldots,k_n\geq1$ be positive integers, let $\Ff:=\{i=1,\ldots,n:\ k_i\text{ is even}\}$ and denote $\Ss_0:=\bigcap_{i\in\Ff}\{\x_i\geq0\}$. We assume the convection $x_i^{2/2\ell}:=|x_i|^{1/\ell}$ for each $\ell\geq1$. The image $\sstar{k_1,\ldots,k_n}:=\{\x_1^{2/k_1}+\cdots+\x_n^{2/k_n}\leq1\}\cap\Ss_0$ of $\ol{\Bb}_n$ under $f:\R^n\to\R^n,\ (x_1,\ldots,x_n)\mapsto(x_1^{k_1},\ldots,x_n^{k_n})$ is called the \em spherical star of weights $k_1\ldots,k_n$\em. 

\begin{lem}\label{ss0}
Let $(k_1,\dots,k_n)\in\{1,2\}^n$. Then the spherical star $\sstar{k_1,\ldots,k_n}$ is convex.
\end{lem}
\begin{proof}
Let $\Ff:=\{i=1,\ldots,n:\ k_i=2\}$ and denote $\Ss_0:=\bigcap_{i\in\Ff}\{\x_i\geq0\}$, which is a convex set. As $k_i\in\{1,2\}$, also $2/k_i\in\{1,2\}$. Pick two points $x:=(x_1,\dots,x_n)$ and $y:=(y_1,\dots,y_n)$ in $\{\x_1^{2/k_1}+\cdots+\x_n^{2/k_n}\leq1\}$. We claim: \em $\lambda x+(1-\lambda)y\in\sstar{k_1,\ldots,k_n}$ for each $\lambda\in[0,1]$\em. 

Fix $i\in\{1,\dots,n\}$. If $k_i=2$, then
\begin{multline*}
(\lambda x_i+(1-\lambda) y_i)^{2/k_i}=|\lambda x_i+(1-\lambda) y_i|\le \lambda|x_i|+(1-\lambda)|y_i|=\lambda x_i^{2/k_i} +(1-\lambda) y_i^{2/k_i}.
\end{multline*}
If $k_i=1$, then
\begin{multline*}
(\lambda x_i+(1-\lambda) y_i)^{2/k_i}=(\lambda x_i+(1-\lambda) y_i)^2=\lambda^2x_i^2+(1-\lambda)^2 y_i^2+2\lambda(1-\lambda)x_iy_i\\
\le \lambda^2x_i^2+(1-\lambda)^2y_i^2+\lambda(1-\lambda)(x_i^2+y_i^2)=\lambda x_i^2 + (1-\lambda) y_i^2=\lambda x_i^{2/k_i}+(1-\lambda) y_i^{2/k_i}.
\end{multline*}
Consequently,
$$
\sum_{i=1}^n (\lambda x_i+(1-\lambda) y_i)^{2/k_i}\le \lambda \sum_{i=1}^n x_i^{2/k_i}+(1-\lambda)\sum_{i=1}^n y_i^{2/k_i}\le 1.
$$
Thus, $\sstar{k_1,\ldots,k_n}$ is convex as it is an intersection of two convex sets.
\end{proof}

\begin{cor}\label{ss}
Let $k_1,\ldots,k_n$ be positive integers. Then the spherical star $\sstar{k_1,\ldots,k_n}$ is an $n$-brick.
\end{cor}
\begin{proof}
Let $\Ff:=\{i=1,\ldots,n:\ k_i\text{ is even}\}$ and let $\Ss_0:=\bigcap_{i\in\Ff}\{\x_i\geq0\}$. Define
$$
k_i':=\begin{cases}
1&\text{if $k_i$ is odd},\\
2&\text{if $k_i$ is even}.
\end{cases}
$$
We have $\Ff=\{i=1,\ldots,n:\ k_i'=2\}$. By Lemma \ref{ss0} the spherical star $\sstar{k_1',\ldots,k_n'}$ is convex. Consider the polynomial map 
$$
f:\R^n\to\R^n,\ (x_1,\ldots,x_n)\mapsto(x_1^{k_1/k'_1},\ldots,x_n^{k_n/k'_n}).
$$
Then $f|_{\sstar{k_1',\ldots,k_n'}}$ is injective and $f(\sstar{k_1',\ldots,k_n'})=\sstar{k_1,\ldots,k_n}$. By Lemma \ref{homeo} we conclude $\sstar{k_1,\ldots,k_n}$ is an $n$-brick.
\end{proof}

\subsection{Images of revolution}

Let $m,\ell\geq1$ and write $\x':=(\x_1,\ldots,\x_{m-1})$. Denote by $\sph^\ell\subset\R^{\ell+1}$ the sphere of center the origin and radius $1$. Taking advantage of the equality 
\begin{equation}\label{Bb}
\ol{\Bb}_{m+\ell}=\varphi(\ol{\Bb}_m\times\sph^\ell)\quad\text{where $\varphi:\ol{\Bb}_m\times\sph^\ell\to\ol{\Bb}_{m+\ell},\ (x',x_m;u)\mapsto(x',x_mu)$}
\end{equation}
we prove that the revolution of certain polynomial images of $\ol{\Bb}_m$ are polynomial images of $\ol{\Bb}_{m+\ell}$.

\begin{lem}[Images of revolution]\label{revolution}
Let $F_1,\ldots,F_k\in\R[\x',\x_m^2]$ be nonzero polynomials and let $\Ss\subset\R^k$ be the image of $\ol{\Bb}_m$ under the polynomial map $F:=(F_1,\ldots,F_{k-1},\x_mF_k):\R^m\to\R^k$. Denote $\x'':=(\x_m,\ldots,\x_{m+\ell})$ and let $\psi:\Ss\times\sph^\ell\to\R^{k+\ell},\ (x',x_m;u)\mapsto(x',x_mu)$. Then $\Tt:=\psi(\Ss\times\sph^\ell)$ is the image of $\ol{\Bb}_{m+\ell}$ under the polynomial map 
$$
G:=(F_1(\x',\|\x''\|),\ldots,F_{k-1}(\x',\|\x''\|),\x''F_k(\x',\|\x''\|)):\R^{m+\ell}\to\R^{k+\ell}.
$$
\end{lem}
\begin{proof}
As $F_1,\ldots,F_k\in\R[\x',\x_m^2]$, we have $F_1(\x',\|\x''\|),\ldots,F_k(\x',\|\x''\|)\in\R[\x',\x'']$, so $G$ is a polynomial map. In addition, if $(y^*,y_k)\in\Ss$ where $y^*:=(y_1,\ldots,y_{k-1})$, then $(y^*,-y_k)\in\Ss$.

Let $x:=(x',x_mu)\in\ol{\Bb}_{m+\ell}$ where $(x',x_m)\in\ol{\Bb}_m$ and $u\in\sph^\ell$ (see \eqref{Bb}). As $F_1,\ldots,F_k\in\R[\x',\x_m^2]$,
\begin{multline*}
G(x',x_mu)=(F_1(x',\|x_mu\|),\ldots,F_{k-1}(x',\|x_mu\|),x_muF_k(x',\|x_mu\|))\\
=(F_1(x',x_m),\ldots,F_{k-1}(x',x_m),x_mF_k(x',x_m)u)\in\Tt,
\end{multline*}
so $G(\ol{\Bb}_{m+\ell})\subset\Tt$. Pick $(y^*,y^{**})\in\Tt$ where $y^*:=(y_1,\ldots,y_{k-1})$ and $y^{**}:=(y_k,\ldots,y_{k+\ell})$. If $y^{**}=0$, then $(y^*,0)\in\Ss$ and there exists $(x',x_m)\in\ol{\Bb}_m$ such that $F(x',x_m)=(y^*,0)$. The point $(x',x_m,0,\ldots,0)\in\ol{\Bb}_{m+\ell}$ and 
$$
G(x',x_m,0,\ldots,0)=(F_1(x',x_m),\ldots,F_{k-1}(x',x_m),x_mF_k(x',x_m),0,\ldots,0)=(y^*,0,\ldots,0),
$$
so $(y^*,y^{**})\in G(\ol{\Bb}_{m+\ell})$. Assume next $y^{**}\neq0$. Then $(y^*,\|y^{**}\|)\in\Ss$ and $u:=\frac{y^{**}}{\|y^{**}\|}\in\sph^\ell$. Let $(x',x_m)\in\ol{\Bb}_m$ be such that $F(x',x_m)=(y^*,\|y^{**}\|)$. Then, $(x',x_mu)\in\ol{\Bb}_{m+\ell}$ (see \eqref{Bb}) and 
\begin{multline*}
G(x',x_mu)=(F_1(x',\|x_mu\|),\ldots,F_{k-1}(x',\|x_mu\|),x_muF_k(x',\|x_mu\|))\\
=(F_1(x',x_m),\ldots,F_{k-1}(x',x_m),x_mF_k(x',x_m)u)=(y^*,y^{**})=y.
\end{multline*}
Thus, $\Tt\subset G(\ol{\Bb}_{m+\ell})$, as required.
\end{proof}

\subsection{Bricks with quadratic boundaries}\label{quadratic}

We present next families of semialgebraic sets of $\R^n$ that are $m$-bricks for some $m\geq1$ and whose boundaries are contained in finite unions of hypersurfaces of degrees $\leq2$. To make the exposition smoother we present the main statements here and postpone most of the proofs until Appendix \ref{a2}. A first easy example is the following:

\begin{cor}[Truncated cone]\label{trun}
The truncated cone $\Ss:=\{\x_1^2+\cdots+\x_{n-1}^2\leq\x_n^2,\, a\leq\x_n\leq b\}$ for $a<b$ is an $n$-brick.
\end{cor}
\begin{proof}
It is enough to observe that the convex semialgebraic set $\Ss$ is the image of the cylinder $\ol{\Bb}_{n-1}\times[a,b]$ under the polynomial map $f:\R^n\to\R^n,\ (x_1,\ldots,x_n)\mapsto(x_1x_n,\ldots,x_{n-1}x_n,x_n)$.
\end{proof}

\subsubsection{Parabolic segments}

Consider the following two types of \em parabolic segments\em:

\begin{itemize}
\item[(i)] $\Pp_n:=\{0\leq x_n\leq2(1-\|x'\|^2)\}$.
\item[(ii)] $\pseg{a}:=\{\x-\y^2\geq0,\ \sqrt{a}\y\geq\x\}\subset\R^2$ for each $a>0$ (Figure \ref{fig:parabolic}).
\end{itemize}

\begin{center}
\begin{figure}[ht]
\begin{tikzpicture}[scale=1]
\begin{axis}[axis x line=center,
axis y line=center, axis equal, no markers, xtick=\empty, ytick=\empty,
xlabel={$x$},
ylabel={$y$},
xlabel style={below right},
ylabel style={above left},
xmin=-0.2,
xmax=1.2,
ymin=-0.2,
ymax=1.2]
\addplot+[domain=0:2,samples=200,name path=A,black] {sqrt(x)};
\addplot+[domain=0:1,samples=100,name path=C,black] {sqrt(x)};
\addplot+[domain=0:1,name path=D,black] {x};
\addplot[domain=0:1,gray!60,opacity=0.5] fill between[of=C and D];
\node at (axis cs:1,-0.05){$a$};
\node at (axis cs:-0.1,1){$\sqrt{a}$};
\node at (axis cs:0.3,0.43){$\pseg{a}$};
\draw[dashed] (axis cs:1,0)--(axis cs:1,1);
\draw[dashed] (axis cs:0,1)--(axis cs:1,1);
\end{axis}
\end{tikzpicture}
\caption{Parabolic segment $\pseg{a}$\label{fig:parabolic}}
\end{figure}
\end{center}

\begin{lem}[Parabolic segments 1]\label{parseg1}
The parabolic segment $\Pp_n$ is an $n$-brick.
\end{lem}
\begin{proof}
The parabolic segment $\Pp_n$, which is a convex set, is the image of the cylinder $\ol{\Bb}_{n-1}\times[-1,1]$ under the polynomial map $f:\R^{n-1}\times\R,\ (x',x_n)\mapsto(x',(1-\|x'\|^2)(x_n+1))$.
\end{proof}

\begin{lem}[Parabolic segments 2]\label{parab}
The parabolic segment $\pseg{a}$ is a $2$-brick.
\end{lem}

\subsubsection{Elliptic sectors and segments}

For each $\alpha$ with $0<\alpha\leq\pi$ consider the following semialgebraic subsets of $\R^2$ (Figures \ref{fig:elliptic}) that we describe using polar coordinates, that is, $(x,y)=(\rho\cos(\theta),\rho\sin(\theta))\equiv(\rho,\theta)$ where $(\rho,\theta)\in[0,+\infty)\times(-\pi,\pi]$:
\begin{itemize}
\item The triangle $\tri{\alpha}:=\{-\alpha\leq\theta\leq\alpha,\ 0\le\rho\cos(\theta)\le1\}$.
\item The elliptic sector $\sector{\alpha}:=\{0\leq\rho\leq1,-\alpha\leq\theta\leq\alpha\}$.
\item The elliptic segment $\seg{\alpha}:=\{-\alpha\leq\theta\leq\alpha,\ \cos(\alpha)\le\rho\cos(\theta)\leq\cos(\theta)\}$.
\end{itemize}

\begin{center}
\begin{figure}[ht]
\begin{minipage}{0.328\textwidth}
\begin{center}
\begin{tikzpicture}[scale=1]
\begin{axis}[unit vector ratio=1 1 1, axis x line=middle,
axis y line=middle, no markers, xtick=\empty, ytick=\empty,
xlabel={$x$},
ylabel={$y$},
xlabel style={below right},
ylabel style={above left},
xmin=-0.2,
xmax=1.2,
ymin=-1.2,
ymax=1.2]
\addplot+[domain=0:1,samples=100,name path=A,black] {-sqrt(1-x^2)};
\addplot+[domain=0:1,samples=100,name path=B,black] {sqrt(1-x^2)};
\addplot+[domain=0.7:1,samples=100,name path=C,black] {-sqrt(1-x^2)};
\addplot+[domain=0.7:1,samples=100,name path=D,black] {sqrt(1-x^2)};
\node at (axis cs:0.7,0.3){$\tri{\alpha}$};
\draw[name path=E] (axis cs:1,1.02)--(axis cs:0,0);
\draw[name path=F] (axis cs:1,-1.02)--(axis cs:0,0);
\draw[] (axis cs:1,-1.02)--(axis cs:1,1.02);
\addplot[domain=1:2,gray!60,opacity=0.5] fill between[of=E and F];
\addplot+[domain=0.21:0.3,samples=50,black] {sqrt(0.09-x^2)};
\node at (axis cs:0.35,0.15){$\alpha$};
\end{axis}
\end{tikzpicture}
\end{center}
\end{minipage}
\hfil
\begin{minipage}{0.328\textwidth}
\begin{center}
\begin{tikzpicture}[scale=1]
\begin{axis}[unit vector ratio=1 1 1, axis x line=middle,
axis y line=middle, no markers, xtick=\empty, ytick=\empty,
xlabel={$x$},
ylabel={$y$},
xlabel style={below right},
ylabel style={above left},
xmin=-0.2,
xmax=1.2,
ymin=-1.2,
ymax=1.2]
\addplot+[domain=0:1,samples=100,name path=A,black] {-sqrt(1-x^2)};
\addplot+[domain=0:1,samples=100,name path=B,black] {sqrt(1-x^2)};
\addplot+[domain=0.7:1,samples=100,name path=C,black] {-sqrt(1-x^2)};
\addplot+[domain=0.7:1,samples=100,name path=D,black] {sqrt(1-x^2)};
\node at (axis cs:0.7,0.3){$\sector{\alpha}$};
\draw[name path=E] (axis cs:0.7,0.714)--(axis cs:0,0);
\draw[name path=F] (axis cs:0.7,-0.714)--(axis cs:0,0);
\addplot[domain=1:2,gray!60,opacity=0.5] fill between[of=C and D];
\addplot[domain=1:2,gray!60,opacity=0.5] fill between[of=E and F];
\addplot+[domain=0.21:0.3,samples=50,black] {sqrt(0.09-x^2)};
\node at (axis cs:0.35,0.15){$\alpha$};
\end{axis}
\end{tikzpicture}
\end{center}
\end{minipage}
\hfil
\begin{minipage}{0.328\textwidth}
\begin{center}
\begin{tikzpicture}[scale=1]
\begin{axis}[unit vector ratio=1 1 1, axis x line=center,
axis y line=center, no markers, xtick=\empty, ytick=\empty,
xlabel={$x$},
ylabel={$y$},
xlabel style={below right},
ylabel style={above left},
xmin=-0.2,
xmax=1.2,
ymin=-1.2,
ymax=1.2]
\addplot+[domain=0:1,samples=100,name path=A,black] {-sqrt(1-x^2)};
\addplot+[domain=0:1,samples=100,name path=B,black] {sqrt(1-x^2)};
\addplot+[domain=0.7:1,samples=100,name path=C,black] {-sqrt(1-x^2)};
\addplot+[domain=0.7:1,samples=100,name path=D,black] {sqrt(1-x^2)};
 
\node at (axis cs:0.85,0.2){$\seg{\alpha}$};
\draw[name path=E] (axis cs:0.7,0.714)--(axis cs:0,0);
\draw[name path=F] (axis cs:0.7,-0.714)--(axis cs:0,0);
\draw (axis cs:0.7,-0.714)--(axis cs:0.7,0.714);
\addplot[domain=1:2,gray!60,opacity=0.5] fill between[of=C and D];
\addplot+[domain=0.21:0.3,samples=50,black] {sqrt(0.09-x^2)};
\node at (axis cs:0.35,0.15){$\alpha$};
\end{axis}
\end{tikzpicture}
\end{center}
\end{minipage}
\caption{Triangle $\tri{\alpha}$, elliptic sector $\sector{a}$ and elliptic segment $\seg{a}$\label{fig:elliptic}}
\end{figure}
\end{center}

The previous semialgebraic sets are generalized to the $n$-dimensional case by means of Lemma \ref{revolution} (for $m=k=2$ and $\ell=n-2$). Define
\begin{itemize}
\item The $n$-dimensional cone $\tri{\alpha,n}:=\{(a,bu):\ (a,b)\in\tri{\alpha},\ u\in\sph^{n-2}\}$.
\item The $n$-dimensional \em elliptic sector \em $\sector{\alpha,n}:=\{(a,bu):\ (a,b)\in\sector{\alpha},\ u\in\sph^{n-2}\}$.
\item The $n$-dimensional \em elliptic segment \em $\seg{\alpha,n}:=\{(a,bu):\ (a,b)\in\seg{\alpha},\ u\in\sph^{n-2}\}$.
\end{itemize}

\begin{thm}\label{circdef}
Each elliptic sector $\sector{\alpha,n}$ and each elliptic segment $\seg{\alpha,n}$ (where $0<\alpha\leq\pi$) are $n$-bricks.
\end{thm}

\subsubsection{Hyperbolic sectors and segments}

Consider the hyperbola $\Hh:=\{\x^2-\y^2=1\}$ and recall the relation $\cos(2\theta)=\cos^2(\theta)-\sin^2(\theta)$ and its consequences:
\begin{align*}
&\Big(\frac{\cos(\theta)}{\sqrt{\cos(2\theta)}}\Big)^2-\Big(\frac{\sin(\theta)}{\sqrt{\cos(2\theta)}}\Big)^2=1,\\
&x^2-y^2=(\rho\cos(\theta))^2-(\rho\sin(\theta))^2=\rho^2\cos(2\theta).
\end{align*}
Define for each $0<\alpha<\frac{\pi}{4}$ the following semialgebraic sets (Figures \ref{fig:hyperbolic}), which we describe again using polar coordinates.
\begin{itemize}
\item The triangle $\htri{\alpha}:=\left\{-\alpha\leq\theta\leq\alpha,\ \rho\cos(\theta)\leq\frac{\cos(\alpha)}{\sqrt{\cos(2\alpha)}}\right\}$.
\item The hyperbolic sector $\hsector{\alpha}:=\{-\alpha\leq\theta\leq\alpha,\ \rho^2\cos(2\theta)\le 1\}$.
\item The hyperbolic segment $\hseg{\alpha}:=\left\{-\alpha\leq\theta\leq\alpha,\ \frac{\cos(\theta)}{\sqrt{\cos(2\theta)}}\leq\rho\cos(\theta)\leq\frac{\cos(\alpha)}{\sqrt{\cos(2\alpha)}}\right\}$.
\end{itemize}
\begin{center}
\begin{figure}[ht]
\begin{minipage}{0.328\textwidth}
\begin{center}
\begin{tikzpicture}[scale=1]
\begin{axis}[unit vector ratio=1 1 1, axis x line=center,
axis y line=center, no markers, xtick=\empty, ytick=\empty,
xlabel={$x$},
ylabel={$y$},
xlabel style={below right},
ylabel style={above left},
xmin=-0.2,
xmax=2,
ymin=-2,
ymax=2]
\addplot+[domain=1:2,samples=100,name path=A,black] {-sqrt(x^2-1)};
\addplot+[domain=1:2,samples=100,name path=B,black] {sqrt(x^2-1)};
\addplot+[domain=1:1.5,samples=100,name path=C,black] {-sqrt(x^2-1)};
\addplot+[domain=1:1.5,samples=100,name path=D,black] {sqrt(x^2-1)};
\node at (axis cs:0.8,0.3){$\htri{\alpha}$};
\draw[name path=E] (axis cs:0,0)--(axis cs:1.5,1.12);
\draw[name path=F] (axis cs:0,0)--(axis cs:1.5,-1.12);
\draw[] (axis cs:1.5,1.12)--(axis cs:1.5,-1.12);
\addplot[domain=1:2,gray!60,opacity=0.5] fill between[of=E and F];
\addplot+[domain=0.25:0.3,samples=50,black] {sqrt(0.09-x^2)};
\node at (axis cs:0.4,0.15){$\alpha$};
\end{axis}
\end{tikzpicture}
\end{center}
\end{minipage}
\hfil
\begin{minipage}{0.328\textwidth}
\begin{center}
\begin{tikzpicture}[scale=1]
\begin{axis}[unit vector ratio=1 1 1, axis x line=center,
axis y line=center, no markers, xtick=\empty, ytick=\empty,
xlabel={$x$},
ylabel={$y$},
xlabel style={below right},
ylabel style={above left},
xmin=-0.2,
xmax=2,
ymin=-2,
ymax=2]
\addplot+[domain=1:2,samples=100,name path=A,black] {-sqrt(x^2-1)};
\addplot+[domain=1:2,samples=100,name path=B,black] {sqrt(x^2-1)};
\addplot+[domain=1:1.5,samples=100,name path=C,black] {-sqrt(x^2-1)};
\addplot+[domain=1:1.5,samples=100,name path=D,black] {sqrt(x^2-1)};

\node at (axis cs:0.8,0.3){$\hsector{\alpha}$};
\draw[name path=E] (axis cs:0,0)--(axis cs:1.5,1.12);
\draw[name path=F] (axis cs:0,0)--(axis cs:1.5,-1.12);
\addplot[domain=1:2,gray!60,opacity=0.5] fill between[of=F and C];
\addplot[domain=1:2,gray!60,opacity=0.5] fill between[of=E and D];
\addplot+[domain=0.25:0.3,samples=50,black] {sqrt(0.09-x^2)};
\node at (axis cs:0.4,0.15){$\alpha$};
\end{axis}
\end{tikzpicture}
\end{center}
\end{minipage}
\hfil
\begin{minipage}{0.328\textwidth}
\begin{center}
\begin{tikzpicture}[scale=1]
\begin{axis}[unit vector ratio=1 1 1, axis x line=center,
axis y line=center, no markers, xtick=\empty, ytick=\empty,
xlabel={$x$},
ylabel={$y$},
xlabel style={below right},
ylabel style={above left},
xmin=-0.2,
xmax=2,
ymin=-2,
ymax=2]
\addplot+[domain=1:2,samples=100,name path=A,black] {-sqrt(x^2-1)};
\addplot+[domain=1:2,samples=100,name path=B,black] {sqrt(x^2-1)};
\addplot+[domain=1:1.5,samples=100,name path=C,black] {-sqrt(x^2-1)};
\addplot+[domain=1:1.5,samples=100,name path=D,black] {sqrt(x^2-1)};
\addplot[domain=1:2,gray!60,opacity=0.5] fill between[of=C and D];
\node at (axis cs:1.3,0.3){$\hseg{\alpha}$};
\draw[] (axis cs:0,0)--(axis cs:1.5,1.12);
\draw[] (axis cs:0,0)--(axis cs:1.5,-1.12);
\draw[] (axis cs:1.5,1.12)--(axis cs:1.5,-1.12);
\addplot+[domain=0.25:0.3,samples=50,black] {sqrt(0.09-x^2)};
\node at (axis cs:0.4,0.15){$\alpha$};
\end{axis}
\end{tikzpicture}
\end{center}
\end{minipage}
\caption{Triangle $\htri{\alpha}$, hyperbolic sector $\hsector{\alpha}$ and hyperbolic segment $\hseg{\alpha}$\label{fig:hyperbolic}}
\end{figure}
\end{center}
The previous semialgebraic sets are generalized to the $n$-dimensional case by means of Lemma \ref{revolution} (for $m=k=2$ and $\ell=n-2$). Define
\begin{itemize}
\item The $n$-dimensional cone $\htri{\alpha,n}:=\{(a,bu):\ (a,b)\in\htri{\alpha},\ u\in\sph^{n-2}\}$.
\item The $n$-dimensional \em hyperbolic sector \em $\hsector{\alpha,n}:=\{(a,bu):\ (a,b)\in\hsector{\alpha},\ u\in\sph^{n-2}\}$.
\item The $n$-dimensional \em hyperbolic segment \em $\hseg{\alpha,n}:=\{(a,bu):\ (a,b)\in\hseg{\alpha},\ u\in\sph^{n-2}\}$.
\end{itemize}

\begin{thm}\label{hiperdef}
Each hyperbolic sector $\hsector{\alpha,n}$ and each hyperbolic segment $\hseg{\alpha,n}$ (where $0<\alpha<\frac{\pi}{4}$) are $n$-bricks.
\end{thm}

\section{Drawing polynomial paths inside semialgebraic sets}\label{s3}

The main purpose of this section is to prove Lemma \ref{smart}, which has a great influence in the proofs of Theorems \ref{main1-i}, \ref{main1-i2} and \ref{main2-i}.

Let $\Ss_1,\Ss_2\subset\R^n$ be two open semialgebraic sets. A \em (polynomial) bridge between $\Ss_1$ and $\Ss_2$ \em is the image $\Gamma$ of a polynomial arc $\alpha:[-1,1]\to\R^n$ such that $\alpha([-1,0))\subset\Ss_1$ and $\alpha((0,1])\subset\Ss_2$. The point $\alpha(0)$ is called the \em base point \em of $\Gamma$. Recall that a map $f:\Ss\to\R^m$ on a semialgebraic set $\Ss\subset\R^n$ is Nash if there exists an open semialgebraic neighborhood $U\subset\R^n$ of $\Ss$ and a smooth and semialgebraic map $F:U\to\R^m$ that extends $f$. Recall that Nash maps on open semialgebraic sets are analytic maps \cite[Prop.8.1.8]{bcr}.

Let $\alpha:[a,b]\to\R^n$ be a continuous semialgebraic path. By \cite[Prop.2.9.10]{bcr} there exists a minimal finite set $\eta(\alpha)\subset[a,b]$ such that $\alpha|_{[a,b]\setminus\eta(\alpha)}$ is a Nash map. By \cite[Prop.8.1.12]{bcr} and after reparameterizing $\alpha$ locally at $a,b$ (if necessary) we may assume that $\alpha$ is analytic at the points $a,b$, so $\eta(\alpha)\subset(a,b)$. 

\begin{center}
\begin{figure}[ht]
\begin{tikzpicture}[scale=1.25]
\draw[dashed,line width=1.5pt] (0,0) -- (0,2) -- (2,2) -- (0,0);
\draw[dashed,line width=1.5pt] (2,2) -- (4,2) -- (4,0) -- (2,2);
\draw[dashed,line width=1.5pt] (5,1) circle (1cm);
\draw[dashed,line width=1.5pt] (6,1) arc (270:360:1cm) arc (180:270:1cm) arc (90:180:1cm) arc (0:90:1cm);
\draw[dashed,line width=1.5pt] (8,1) -- (9,1) -- (9,2) -- (8,1);
\draw[dashed,line width=1.5pt] (9,1) -- (10,1) -- (10,0) -- (9,1);

\draw[fill=gray!100,opacity=0.5,draw=none] (0,0) -- (0,2) -- (2,2) -- (0,0);
\draw[fill=gray!100,opacity=0.5,draw=none] (2,2) -- (4,2) -- (4,0) -- (2,2);
\draw[fill=gray!100,opacity=0.5,draw=none] (5,1) circle (1cm);
\draw[fill=gray!100,opacity=0.5,draw=none] (6,1) arc (270:360:1cm) arc (180:270:1cm) arc (90:180:1cm) arc (0:90:1cm);
\draw[fill=gray!100,opacity=0.5,draw=none] (8,1) -- (9,1) -- (9,2) -- (8,1);
\draw[fill=gray!100,opacity=0.5,draw=none] (9,1) -- (10,1) -- (10,0) -- (9,1);

\draw[color=red,line width=1.5pt] (1.3,1.875) arc (110:70:2cm);
\draw[color=red,line width=1.5pt] (3.5,1) -- (4.5,1);
\draw[color=red,line width=1.5pt] (5.5,1) -- (6.5,1);
\draw[color=red,line width=1.5pt] (7.25,1.125) arc (250:290:2cm);

\draw[color=red,line width=1.5pt] (8.5,1.25) -- (9.5,0.75);

\draw (2,2) node{$\bullet$};
\draw (4,1) node{$\bullet$};
\draw (6,1) node{$\bullet$};
\draw (8,1) node{$\bullet$};
\draw (9,1) node{$\bullet$};

\draw (2,2.3) node{$q_1$};
\draw (3.6,1.3) node{$q_2$};
\draw (6.2,1.3) node{$q_3$};
\draw (8,1.3) node{$q_4$};
\draw (9,0.7) node{$q_5$};

\draw (0,1) node{$\bullet$};
\draw (4,2) node{$\bullet$};
\draw (5,0) node{$\bullet$};
\draw (7,0) node{$\bullet$};
\draw (9,2) node{$\bullet$};
\draw (10,0) node{$\bullet$};

\draw (-0.3,1) node{$p_1$};
\draw (4.3,2.3) node{$p_2$};
\draw (5,-0.3) node{$p_3$};
\draw (7,-0.3) node{$p_4$};
\draw (9,2.3) node{$p_5$};
\draw (10,-0.3) node{$p_6$};


\draw[dashed] (0,-0.75) -- (10,-0.75);

\end{tikzpicture}

\vspace*{1mm}
\begin{tikzpicture}[scale=1.25]
\draw[dashed,line width=1.5pt] (0,0) -- (0,2) -- (2,2) -- (0,0);
\draw[dashed,line width=1.5pt] (2,2) -- (4,2) -- (4,0) -- (2,2);
\draw[dashed,line width=1.5pt] (5,1) circle (1cm);
\draw[dashed,line width=1.5pt] (6,1) arc (270:360:1cm) arc (180:270:1cm) arc (90:180:1cm) arc (0:90:1cm);
\draw[dashed,line width=1.5pt] (8,1) -- (9,1) -- (9,2) -- (8,1);
\draw[dashed,line width=1.5pt] (9,1) -- (10,1) -- (10,0) -- (9,1);

\draw[fill=gray!100,opacity=0.5,draw=none] (0,0) -- (0,2) -- (2,2) -- (0,0);
\draw[fill=gray!100,opacity=0.5,draw=none] (2,2) -- (4,2) -- (4,0) -- (2,2);
\draw[fill=gray!100,opacity=0.5,draw=none] (5,1) circle (1cm);
\draw[fill=gray!100,opacity=0.5,draw=none] (6,1) arc (270:360:1cm) arc (180:270:1cm) arc (90:180:1cm) arc (0:90:1cm);
\draw[fill=gray!100,opacity=0.5,draw=none] (8,1) -- (9,1) -- (9,2) -- (8,1);
\draw[fill=gray!100,opacity=0.5,draw=none] (9,1) -- (10,1) -- (10,0) -- (9,1);


\draw (2,2) node{$\bullet$};
\draw (4,1) node{$\bullet$};
\draw (6,1) node{$\bullet$};
\draw (8,1) node{$\bullet$};
\draw (9,1) node{$\bullet$};

\draw (2,2.3) node{$q_1$};
\draw (3.6,1.3) node{$q_2$};
\draw (6.2,1.3) node{$q_3$};
\draw (8,1.3) node{$q_4$};
\draw (9,0.7) node{$q_5$};

\draw (0,1) node{$\bullet$};
\draw (4,2) node{$\bullet$};
\draw (5,0) node{$\bullet$};
\draw (7,0) node{$\bullet$};
\draw (9,2) node{$\bullet$};
\draw (10,0) node{$\bullet$};

\draw (-0.3,1) node{$p_1$};
\draw (4.3,2.3) node{$p_2$};
\draw (5,-0.3) node{$p_3$};
\draw (7,-0.3) node{$p_4$};
\draw (9,2.3) node{$p_5$};
\draw (10,-0.3) node{$p_6$};

\draw[color=blue,line width=1.5pt] (0,1) .. controls (0.5,1.5) and (1.75,2) .. (2,2) .. controls (2.5,1.9) and (3.5,1.5) .. (4,2) .. controls (3.75,2) and (3.75,1) .. (4,1) .. controls (4,1) and (5,1) .. (5,0) .. controls (5,1) and (6,1) .. (6,1) .. controls (6,1) and (7,1.1) .. (7,0) .. controls (7,1.1) and (8,1) .. (8,1) .. controls (8,1) and (9,1.25) .. (9,2) .. controls (9,1.5) and (8.5,1.25) .. (9,1) .. controls (9,1) and (10,1) .. (10,0);

\end{tikzpicture}
\caption{Illustration of the statement of Lemma \ref{smart}\label{smartfig}}
\end{figure}
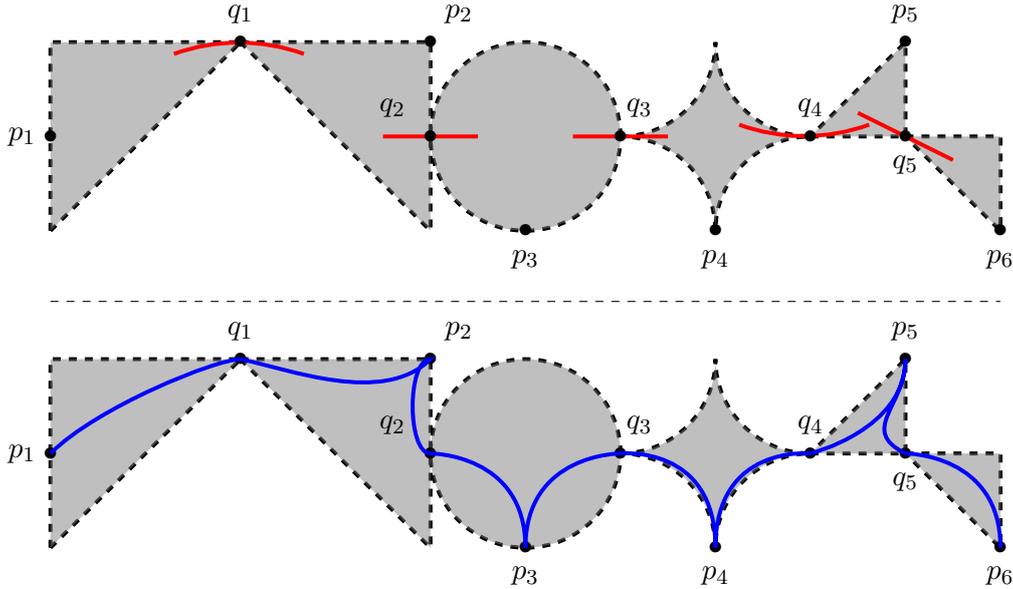
\end{center}

\begin{lem}[Smart polynomial curve]\label{smart}
Let $\Ss_1,\ldots,\Ss_\ell\subset\R^n$ be connected open semialgebraic sets (not necessarily pairwise different) and denote $\Ss:=\bigcup_{i=1}^\ell\Ss_i$. Pick points $p_i\in\cl(\Ss_i)$ and assume there exists a bridge $\Gamma_i$ between $\Ss_i$ and $\Ss_{i+1}$. Denote by $q_i\in\cl(\Ss_i)\cap\cl(\Ss_{i+1})$ the base point of $\Gamma_i$. Fix real values $s_0:=0<t_1<\cdots<t_\ell<1=:s_\ell$ and $s_i\in(t_i,t_{i+1})$ for $i=1,\ldots,\ell-1$. Then there exists a polynomial path $\alpha:\R\to\R^n$ that satisfies: 
\begin{itemize}
\item[(i)] $\alpha([0,1])\subset\Ss\cup\{p_1,\ldots,p_\ell,q_1,\ldots,q_{\ell-1}\}$.
\item[(ii)] $\alpha(t_i)=p_i$ for $i=1,\ldots,\ell$.
\item[(iii)] $\alpha((t_i,s_i))\subset\Ss_i$, $\alpha((s_i,t_{i+1}))\subset\Ss_{i+1}$ and $\alpha(s_i)=q_i$ for each $i$.
\end{itemize}
In addition, if $\veps>0$ and $\beta:[0,1]\to\Ss\cup\{p_1,\ldots,p_\ell,q_1,\ldots,q_{\ell-1}\}$ is a continuous semialgebraic path such that $\eta(\beta)\subset(0,1)\setminus\{t_1,\ldots,t_\ell,s_1,\ldots,s_{\ell-1}\}$, $\beta(\eta(\beta))\subset\Ss$ and satisfies conditions \em (ii) \em and \em (iii) \em above, we may assume that $\|\alpha-\beta\|<\veps$.
\end{lem}

The previous result (illustrated in Figure \ref{smartfig}) is improved in \cite{f3} in two directions: it is extended to general semialgebraic sets that are connected by analytic paths (using Nash paths instead polynomial paths) and we provide a constructive version using Bernstein's polynomials. The proof of Lemma \ref{smart} we present here is not constructive, but its presentation is shorter and less demanding than the one in \cite{f3}. We begin with some preliminary results.

\subsection{Polynomial approximation combined with interpolation}

We consider on the space $\Cont^\nu([a,b],\R)$ of differentiable functions of class $\Cont^\nu$ on the interval $[a,b]$ the $\Cont^\nu$ compact-open topology. Recall that a basis of open neighborhoods of $g\in\Cont^\nu([a,b],\R)$ in this topology is constituted by the sets of the type:
$$
{\mathcal U}^\nu_{g,\veps}:=\{f\in\Cont^\nu([a,b],\R):\ \|f^{(\ell)}-g^{(\ell)}\|_{[a,b]}<\veps:\ \ell=0,\ldots,\nu\}
$$
where $\veps>0$ and $\|h\|_{[a,b]}:=\max\{h(x):\ x\in[a,b]\}$. Observe that $\Cont^\nu([a,b],\R^n)=\Cont^\nu([a,b],\R)\times\overset{(n)}{\cdots}\times\Cont^\nu([a,b],\R)$ and we endow this space with the product topology. If $X\subset[a,b]$, one defines analogously the $\Cont^\nu$ compact-open topology of the space $\Cont^\nu(X,\R^n)$. The following result is well-known and its proof follows straightforwardly from \cite[\S2.5. Ex.10, pp. 64-65]{hir3} using standard arguments.
\begin{lem}\label{cont}
Let $U\subset\R^n$ be an open set and let $\varphi:U\to\R^m$ be a $\Cont^\ell$ map for some $0\leq\ell\leq\nu$. Consider the map $\varphi^*:\Cont^\nu([a,b],U)\to\Cont^\ell([a,b],\R^m),\ f\mapsto\varphi\circ f$, where both spaces are endowed with their $\Cont^\ell$ compact-open topologies. Then $\varphi^*$ is continuous. 
\end{lem}

In addition, one has the following.
\begin{lem}\label{cont2}
Let $X\subset[a,b]$ and consider the restriction map
$$
\rho:\Cont^\nu([a,b],\R^n)\to\Cont^\nu(X,\R^n),\ f\mapsto f|_{X},
$$ 
where the spaces are endowed with their respective $\Cont^\nu$ compact-open topologies. Then $\rho$ is continuous and if in addition $X\subset[a,b]$ is closed, then $\rho$ is surjective. 
\end{lem}

We borrow the following result from \cite{b} that combines polynomial approximation with interpolation on a finite set. We include full details for the sake of completeness.

\begin{lem}\label{swdp}
Let $a<t_1<\cdots<t_r<b$ be real numbers and let $f:[a,b]\to\R$ be a $\Cont^\nu$ function. Write $a_{ik}:=f^{(k)}(t_i)$ for $i=1,\ldots,r$ and $0\leq k\leq\nu$. Fix $\veps>0$. Then there exists a polynomial $g\in\R[\t]$ such that: 
\begin{itemize}
\item[(i)] $\|f^{(k)}-g^{(k)}\|_{[a,b]}<\veps$ for $k=0,\ldots,\nu$.
\item[(ii)] $g^{(k)}(t_i)=a_{ik}$ for $i=1,\ldots,r$ and $0\leq k\leq\nu$. 
\end{itemize}
\end{lem}
\begin{proof}
The proof is conducted in two steps:

\noindent{\sc Step 1}. \em There exists a polynomial $h\in\R[\t]$ such that $\|h^{(k)}-f^{(k)}\|_{[a,b]}<\veps$ for $k=0,\ldots,\nu$ \em (Condition (i) in the statement).

We proceed by induction on the integer $\nu\geq0$. If $\nu=0$, the result is classical Stone-Weierstrass' polynomial approximation theorem. Assume the result is true for $\nu-1\geq0$ and let us check that it is also true for $\nu$.

Consider the $\Cont^{\nu-1}$ function $f'$ on the interval $[a,b]$ and extend it as a $\Cont^{\nu-1}$ function to a bigger interval $[a',b']$ that contains $[a,b]$ in its interior. By induction hypotheses there exists a polynomial $h_0\in\R[\t]$ such that $|f^{(k+1)}-h_0^{(k)}|<\frac{\veps}{1+(b-a)}$ on $[a',b']$ for $k=0,\ldots,\nu-1$. By Barrow's rule
$$
f(t)=f(a)+\int_a^tf'(s)ds.
$$
Define 
$$
h(t):=f(a)+\int_a^th_0(s)ds,
$$
which is a polynomial of $\R[\t]$. Observe that $h'=h_0$, so 
$$
\|h^{(k)}-f^{(k)}\|_{[a,b]}=\|h_0^{(k-1)}-f^{(k)}\|_{[a,b]}<\tfrac{\veps}{1+(b-a)}<\veps
$$
for $k=1,\ldots,\nu$. In addition,
\begin{multline*}
\|h-f\|_{[a,b]}=\Big\|\int_a^th_0(s)ds-\int_a^tf'(s)ds\Big\|_{[a,b]}=\max_{[a,b]}\Big\{\Big|\int_a^t(h_0(s)-f'(s))ds\Big|\Big\}\\
\leq\max_{[a,b]}\Big\{\int_a^t|h_0(s)-f'(s)|ds\Big\}<(b-a)\frac{\veps}{1+(b-a)}<\veps.
\end{multline*}

\noindent{\sc Step 2}. We show how to modify $h$ in order to have also condition (ii). 

Take polynomials $P_{ik}$ such that 
$$
P_{ik}^{(\ell)}(t_j)=\begin{cases}
0&\text{if $i\neq j$ or $k\neq\ell$},\\
1&\text{if $i=j$ and $k=\ell$,}
\end{cases}
$$ 
for $i=1,\ldots,r$ and $0\leq k,\ell\leq\nu$. For instance, {\em we may take
$$
P_{ik}:=c_{ik}(\t-\t_i)^k\prod_{j\neq i}((\t-t_i)^{\nu+1}-(t_j-t_i)^{\nu+1})^{\nu+1}
$$ 
for $c_{ik}:=\frac{1}{k!}\frac{(-1)^{\nu+1}}{\prod_{j\neq i}(t_j-t_i)^{(\nu+1)^2}}$.} 

The Taylor expansion of $P_{ik}$ at $t_i$ has the form
$$
P_{ik}=\frac{1}{k!}(\t-t_i)^k+d_{ik}(\t-t_i)^{\nu+1}+\cdots
$$
for some $d_{ik}\in\R$, whereas the Taylor expansion of $P_{ik}$ at $t_j$ (for $j\neq i$) has the form
\begin{multline*}
P_{ik}=e_{ik}(\t-t_j)^{\nu+1}+\cdots\\
\text{where } e_{ik}:=\Big(c_{ik}(t_j-t_i)^k((\nu+1)(t_j-t_i)^{\nu})^{\nu+1}\prod_{m\neq i,j}((t_j-t_i)^{\nu+1}-(t_m-t_i)^{\nu+1})^{\nu+1}\Big).
\end{multline*}
In both cases above the symbol $+\cdots$ means `plus terms of higher degree' with respect to either $\t-t_i$ or $\t-t_j$ depending on each case.

Define
$$
M:=\max\{\|P_{ik}^{(\ell)}\|_{[a,b]}:\ 1\leq i\leq r,\ 0\leq k,\ell\leq\nu\}\quad\text{and}\quad\delta:=\frac{\veps}{1+r(\nu+1)M}.
$$
Let $h\in\R[\t]$ be a polynomial such that $\|h^{(k)}-f^{(k)}\|_{[a,b]}<\delta$ for $k=0,\ldots,\nu$. Define
$$
g:=h+\sum_{i=1}^r\sum_{k=0}^{\nu}b_{ik}P_{ik}
$$
where $b_{ik}:=a_{ik}-h^{(k)}(t_i)=f^{(k)}(t_i)-h^{(k)}(t_i)$ for $i=1,\ldots,r$ and $k=0,\ldots,\nu$. Thus, 
$$
g^{(\ell)}(t_j)=h^{(\ell)}(t_j)+\sum_{i=1}^r\sum_{k=0}^{\nu}b_{ik}P_{ik}^{(\ell)}(t_j)=h^{(\ell)}(t_j)+b_{j\ell}=a_{j\ell}=f^{(\ell)}(t_j)
$$
for $j=1,\ldots,r$ and $\ell=0,\ldots,\nu$.

Observe that $|b_{ik}|=|f^{(k)}(t_i)-h^{(k)}(t_i)|<\delta$, so
$$
\|g^{(\ell)}-f^{(\ell)}\|_{[a,b]}\leq\|h^{(\ell)}-f^{(\ell)}\|_{[a,b]}+\sum_{i=1}^r\sum_{k=0}^{\nu}|b_{ik}|\|P_{ik}^{(\ell)}\|_{[a,b]}<\delta+r(\nu+1)M\delta=\veps, 
$$
for each $\ell=0,\ldots,\nu$, as required.
\end{proof}

\subsection{Polynomial paths with prescribed behavior at points and intervals}

We prove next (as a consequence of Lemma \ref{swdp}) a key result to prove Lemma \ref{smart}. When we write a series in the form $h:=a_k\t^k+\cdots$, we mean that the lowest order term is $a_k\t^k$ (with $a_k\neq0$) and the remaining terms have higher order and are not relevant for our computation. 

\begin{lem}\label{clue}
Let $\Ss_0,\ldots,\Ss_r\subset\R^n$ be connected open semialgebraic sets (not necessarily pairwise different) and pick points $z_i\in\cl(\Ss_{i-1})\cap\cl(\Ss_i)$ for $i=1,\ldots,r$. Assume that there exist a $\Cont^\nu$ path $\beta:[a,b]\to\bigcup_{k=0}^r\Ss_k\cup\{z_1,\ldots,z_r\}$ (for some $\nu\geq0$) and values $a:=t_0<t_1<\cdots<t_r<t_{r+1}:=b$ satisfying the following properties:
\begin{itemize}
\item[(i)] $\beta((t_k,t_{k+1}))\subset\Ss_k$ for $k=0,\ldots,r$,
\item[(ii)] $\beta(t_i)=z_i$ and $\beta$ is an analytic path in a neighborhood of $t_i$ for $i=1,\ldots,r$,
\item[(iii)] there exist polynomials $f_{ij}\in\R[\x]$ such that $\{f_{i1}>0,\ldots,f_{is}>0\}\subset\Ss_{i-1}$ is adherent to $z_i$ and the analytic series $(f_{ij}\circ\beta)(t_i-\t)$ has the form $a_{ij}\t^{n_{ij}}+\cdots$ where $a_{ij}>0$,
\item[(iv)] there exist polynomials $g_{ij}\in\R[\x]$ such that $\{g_{i1}>0,\ldots,g_{is}>0\}\subset\Ss_i$ is adherent to $z_i$ and the analytic series $(g_{ij}\circ\beta)(t_i+\t)$ has the form $b_{ij}\t^{p_{ij}}+\cdots$ where $b_{ij}>0$,
\item[(v)] $n_{ij},p_{ij}<\nu$ for each $i,j$.
\end{itemize}
We have: 
\begin{itemize}
\item[(1)] There exists an open neighborhood ${\mathcal U}$ of $\beta$ in the $\Cont^\nu$-topology such that if $\alpha\in{\mathcal U}$ and $\alpha^{(m)}(t_i)=\beta^{(m)}(t_i)$ for $i=1,\ldots,r$ and $m=0,\ldots,\nu-1$, then $\alpha((t_k,t_{k+1}))\subset\Ss_k$ for $k=0,\ldots,r$. 
\item[(2)] There exists a polynomial path $\alpha:[a,b]\to\bigcup_{k=0}^r\Ss_k\cup\{z_1,\ldots,z_r\}$ close to $\beta$ in the $\Cont^\nu$-topology such that $\alpha(t_i)=z_i$ for $i=1,\ldots,r$ and $\alpha((t_k,t_{k+1}))\subset\Ss_k$ for $i=0,\ldots,r$.
\end{itemize}
\end{lem}
\begin{proof}
We prove this result as an application of Lemma \ref{swdp}. Observe that $(f_{ij}\circ\beta)^{(n_{ij})}(t_i)>0$ and $(g_{ij}\circ\beta)^{(p_{ij})}(t_i)>0$. Thus, there exists $\delta>0$ such that the compact interval $I_i:=[t_i-\delta,t_i+\delta]\subset[a,b]$, $(f_{ij}\circ\beta|_{I_i})^{(n_{ij})}>0$ and $(g_{ij}\circ\beta|_{I_i})^{(p_{ij})}>0$ for $i=1,\ldots,r$ and $j=1,\ldots,s$. Denote $J_0:=[a,t_1-\delta]$, $J_k:=[t_k+\delta,t_{k+1}-\delta]$ for $k=1,\ldots,r-1$ and $J_r:=[t_r+\delta,b]$. By Lemmas \ref{cont} and \ref{cont2} the maps
\begin{align*}
&\varphi_{ij}:\Cont^\nu([a,b],\R^n)\to\Cont^\nu(I_i,\R),\ \gamma\mapsto f_{ij}\circ\gamma|_{I_i},\\
&\phi_{ij}:\Cont^\nu([a,b],\R^n)\to\Cont^\nu(I_i,\R),\ \gamma\mapsto g_{ij}\circ\gamma|_{I_i},\\
&\psi_k:\Cont^\nu([a,b],\R^n)\to\Cont^0(J_k,\R),\ \gamma\mapsto\dist(\gamma|_{J_k}(\t),\R^n\setminus \Ss_k)
\end{align*}
are continuous. In addition, each function $\psi_k(\beta)$ is strictly positive. Define 
$$
\veps:=\min_{i,j,k}\{\min\{(f_{ij}\circ\beta|_{I_i})^{(n_{ij})}\},\min\{(g_{ij}\circ\beta|_{I_i})^{(p_{ij})}\},\min\{\psi_k(\beta)\}\}>0
$$ 
and consider
\begin{equation*}
\begin{split}
{\mathcal U}_0&:=\bigcap_{i=1}^r\bigcap_{j=1}^s\{\gamma\in\Cont^\nu([a,b],\R^n):\ |\varphi_{ij}(\gamma|_{I_i})^{(n_{ij})}-\varphi_{ij}(\beta|_{I_i})^{(n_{ij})}|<\veps\}\\
&\cap\bigcap_{i=1}^r\bigcap_{j=1}^s\{\gamma\in\Cont^\nu([a,b],\R^n):\ |\phi_{ij}(\gamma|_{I_i})^{(p_{ij})}-\phi_{ij}(\beta|_{I_i})^{(p_{ij})}|<\veps\}\\ 
&\cap\bigcap_{k=0}^r\{\gamma\in\Cont^\nu([a,b],\R^n):\ |\psi_k(\gamma|_{J_k})-\psi_k(\beta|_{J_k})|<\veps\},
\end{split}
\end{equation*}
which is an open subset of $\Cont^\nu([a,b],\R^n)$. Then there exists $\delta>0$ such that 
$$
{\mathcal U}:=\{\gamma\in\Cont^\nu([a,b],\R^n):\ \|\gamma^{(m)}-\beta^{(m)}\|_{[a,b]}<\delta,\ m=0,\ldots,\nu\}\subset{\mathcal U}_0.
$$

We are ready to prove the assertions in the statement:

(1) We claim: \em If $\alpha\in{\mathcal U}$ and $\alpha^{(m)}(t_i)=\beta^{(m)}(t_i)$ for $i=1,\ldots,r$ and $m=0,\ldots,\nu$, then $\alpha((t_k,t_{k+1}))\subset\Ss_k$ for $k=0,\ldots,r$\em.

It holds $\alpha(J_k)\subset\Ss_k$, because $\alpha\in\{\gamma\in\Cont^\nu([a,b]):\ |\psi_k(\gamma|_{J_k})-\psi_k(\beta|_{J_k})|<\veps\}$. Thus, to prove the claim it is enough to check:
\begin{align}
\label{2161}&\alpha([t_i-\delta,t_i))\subset\{f_{i1}>0,\ldots,f_{is}>0\}\subset\Ss_{i-1},\\
\label{2162}&\alpha((t_i,t_i+\delta])\subset\{g_{i1}>0,\ldots,g_{is}>0\}\subset\Ss_i
\end{align}
for $i=1,\ldots,r$. We show only \eqref{2161} because the proof of \eqref{2162} is analogous.

Using Taylor's expansion, we know that $\alpha$ around $t_i$ has the form
\begin{multline*}
\alpha(\t)=\sum_{m=0}^{\nu-1}\frac{1}{m!}\alpha^{(m)}(t_i)(\t-t_i)^m+(\t-t_i)^\nu\theta(\t-t_i)\\
=\sum_{m=0}^{\nu-1}\frac{1}{m!}\beta^{(m)}(t_i)(\t-t_i)^m+(\t-t_i)^\nu\theta(\t-t_i)
\end{multline*}
where $\theta$ is a continuous map defined on an interval around $0$. As $\beta$ is analytic in a neighborhood of $t_i$, there exists a tuple of analytic series $\tau\in\R\{\t\}^n$ such that
$$
\beta(\t)=\sum_{m=0}^{\nu-1}\frac{1}{m!}\beta^{(m)}(t_i)(\t-t_i)^m+(\t-t_i)^\nu\tau(\t-t_i).
$$
Thus, if $\zeta:=\theta-\tau$, which is a continuous function around $0$, we deduce 
$$
\alpha(\t)-\beta(\t)=(\t-t_i)^\nu\zeta(\t-t_i)\quad\leadsto\quad\alpha(t_i-\t)-\beta(t_i-\t)=(-\t)^\nu\zeta(-\t).
$$
Write $\x:=(\x_1,\ldots,\x_n)$, $\y:=(\y_1,\ldots,\y_n)$ and let $\z$ be a single variable. As the polynomial $f_{ij}(\x+\z\y)-f_{ij}(\x)$ vanishes on the real algebraic set $\{\z=0\}$, there exists a polynomial $h_{ij}\in\R[\x,\y,\z]$ such that
$$
f_{ij}(\x+\z\y)=f_{ij}(\x)+\z h_{ij}(\x,\y,\z).
$$
As $\nu>n_{ij}$, we deduce
\begin{multline*}
f_{ij}(\alpha(t_i-\t))=f_{ij}(\beta(t_i-\t)+\alpha(t_i-\t)-\beta(t_i-\t))\\
=f_{ij}(\beta(t_i-\t))+(-1)^\nu\t^\nu h_{ij}(\beta(t_i-\t),\zeta(-\t),(-1)^\nu\t^\nu)
=a_{ij}\t^{n_{ij}}+\cdots.
\end{multline*}
Consequently, $(f_{ij}\circ\alpha)^{(m)}(t_i)=0$ for $m=0,\ldots,n_{ij}-1$ and $(f_{ij}\circ\alpha)^{(n_{ij})}(t_i)=n_{ij}!\,a_{ij}>0$. In addition, $\alpha(t_i-t)\in\{f_{i1}>0,\ldots,f_{is}>0\}$ for $t\in(0,\delta)$ close to $0$.

As $(f_{ij}\circ\beta|_{I_i})^{(n_{ij})}(t_i-\t)>\veps>0$ on $[-\delta,\delta]$ and $|(f_{ij}\circ\beta|_{I_i})^{(n_{ij})}-(f_{ij}\circ\alpha|_{I_i})^{(n_{ij})}|<\veps$, we conclude that $(f_{ij}\circ\alpha|_{I_i})^{(n_{ij})}(t_i-\t)>0$ on $[-\delta,\delta]$ for each $j=1,\ldots,s$. Suppose there exists $t^*\in[t_i-\delta,t_i)$ such that $\alpha(t^*)\not\in\{f_{i1}>0,\ldots,f_{is}>0\}$ and assume $(f_{i1}\circ\alpha)(t^*)\leq0$. As $\alpha(t_i-t)\in\{f_{i1}>0,\ldots,f_{is}>0\}$ for $t\in(0,\delta)$ close to $0$, there exists $\xi_0\in(0,\delta)$ such that $(f_{i1}\circ\alpha)(t_i-\xi_0)=0$. Assume by induction on $m\leq n_{i1}-1$ that there exist values $0<\xi_m<\cdots<\xi_1<\xi_0<\delta$ such that $(f_{i1}\circ\alpha)^{(j)}(t_i-\xi_j)=0$ for $j=0,\ldots,m$. As $(f_{i1}\circ\alpha)^{(m)}(t_i)=0$ and $(f_{i1}\circ\alpha)^{(m)}(t_i-\xi_m)=0$, there exists by Rolle's theorem $\xi_{m+1}\in(0,\xi_m)$ such that $(f_{i1}\circ\alpha)^{(m+1)}(t_i-\xi_{m+1})=0$. In particular, $(f_{i1}\circ\alpha)^{(n_{i1})}(t_i-\xi_{n_{i1}})=0$ and $\xi_{n_{i1}}\in(0,\delta)$, which contradicts the fact that $(f_{i1}\circ\alpha|_{I_i})^{(n_{i1})}(t_i-\t)>0$ on $[-\delta,\delta]$. Consequently, $\alpha(t)\in\{f_{i1}>0,\ldots,f_{is}>0\}$ for each $t\in[t_i-\delta,t_i)$.

(2) By Lemma \ref{swdp} there exists a polynomial tuple $\alpha\in\R[\t]^n$ such that $\|\alpha^{(m)}-\beta^{(m)}\|_{[a,b]}<\delta$ for $m=0,\ldots,\nu$ (that is, $\alpha\in{\mathcal U}$) and $\alpha^{(m)}(t_i)=\beta^{(m)}(t_i)$ for $i=1,\ldots,r$ and $m=0,\ldots,\nu$. By (1) we deduce $\alpha((t_k,t_{k+1}))\subset\Ss_k$ for $k=0,\ldots,r$, as required.
\end{proof}

\begin{remark}
If $\Ss_{i-1}=\Ss_i$ for some $i=1,\ldots,r$ in the statement of Lemma \ref{clue}, the condition $z_i\in\cl(\Ss_{i-1})\cap\cl(\Ss_i)$ means $z_i\in\cl(\Ss_i)$ and condition (i) reads as $\beta((t_{i-1},t_{i+1})\setminus\{t_i\})\subset\Ss_i$. The reader has to take this into account when applying Lemma \ref{clue} to prove Lemma \ref{smart}.
\end{remark}

We still need some preliminary results to prove Lemma \ref{smart}:

\begin{cor}[Connexion by polynomial paths]\label{polcon}
Let $\Ss\subset\R^n$ be a connected open semialgebraic set and let $x,y\in\Ss$. Then there exists a polynomial path $\alpha:[0,1]\to\Ss$ such that $\alpha(0)=x$ and $\alpha(1)=y$. 
\end{cor}
\begin{proof}
As $\Ss$ is connected, it is connected by semialgebraic paths \cite[Prop.2.5.13]{bcr}. Thus, there exists a semialgebraic path $\beta:[0,1]\to\Ss$ such that $\beta(0)=x$ and $\beta(1)=y$. By \cite[Prop.8.1.12]{bcr} and after reparameterizing (if necessary) we may assume that $\beta$ is Nash at the points $0,1$ and we extended $\beta$ continuously and semialgebraically to an interval $[-\veps,1+\veps]$. By Lemma \ref{clue}(2) there exists a polynomial path $\alpha:[0,1]\to\Ss$ such that $\alpha(0)=x$ and $\alpha(1)=y$, as required.
\end{proof}

\begin{lem}[Double polynomial curve selection lemma]\label{doublecurve}
Let $\Ss\subset\R^n$ be an open semialgebraic set and let $p\in\cl(\Ss)$. Then there exists a polynomial arc $\alpha:[-1,1]\to\R^n$ such that $\alpha(0)=p$, $\alpha([-1,1]\setminus\{0\})\subset\Ss$ and $\alpha([-1,0))\cap\alpha((0,1])=\varnothing$.
\end{lem}
\begin{proof}
By \cite[Prop.8.1.13]{bcr} there exists a Nash arc $\eta:=(\eta_1,\ldots,\eta_n):[-1,1]\to\R^n$ such that $\eta(0)=p$ and $\eta((0,1])\subset\Ss$. After shrinking the domain of $\beta$ we may assume that each $\eta_i\in\R[[\t]]_{\rm alg}$ is an algebraic analytic series. Let $f_1,\ldots,f_r\in\R[\x]$ be polynomials such that 
$$
\eta((0,\veps])\subset\{f_1>0,\ldots,f_r>0\}\subset\Ss
$$
for some $0<\veps<1$. The algebraic series $f_j(\eta)\in\R[[\t]]_{\rm alg}$ satisfies $f_j(\eta)=a_j\t^{k_j}+\cdots$ for some $a_j>0$ and $k_j\geq1$. Define $\ell:=\max\{ k_j:\ j=1,\ldots,r\}+1$. Let $\x:=(\x_1,\ldots,\x_n)$, $\y:=(\y_1,\ldots,\y_n)$ and let $\z$ be a single variable. Write $f_j(\x+\z\y)=f_j(\x)+\z h_j(\x,\y,\z)$ where $h_j\in\R[\x,\y,\z]$. Let $\zeta_j\in\R[[\t^*]]_{\rm alg}$ be an algebraic series such that $\xi_j:=\eta_j+\t^\ell\zeta_j\in\R[\t]$ is a univariate polynomial. Denote $\xi:=(\xi_1,\ldots,\xi_n)$ and $\zeta:=(\zeta_1,\ldots,\zeta_n)$. Then
$$
f_j(\xi(\t^2))=f_j(\eta(\t^2)+\t^{2\ell}\zeta(\t^2))=f_j(\eta(\t^2))+\t^{2\ell}h_j(\eta(\t^2),\zeta(\t^2),\t^{2\ell})=a_j\t^{2k_j}+\cdots,
$$
which is strictly positive for non-zero small values of $\t$. Let $q>\ell$ be an odd positive integer strictly bigger than the degrees of all the polynomials $\xi_j(\t^2)$. Define $\gamma:=\xi(\t^2)+\t^q(1,0,\ldots,0)\in\R[\t]^n$. As the exponent $q$ is odd and all the exponents of the non-zero monomials of the polynomials $\xi_j(\t^2)$ are even, $\gamma([-\veps,0))\cap\gamma((0,\veps])=\varnothing$ if $\veps>0$ is small enough. Observe that
$$
f_j(\gamma)=f_j(\xi(\t^2)+\t^q(1,0,\ldots,0))=f_j(\xi(\t^2))+\t^qh_j(\eta(\t^2),\zeta(\t^2),\t^q)=a_j\t^{2k_j}+\cdots>0,
$$
so for $\veps>0$ small enough $\gamma:[-\veps,\veps]\to\R^n$ is a polynomial arc such that $\gamma([-\veps,\veps]\setminus\{0\})\subset\{f_1>0,\ldots,f_r>0\}\subset\Ss$ and $\gamma(0)=0=p$. Thus, after a linear reparameterization in order to have $[-1,1]$ as the domain of $\gamma$, we deduce $\gamma$ is the searched polynomial path.
\end{proof}

There are many ways to smooth corners of a continuous semialgebraic path. The strategy we propose here is far from being the one with less complexity, but it is quick to be presented as it only uses Hermite's interpolation (and a control of the images of the paths via Nash diffeomorphisms).

\begin{lem}[Smoothing a corner]\label{smco}
Let $a<\tau<t_0<\theta<b$ and let $\veps>0$. Pick a point $x_0\in\R^n$ and let $\beta:[a,b]\to\Bb_n(x_0,\veps)$ be a continuous semialgebraic path. Assume $\eta(\beta)=\{t_0\}$ and $\beta(t_0)=x_0$. For each $\nu\geq0$ there exists a $\Cont^\nu$-semialgebraic path $\gamma:[a,b]\to\Bb_n(x_0,\veps)$ such that $\gamma|_{[a,b]\setminus(\tau,\theta)}=\beta|_{[a,b]\setminus(\tau,\theta)}$.
\end{lem}
\begin{proof}
Consider the Nash diffeomorphism 
$$
\varphi:\Bb_n(x_0,\veps)\to\R^n,\ x\mapsto\frac{x-x_0}{\sqrt{\veps^2-\|x-x_0\|^2}}
$$
and the continuous semialgebraic path $\beta^*:=\varphi\circ\beta:[a,b]\to\R^n$. We have $\eta(\beta^*)=\{t_0\}$ and $\beta^*(t_0)=0$. By Hermite's interpolation there exists a polynomial map $\lambda:[a,b]\to\R^n$ such that $\lambda^{(k)}(\tau)=(\beta^*)^{(k)}(\tau)$ and $\lambda^{(k)}(\theta)=(\beta^*)^{(k)}(\theta)$ for $k=0,\ldots,\nu$. Consider the $\Cont^\nu$-semialgebraic path
$$
\gamma^*:[a,b]\to\R^n,\ t\mapsto
\begin{cases}
\lambda(t)&\text{if $t\in[\tau,\theta]$,}\\
\beta^*(t)&\text{if $t\in[a,b]\setminus(\tau,\theta)$.}
\end{cases}
$$
The composition $\gamma:=\varphi^{-1}\circ\gamma^*$ is a $\Cont^\nu$-semialgebraic path that satisfies the required conditions.
\end{proof}

\subsection{Proof of Lemma \ref{smart}}

We prove next Lemma \ref{smart} as an application of Lemma \ref{clue}(2).

\begin{proof}[Proof of Lemma \em\ref{smart}]
The proof is conducted in several setps: 

\noindent{\sc Step 1. \em Construction of the suitable continuous semialgebraic path $\beta$.}
Let $\beta_i:[-1,1]\to\Gamma_i\subset\Ss$ be a polynomial parameterization of the bridge $\Gamma_i$ such that $\beta_i(0)=q_i$, $\beta_i([-1,0))\subset\Ss_i\cup\{q_i\}$ and $\beta_i((0,1])\subset\Ss_{i+1}$. By Lemma \ref{doublecurve} there exists a polynomial arc $\alpha_j:[-1,1]\to\Ss_j\cup\{p_j\}$ such that $\alpha_j(0)=p_j$, $\alpha_j([-1,1]\setminus\{0\})\subset\Ss_j$ and $\alpha_j([-1,0))\cap\alpha_j((0,1])=\varnothing$. Denote $\Lambda_j:=\alpha_j([-1,1])$ and after shrinking the domains of $\beta_i$ and $\alpha_j$ we may assume that the collection of semialgebraic sets $\Gamma_i\setminus\{q_i\},\Lambda_j\setminus\{p_j\}$ is a pairwise disjoint family. We reparameterize linearly the domains of $\beta_i$ and $\alpha_j$ and shrink them (if necessary) in such a way that there exist values
\begin{multline*}
\tau_0:=s_0=0<t_1<\zeta_1<\xi_1<s_1<\theta_1<\tau_1<t_2<\zeta_2<\cdots\\
<\tau_{\ell-2}<t_{\ell-1}<\zeta_{\ell-1}<\xi_{\ell-1}<s_{\ell-1}<\theta_{\ell-1}<\tau_{\ell-1}<t_\ell<1=s_\ell=:\zeta_\ell
\end{multline*}
such that:
\begin{itemize}
\item[$\bullet$] $\alpha_i:[\tau_{i-1},\zeta_i]\to\Lambda_i\subset\Ss_i\cup\{p_i\}$ and $\alpha_i(t_i)=p_i$.
\item[$\bullet$] $\beta_i:[\xi_i,\theta_i]\to\Gamma_i$ and $\beta_i(s_i)=q_i$.
\end{itemize}

The points $\alpha_i(\tau_{i-1}),\alpha_i(\zeta_i),\beta_{i-1}(\theta_{i-1}),\beta_i(\xi_i)$ belong to $\Ss_i$, which is a connected open semialgebraic set, and they are pairwise different. By Lemma \ref{polcon} there exist: 
\begin{itemize}
\item a polynomial path $\gamma_i:[\theta_{i-1},\tau_{i-1}]\to\Ss_i$ such that $\gamma_i(\theta_{i-1})=\beta_{i-1}(\theta_{i-1})$ and $\gamma_i(\tau_{i-1})=\alpha_i(\tau_{i-1})$,
\item a polynomial path $\eta_i:[\zeta_i,\xi_i]\to\Ss_i$ such that $\eta_i(\zeta_i)=\alpha_i(\zeta_i)$ and $\eta_i(\xi_i)=\beta_i(\xi_i)$.
\end{itemize}
Denote $Z:=\{\tau_0,\ldots,\tau_{\ell-1},\zeta_1,\ldots,\zeta_\ell,\xi_1,\ldots,\xi_{\ell-1},\theta_1,\ldots,\theta_{\ell-1}\}$. Thus, concatenating all the previous polynomial paths and arcs we construct a piecewise polynomial path $\beta:[0,1]\to\R^n$ such that
\begin{itemize}
\item[(1)] $\beta([0,1])\subset\Ss\cup\{p_1,\ldots,p_\ell,q_1,\ldots,q_{\ell-1}\}$.
\item[(2)] $\beta(t_i)=p_i$ for $i=1,\ldots,\ell$.
\item[(3)] $\beta((t_i,s_i))\subset\Ss_i$, $\beta((s_i,t_{i+1}))\subset\Ss_{i+1}$ and $\beta(s_i)=q_i$.
\item[(4)] $\eta(\beta)\subset(0,1)\setminus\{t_1,\ldots,t_\ell,s_1,\ldots,s_{\ell-1}\}$ (because $\beta|_{[0,1]\setminus Z}$ is a Nash map) and $\beta(\eta(\beta))\subset\Ss$ (because $\eta(\beta)\subset Z$ and $\beta(Z)\subset\Ss$).
\end{itemize}

Thus, we have provided a procedure to construct a continuous semialgebraic path $\beta:[0,1]\to\Ss\cup\{p_1,\ldots,p_\ell,q_1,\ldots,q_{\ell-1}\}$ such that $\eta(\beta)\subset(0,1)\setminus\{t_1,\ldots,t_\ell,s_1,\ldots,s_{\ell-1}\}$, $\beta(\eta(\beta))\subset\Ss$ and satisfies conditions (ii) and (iii) in the statement.

In the following we fix a continuous semialgebraic path $\beta:[0,1]\to\Ss\cup\{p_1,\ldots,p_\ell,q_1,\ldots,q_{\ell-1}\}$ such that $\eta(\beta)\subset(0,1)\setminus\{t_1,\ldots,t_\ell,s_1,\ldots,s_{\ell-1}\}$, $\beta(\eta(\beta))\subset\Ss$ and satisfies conditions (ii) and (iii) in the statement. Note that conditions (1), (2), (3) and (4) hold for such a $\beta$.

\noindent{\sc Step 2.} {\em Computing the order of differentiability.} We need to compute a positive integer $\nu$ in order to apply Lemma \ref{clue}(2). As each $\Ss_i$ is a semialgebraic set and as a consequence of conditions (1), (2), (3) and (4) above, there exist polynomials $f_{ij},g_{ij},h_{ij},m_{ij}\in\R[\x]$ such that: 
\begin{itemize}
\item[$\bullet$] $\{f_{i1}>0,\ldots,f_{is}>0\}\subset\Ss_i$ is adherent to $p_i$ and $(f_{ij}\circ\beta)(t_i-\t)=a_{ij}\t^{e_{ij}}+\cdots$, where $a_{ij}>0$ and $e_{ij}$ is a positive integer.
\item[$\bullet$] $\{g_{i1}>0,\ldots,g_{is}>0\}\subset\Ss_i$ is adherent to $p_i$ and $(g_{ij}\circ\beta)(t_i+\t)=b_{ij}\t^{u_{ij}}+\cdots$, where $b_{ij}>0$ and $u_{ij}$ is a positive integer.
\item[$\bullet$] $\{h_{i1}>0,\ldots,h_{is}>0\}\subset\Ss_i$ is adherent to $q_i$ and $(h_{ij}\circ\beta)(s_i-\t)=c_{ij}\t^{v_{ij}}+\cdots$, where $c_{ij}>0$ and $v_{ij}$ is a positive integer.
\item[$\bullet$] $\{m_{i1}>0,\ldots,m_{is}>0\}\subset\Ss_{i+1}$ is adherent to $q_i$ and $(m_{ij}\circ\beta)(s_i+\t)=d_{ij}\t^{w_{ij}}+\cdots$, where $d_{ij}>0$ and $w_{ij}$ is a positive integer.
\end{itemize}
Define $\nu:=1+\max\{e_{ij},u_{ij},v_{ij},w_{ij}:\ 1\leq i\leq\ell,1\leq j\leq s\}$.

\noindent{\sc Step 3.} {\em `Smoothing' the corners of the (continuous) semialgebraic path $\beta$.} We smooth next the `corners' of the (continuous) semialgebraic path $\beta$, which are contained in the open semialgebraic set $\Ss$. To that end, we use Lemma \ref{smco}. In \cite{f3} we present an alternative construction in terms of Bernstein polynomials, which avoids the `smoothing' of corners and the corresponding increase in the complexity of the construction. 

The mentioned `corners' appear only at the points of the finite set $\eta(\beta):=\{\lambda_1,\ldots,\lambda_r\}\subset(0,1)\setminus\{t_1,\ldots,t_\ell,s_1,\ldots,s_{\ell-1}\}$. Pick $\delta, \varepsilon>0$ small enough so that $[\lambda_k-\delta,\lambda_k+\delta]\subset(0,1)\setminus\{t_1,\ldots,t_\ell,s_1,\ldots,s_{\ell-1}\}$ and 
$$
\beta([\lambda_k-\delta,\lambda_k+\delta])\subset\Bb(\beta(\lambda_k),\tfrac{\veps}{4})\subset\Bb(\beta(\lambda_k),\veps)\subset\Ss
$$
for $k=1,\ldots,r$. Define $T:=\bigcup_{k=1}^r[\lambda_k-\delta,\lambda_k+\delta]$. By Lemma \ref{smco} there exists a $\Cont^\nu$-semialgebraic path $\gamma:[0,1]\to\Ss\cup\{p_1,\ldots,p_\ell,q_1,\ldots,q_{\ell-1}\}$ such that $\gamma|_{[0,1]\setminus T}=\beta|_{[0,1]\setminus T}$ and $\gamma([\lambda_k-\delta,\lambda_k+\delta])\subset\Bb(\beta(\lambda_k),\tfrac{\veps}{4})$ for $k=1,\ldots,r$. Consequently,
$$
\|\beta(t)-\gamma(t)\|\leq
\|\beta(t)-\beta(\lambda_k)\|+\|\beta(\lambda_k)-\gamma(t)\|<\frac{\veps}{2}
$$
if $t\in[\lambda_k-\delta,\lambda_k+\delta]$ for $k=1,\ldots,r$. Thus, $\|\beta-\gamma\|<\frac{\veps}{2}$. Observe that $\gamma$ satisfies conditions (ii) and (iii) in the statement. 

\noindent{\sc Conclusion. \em Final approximation.}
By Lemma \ref{clue}(2) there exists a polynomial map $\alpha:\R\to\R^n$ such that $\|\alpha-\gamma\|<\frac{\veps}{2}$ and satisfies conditions (i), (ii) and (iii) in the statement. We have in addition $\|\alpha-\beta\|\leq\|\alpha-\gamma\|+\|\gamma-\beta\|<\veps$, as required.
\end{proof}

\section{Finite unions of convex polyhedra}\label{s4}

A semialgebraic set $\Ss\subset\R^n$ is a \em PL semialgebraic set \em if there exist finitely many polynomials $f_{ij}\in\R[\x]$ of degree $\leq1$ such that $\Ss=\bigcup_{i=1}^s\{f_{i1}\geq0,\ldots,f_{ir}\geq0\}$. Observe that each {\em basic PL semialgebraic set} $\{f_{i1}\geq0,\ldots,f_{ir}\geq0\}$ is a convex polyhedron, so PL semialgebraic sets are finite unions of convex polyhedra $\pol_i$.

The closure of the difference of two convex polyhedra is a finite union of convex polyhedra. Thus, each PL semialgebraic set $\Ss$ can be written as a finite union of convex polyhedra $\Ss_i$ such that their relative interiors are pairwise disjoint. The boundary $\partial\pol:=\cl(\pol)\setminus\Int(\pol)$ of a convex polyhedron $\pol$ is a finite union of convex polyhedra of smaller dimensions. Each triangulation of $\partial\pol$ and each interior point $p$ of $\pol$ induces a standard triangulation of $\pol$. Using these facts, one can triangulate $\Ss$ by means of a standard induction process based on the dimension of $\Ss$. Thus, there exist finitely many simplices $\sigma_1,\ldots,\sigma_\ell$ such that $\Ss=\bigcup_{i=1}^\ell\sigma_i$ and $\sigma_i\cap\sigma_j$ is either the empty-set or a common face of $\sigma_i$ and $\sigma_j$. Consequently, PL semialgebraic sets coincide with the realizations of finite simplicial complexes. We will use in each case the description of PL semialgebraic sets that fits better each situation.

Let $\Ss\subset\R^p$ be an $n$-dimensional PL semialgebraic set that is connected by analytic paths. By \cite[Lem.7.1]{f2} $\Ss$ is pure dimensional and by \cite[Cor.7.10]{f2} its Zariski closure $\ol{\Ss}^{\zar}$ is irreducible. Let $\Ss_1,\ldots,\Ss_\ell$ be $n$-dimensional convex polyhedra such that $\Ss=\bigcup_{k=1}^\ell\Ss_k$. We have $\ol{\Ss}^{\zar}=\bigcup_{k=1}^\ell\ol{\Ss}^{\zar}_k$. The Zariski closure of an $n$-dimensional convex polyhedron is the affine $n$-dimensional subspace spanned by it. Thus, $\ol{\Ss}^{\zar}=\ol{\Ss}^{\zar}_k$ for each $k=1,\ldots,n$ and $\ol{\Ss}^{\zar}$ is an $n$-dimensional subspace, so we may assume $p=n$. This means that the statements of Theorems \ref{main1-i} and \ref{main1-i2} are not restrictive with respect to the embedding dimension.

Before proving Theorems \ref{main1-i} and \ref{main1-i2} in this section we need some preliminary results.

\subsection{Semialgebraic sets connected by analytic paths}

Our purpose is to establish a result (Lemma \ref{bridges}) that relates connectedness by analytic paths with the existence of suitable bridges. Before that we need the following result.

\begin{lem}\label{freedom}
Let $\Ss\subset\R^n$ be an $n$-dimensional semialgebraic set and let $\alpha:[-1,1]\to\Ss$ be an analytic arc such that $\alpha([-1,1]\setminus\{0\})\subset\Int(\Ss)$. For each $m\geq0$ denote by $\alpha_m$ the jet of degree $m$ of $\alpha$ at the origin. Then there exist $m\geq1$ and $\veps>0$ such that the polynomial map
$$
\varphi:\R\times\R^n\to\R^n,\ (t,x)\mapsto\alpha_m(t)+t^{m+1}x
$$
satisfies $\varphi(([-\veps,\veps]\setminus\{0\})\times\ol{\Bb}_n)\subset\Int(\Ss)$.
\end{lem}
\begin{proof}
Let $f_i,g_i\in\R[\x]$ be polynomials such that 
\begin{align*}
\alpha((0,\veps])&\subset\{f_1>0,\ldots,f_r>0\}\subset\Int(\Ss),\\
\alpha([-\veps,0))&\subset\{g_1>0,\ldots,g_r>0\}\subset\Int(\Ss)
\end{align*}
for some $0<\veps<1$. Consider the series $f_i(\alpha(\t)),g_i(\alpha(-\t))\in\R[[\t]]$, which satisfy
\begin{align*}
&f_i(\alpha(\t))=a_i\t^{k_i}+\cdots\\
&g_i(\alpha(-\t))=b_i\t^{\ell_i}+\cdots
\end{align*}
for some $a_i,b_i>0$. Let $m:=\max\{k_i,\ell_i:\ i=1,\ldots,r\}+1$. Let $\x:=(\x_1,\ldots,\x_n)$, $\y:=(\y_1,\ldots,\y_n)$ and let $\z$ be a single variable. Write
\begin{align*}
&f_i(\x+\z\y)=f_i(\x)+\z h_i(\x,\y,\z),\\
&g_i(\x+\z\y)=g_i(\x)+\z h'_i(\x,\y,\z),
\end{align*}
where $h_i,h_i'\in\R[\x,\y,\z]$. Then
\begin{align*}
f_i(\alpha_m(\t)+\t^{m+1}\y)&=f_i(\alpha_m(\t))+\t^{m+1}h_i(\alpha_m(\t),\y,\t^{m+1})=a_i\t^{k_i}+\cdots,\\
g_i(\alpha_m(-\t)+(-\t)^{m+1}\y)&=g_i(\alpha_m(-\t))+(-\t)^{m+1}h_i'(\alpha_m(-\t),\y,(-\t)^{m+1})=b_i\t^{\ell_i}+\cdots.
\end{align*}
Consider the continuous semialgebraic functions
\begin{align*}
H_i:[-\veps,\veps]\times\ol{\Bb}_n\to\R,\ (t,x)\mapsto h_i(\alpha_m(t),x,t^{m+1}),\\
H_i':[-\veps,\veps]\times\ol{\Bb}_n\to\R,\ (t,x)\mapsto h_i'(\alpha_m(t),x,t^{m+1})
\end{align*}
and let $M>0$ be such that $H_i([-\veps,\veps]\times\ol{\Bb}_n),H_i'([-\veps,\veps]\times\ol{\Bb}_n)\subset[-M,M]$ for $i=1,\ldots,r$. As 
\begin{align*}
f_i(\alpha_m(t)+t^{m+1}x)&\ge f_i(\alpha_m(t))-t^{m+1}|h_i(\alpha_m(t),x,t^{m+1})|\ge f_i(\alpha_m(t))-Mt^{m+1},\\
g_i(\alpha_m(-t)+(-t)^{m+1}x)&\ge g_i(\alpha_m(-t))-(-t)^{m+1}|h' _i(\alpha_m(-t),x,(-t)^{m+1})|\\
&\ge g_i(\alpha_m(-t))-M(-t)^{m+1},
\end{align*}
after shrinking $\veps>0$ if necessary, we may assume 
\begin{align*}
&f_i(\alpha_m(t)+t^{m+1}x)>0,\\
&g_i(\alpha_m(-t)+(-t)^{m+1}x)>0
\end{align*}
for each $(t,x)\in(0,\veps]\times\ol{\Bb}_n$ and each $i=1,\ldots,r$. We conclude $\varphi(([-\veps,\veps]\setminus\{0\})\times\ol{\Bb}_n)\subset\Int(\Ss)$, as required. 
\end{proof}

\begin{lem}\label{bridges}
Let $\Ss_1,\ldots,\Ss_\ell\subset\R^n$ be connected open semialgebraic sets. Assume there exists a semialgebraic set $\Tt\subset\R^n$ connected by analytic paths such that $\Ss:=\bigcup_{i=1}^\ell\Ss_i\subset\Tt\subset\cl(\Ss)$. Then we can reorder the indices $i$ in such a way that there exist bridges $\Gamma_i\subset\Tt$ between $\Ss_i$ and $\bigsqcup_{j=1}^{i-1}\Ss_j$ for $i=2,\ldots,\ell$.
\end{lem}
\begin{proof}
Denote $\Tt_i:=\cl(\Ss_i)\cap\Tt$. We prove the statement by induction on $\ell$. If $\ell=1$, there is nothing to prove. Suppose the result is true for $i-1<\ell$ and let us check that it is also true for $i$.

If $\Ss_j\cap\Ss_k\neq\varnothing$ for some $1\leq j\leq i-1$ and some $i\leq k\leq\ell$, we pick a point $p\in\Ss_j\cap\Ss_k$ and $\veps>0$ such that $\Bb(p,\veps)\subset\Ss_j\cap\Ss_k$. The polynomial arc $\alpha:[-1,1]\to\R^n,\ t\mapsto p+\veps(t,0,\ldots,0)$ is a bridge between $\Ss_j\subset\bigsqcup_{j=1}^{i-1}\Ss_j$ and $\Ss_k$. We interchange the indices $k$ and $i$.

Assume $\Ss_j\cap\Ss_k=\varnothing$ if $1\leq j\leq i-1$ and $i\leq k\leq\ell$. Define $\Tt':=\bigcup_{j=1}^{i-1}\Tt_j$ and $\Tt'':=\bigcup_{j=i}^\ell\Tt_j$. Let $\Ee$ be the Zariski closure of $\bigcup_{i=1}^\ell\partial\Ss_i$ where $\partial\Ss_i:=\cl(\Ss_i)\setminus\Ss_i$. The algebraic set $\Ee$ has dimension $\leq n-1$, so it contains none of the $\Ss_i$. Observe that $\Tt'\cap\Tt''\subset\Ee$. Let $x\in\Tt'\setminus\Ee$ and let $y\in\Tt''\setminus\Ee$. As $\Tt=\Tt'\cup\Tt''$ is connected by analytic paths, there exists an analytic path $\beta:[-1,1]\to\Tt$ such that $\beta(-1)=x$ and $\beta(1)=y$. By the identity principle for analytic maps in a single variable $\Ff:=\beta^{-1}(\Ee)$ is a finite set (because otherwise $x,y\in\beta([-1,1])\subset\Ee$). As $x\in\Tt'\setminus\Ee$, $y\in\Tt''\setminus\Ee$ and $\Tt'\cap\Tt''\subset\Ee$, there exists $t_0\in[-1,1]$ such that $\beta(t)\in\Tt'$ if $t\leq t_0$ and $\beta(t)\in\Tt''\setminus\Ee$ if $t_0<t<t_0+\veps$ for some $\veps>0$. As $\Ff$ is a finite set, we may assume shrinking $\veps$ if necessary that $\beta((t_0-\veps,t_0))\subset\Ss_j$ for some $j=1,\ldots,i-1$ and $\beta((t_0,t_0+\veps))\subset\Ss_k$ for some $k=i,\ldots,\ell$. We interchange indices $k$ and $i$ and after a translation we assume $t_0=0$. By Lemma \ref{freedom} we can construct from the analytic arc $\beta$ a bridge between $\Ss_j$ and $\Ss_i$, as required.
\end{proof}

\begin{cor}\label{graph}
Let $\Ss_1,\ldots,\Ss_\ell\subset\R^n$ be connected open semialgebraic sets. Assume there exists a semialgebraic set $\Tt\subset\R^n$ connected by analytic paths such that $\Ss:=\bigcup_{i=1}^\ell\Ss_i\subset\Tt\subset\cl(\Ss)$. Then there exists a sequence of semialgebraic sets $\Rr_1,\ldots,\Rr_r$ such that $\{\Ss_1,\ldots,\Ss_\ell\}=\{\Rr_1,\ldots,\Rr_r\}$ and for each index $i=1,\ldots,r-1$ there exists a bridge between $\Rr_i$ and $\Rr_{i+1}$. 
\end{cor}
\begin{proof}
Consider the graph $\Lambda$ whose vertices are the connected open semialgebraic sets $\Ss_i$ and such that there exists an edge between a pair of vertices $\Ss_i$ and $\Ss_j$ if and only if there exists a bridge $\Gamma$ between the connected open semialgebraic sets $\Ss_i$ and $\Ss_j$. By Lemma \ref{bridges} the graph $\Lambda$ is connected. Thus, there exists a path inside $\Lambda$ through all its vertices. Denote by $\Rr_1,\ldots,\Rr_r$ a sequence of the vertices $\Ss_1,\ldots,\Ss_\ell$ (including repetitions if needed) such that all the vertices appear at least once and there exists an edge (that is, a bridge $\Gamma_k$ inside $\Ss$) between $\Rr_i$ and $\Rr_{i+1}$ for $i=1,\ldots,r-1$, as required.
\end{proof}

\begin{center}
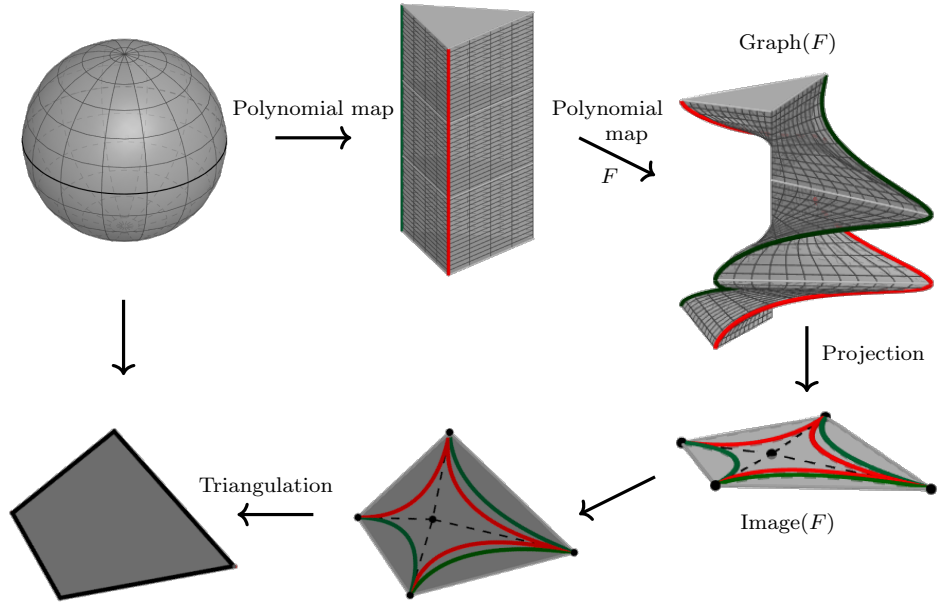
\begin{figure}[ht]
\centerline{
 \resizebox{13cm}{!}{
 \begin{tikzpicture}[
 scale=0.5,
 level/.style={thick},
 virtual/.style={thick,densely dashed},
 trans/.style={thick,<->,shorten >=2pt,shorten <=2pt,>=stealth},
 classical/.style={thin,double,<->,shorten >=4pt,shorten <=4pt,>=stealth}
 ]

 \path[->] (0cm,2cm) node{\includegraphics[width=3cm]{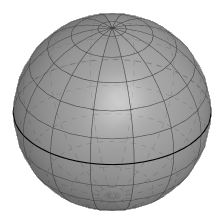}} -- (9cm,2cm) node{\includegraphics[width=2.3cm]{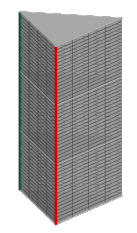}}--(18cm,-2cm) node{\includegraphics[width=4cm]{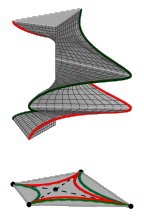}}--(9cm,-8cm) node{\includegraphics[width=3cm]{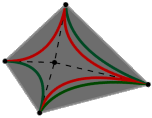}}--(0cm,-8cm) node{\includegraphics[width=3cm]{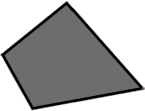}};
 \draw[->,very thick] (4,2)--(6,2);
 \draw[->,very thick] (12,2)--(14,1);
 \draw[->,very thick] (5,-8)--(3,-8);
 \draw[->,very thick] (14,-7)--(12,-8);
 \draw[->,very thick] (0,-2.3)--(0,-4.3);
 \draw[->,very thick] (18,-3)--(18,-4.75);

\draw (5,2.75) node{{\scriptsize Polynomial map}}; 
\draw (12.75,2.75) node{{\scriptsize Polynomial}};
\draw (13.25,2) node{{\scriptsize map}};
\draw (12.8,1) node{{\scriptsize $F$}};
\draw (17.5,4.5) node{{\scriptsize Graph($F$)}};
\draw (19.75,-3.75) node{{\scriptsize Projection}};
\draw (17.5,-8.25) node{{\scriptsize Image($F$)}}; 
\draw (3.75,-7.25) node{{\scriptsize Triangulation}};
 \end{tikzpicture}
 }
}
\caption{Sketch of proof of Theorem \ref{main1-i}\label{figa}}
\end{figure}
\end{center}

\subsection{Proof of Theorem \ref{main1-i}}

We present here the first main result of this article (Figure \ref{figa}).

\begin{proof}[Proof of Theorem \em\ref{main1-i}]
(ii) $\Longrightarrow$ (i) By \cite[Thm.1.5]{f2} there exists a Nash map $g:\R^{n+1}\to\R^{n+1}$ such that $g(\R^{n+1})=\ol{\Bb}_{n+1}$. Consider the Nash map $f\circ g:\R^{n+1}\to\R^n$ whose image is $\Ss$. By \cite[Main Thm.1.4]{f2} the semialgebraic set $\Ss$ is connected by analytic paths.

(i) $\Longrightarrow$ (ii) As $\Ss$ is the realization of a finite simplicial complex, we may assume by Lemma \ref{graph} that $\Ss=\bigcup_{k=1}^\ell\sigma_k$ where each $\sigma_k$ is an $n$-dimensional simplex and for each $k=1,\ldots,\ell-1$ there exists a bridge $\Gamma_k$ between $\Int(\sigma_k)$ and $\Int(\sigma_{k+1})$.

Denote by $q_k$ the base point of $\Gamma_k$. Write $t_k:=\frac{k-1}{\ell-1}$ and $s_k:=\frac{2k-1}{2(\ell-1)}$ (that is, the midpoint between $t_k$ and $t_{k+1}$). Denote by $p_{0k},\ldots,p_{nk}$ the vertices of the simplex $\sigma_k$ for $k=1,\ldots,\ell$. By Lemma \ref{smart} there exists for each $i=0,\ldots,n$ a polynomial path $\alpha_i:\R\to\R^n$ such that
\begin{itemize}
\item[(i)] $\alpha_i([0,1])\subset\Ss$.
\item[(ii)] $\alpha_i(t_k)=p_{ik}$ for $k=1,\ldots,\ell$.
\item[(iii)] $\alpha_i((t_k,s_k))\subset\Int(\sigma_k)$, $\alpha_i((s_k,t_{k+1}))\subset\Int(\sigma_{k+1})$ and $\alpha_i(s_k)=q_k$ for $k=1,\ldots,\ell$.
\end{itemize}

Denote $\Delta_n:=\{\x_1\geq0,\ldots,\x_n\geq0,\sum_{k=1}^n\x_i\leq1\}$ and consider the polynomial map
$$
F:\R^n\times\R\to\R^n,\ (x,t):=(x_1,\ldots,x_n,t)\mapsto\Big(1-\sum_{i=1}^nx_i\Big)\alpha_0(t)+\sum_{i=1}^nx_i\alpha_i(t).
$$
Pick $(x,t)\in\Delta_n\times[0,1]$. As each simplex $\sigma_k$ and its relative interior $\Int(\sigma_k)$ are convex sets,
$$
F(x,t)=\Big(1-\sum_{i=1}^nx_i\Big)\alpha_0(t)+\sum_{i=1}^nx_i\alpha_i(t)\in\begin{cases}
\sigma_k\subset\Ss&\text{if $t=t_k$},\\
\{q_k\}\subset\Ss&\text{if $t=s_k$},\\
\Int(\sigma_k)\subset\Ss&\text{if $t\in(t_k,s_k)$},\\
\Int(\sigma_{k+1})\subset\Ss&\text{if $t\in(s_k,t_{k+1})$},
\end{cases}
$$
so $F(\Delta_n\times[0,1])\subset\Ss$. As
$$
F(\Delta_n\times\{t_k\})=\Big\{\Big(1-\sum_{i=1}^nx_i\Big)p_{0k}+\sum_{i=1}^nx_ip_{ik}:\ x\in\Delta_n\Big\}=\sigma_k
$$ 
for $k=1,\ldots,\ell$, we deduce $\Ss=\bigcup_{k=1}^\ell\sigma_k\subset F(\Delta_n\times[0,1])$. Consequently, $F(\Delta_n\times[0,1])=\Ss$ and by Corollary \ref{tetra1} $\Ss$ is a polynomial image of $\ol{\Bb}_{n+1}$, as required.
\end{proof}

\subsection{Proof of Theorem \ref{main1-i2}}\label{s6b}

Our next purpose is to prove Theorem \ref{main1-i2}. As it is quite technical, we have decided to break its proof into several parts and present some preliminary lemmas to lighten notations and clarify ideas. This may increase the complexity of the construction but make the presentation more intuitive and readable. As usual, we call \em facets \em the faces of dimension $n-1$ of a convex polyhedron of dimension $n$. Given a hyperplane $H:=\{h=0\}$ of $\R^n$ we represent the subspaces determined by $H$ as $H^+:=\{h\geq0\}$ and $H^-:=\{h\leq0\}$. Denote 
$$
\Delta_{n-1}:=\Big\{(\lambda_1,\ldots,\lambda_n)\in\R^n:\ \lambda_1\geq0,\ldots,\lambda_n\geq0,\ \sum_{k=1}^n\lambda_k=1\Big\},
$$
which is an $(n-1)$-dimensional simplex. Write $\Int_r(\cdot)$ to refer to the relative interior of a convex polyhedron. Assume $n\geq2$ for the following three preliminary lemmas. The first one is illustrated in Figure \ref{simplexfig}.

\begin{lem}\label{simplex}
Let $\pol\subset\R^n$ be an $n$-dimensional convex polyhedron. Let $\sigma$ be an $(n-1)$-dimen\-sional simplex contained in one of the facets of $\pol$ and denote by $v_1,\ldots,v_n$ its vertices. Pick $p\in\Int(\pol)$ and consider the simplex $\widehat{\sigma}$ of vertices $\{p,v_1,\ldots,v_n\}$. Let $h_1,\ldots,h_n\in\R[\x]$ be polynomials of degree $1$ such that the hyperplane $H_k:=\{h_k=0\}$ contains the facet of $\widehat{\sigma}$ that contains $p$, but does not contain the vertex $v_k$. Assume $\widehat{\sigma}\subset H_k^+$ for each $k$. Consider the convex polyhedra $\pol_k:=\pol\cap\bigcap_{j\neq k}H_j^-$ and let $\alpha_k:[0,1]\to\pol_k\subset\pol$ be (continuous) semialgebraic paths such that $\alpha_k(0)=v_k$ and $\alpha_k(1)=p$. Consider the (continuous) semialgebraic map $F:\Delta_{n-1}\times[0,1]\to\R^n,\ (\lambda,t)\mapsto\sum_{k=1}^n\lambda_k\alpha_k(t)$. Then $\widehat{\sigma}\subset F(\Delta_{n-1}\times[0,1])\subset\pol$.
\end{lem}

\begin{center}
\begin{figure}[ht]
\begin{tikzpicture}[scale=0.85]

\draw[fill=gray!50,opacity=0.5,draw=none] (2,0) -- (4.5,0) -- (6,2) -- (5,4) -- (2.25,4.25) -- (0,2) -- (2,0);
\draw[fill=gray!90,opacity=0.5,draw=none] (2,0) -- (4.5,0) -- (3,2) -- (2,0);

\draw[fill=blue!25,opacity=0.5,draw=none] (2,0) -- (3,2) -- (4.05,4.1) -- (2.25,4.25) -- (0,2) -- (2,0);
\draw[fill=red!25,opacity=0.5,draw=none] (4.5,0) -- (6,2) -- (5,4) -- (2.25,4.25) -- (1.7,3.7) -- (3,2) -- (4.5,0);

\draw[color=blue,line width=1.5pt] (2,0) .. controls (1.5,1) and (2.5,1.75) .. (3,2);
\draw[color=red,line width=1.5pt] (4.5,0) .. controls (4.5,1) and (3.5,1.75) .. (3,2);

\draw[line width=1.5pt] (2,0) -- (4.5,0) -- (6,2) -- (5,4) -- (2.25,4.25) -- (0,2) -- (2,0);

\draw[dashed,line width=1.5pt] (2,0) -- (3,2);
\draw[dashed,line width=1.5pt] (4.5,0) -- (3,2);
\draw[dashed,line width=1.5pt] (6,2) -- (3,2);
\draw[dashed,line width=1.5pt] (5,4) -- (3,2);
\draw[dashed,line width=1.5pt] (2.25,4.25) -- (3,2);
\draw[dashed,line width=1.5pt] (0,2) -- (3,2);

\draw (3,2) node{$\bullet$};
\draw (2,-0.375) node{{\small $v_1$}};
\draw (4.5,-0.375) node{{\small $v_2$}};
\draw (3.25,-0.375) node{{\small $\sigma$}};
\draw (3.25,0.75) node{{\small $\widehat{\sigma}$}};
\draw (3.1,2.4) node{{\small $p$}};
\draw (1,4) node{{\small $\pol$}};
\draw (1.5,2.4) node{{\small $\pol_1$}};
\draw (4.5,2.4) node{{\small $\pol_2$}};
\draw (1.75,1.25) node{{\small $\alpha_1$}};
\draw (4.35,1.25) node{{\small $\alpha_2$}};

\end{tikzpicture}
\caption{Sketch of proof of Lemma \ref{simplex}\label{simplexfig}}
\end{figure}
\end{center}

\begin{proof}
As $\pol$ is convex, $F(\Delta_{n-1}\times[0,1])\subset\pol$. Let us prove: $\widehat{\sigma}\subset F(\Delta_{n-1}\times[0,1])$. As $\Delta_{n-1}\times[0,1]$ is compact and $\widehat{\sigma}=\cl(\Int(\widehat{\sigma}))$, it is enough to check: $\Int(\widehat{\sigma})\subset F(\Delta_{n-1}\times[0,1])$.

Let $H_0:=\{h_0=0\}$ be the hyperplane generated by $\sigma$. We claim: $F(\partial(\Delta_{n-1}\times[0,1]))\cap\Int(\widehat{\sigma})=\varnothing$ and $F(\partial\Delta_{n-1}\times(0,1))\cap\Int_r(\sigma)=\varnothing$.

Recall that $\Int(\widehat{\sigma})\cup\Int_r(\sigma)=\{h_0\geq0,h_1>0,\ldots,h_n>0\}$. We have $F(\Delta_{n-1}\times\{0\})=\sigma$ and $F(\Delta_{n-1}\times\{1\})=\{p\}$. Let $\tau_j$ be the facet of $\Delta_{n-1}$ that does not contain the vertex ${\tt e_j}:=(0,\ldots,0,\overset{(j)}{1},0,\ldots,0)$. If $k\neq j$, then $\pol_k\subset\bigcap_{i\neq k}H_i^-\subset H_j^-$. If $t\in(0,1)$ and $\lambda:=(\lambda_1,\ldots,\lambda_n)\in\tau_j$, then $\lambda_j=0$ and
$$
F(\lambda,t)=\sum_{k\neq j}\lambda_k\alpha_k(t)\in H_j^-=\{h_j\leq0\}.
$$
Thus, $F(\tau_j\times(0,1))\subset H_j^-$, so $F(\tau_j\times(0,1))\cap(\Int(\widehat{\sigma})\cup\Int_r(\sigma))\subset\{h_j>0,h_j\leq0\}=\varnothing$ and the claim follows.

Suppose there exists $z\in\Int(\widehat{\sigma})\setminus F(\Delta_{n-1}\times[0,1])$. Let us construct a (continuous) semialgebraic retraction $\rho:\pol\setminus\{z\}\to\partial\widehat{\sigma}$ such that $\rho^{-1}(\Int_r(\sigma))\subset\Int(\widehat{\sigma})\cup\Int_r(\sigma)$. 

For each $x\in\R^n\setminus\{z\}$ let $\ell_x$ be the ray $\{z+t\vec{zx}:\ t\in[0,+\infty)\}$, where $\vec{zx}$ denotes the vector of initial point $z$ and terminal point $x$. By \cite[11.1.2.3 \& 11.1.2.7]{ber1} $\ell_x\cap\partial\widehat{\sigma}=\{\rho(x)\}$ is a singleton and if $x\in\partial\widehat{\sigma}$, then $\rho(x)=x$. Define $\rho:\pol\setminus\{z\}\to\partial\widehat{\sigma},\ x\mapsto\rho(x)$. Note that $\rho(x)=z+\lambda\vec{zx}$, where $\lambda$ is the smallest value $\mu>0$ such that $h_i(z+\mu\vec{zx})=0$ for some $i=0,\ldots,n$. As $z\in\Int(\widehat{\sigma})$, we have $h_i(z)>0$ for $i=0,\ldots,n$. Thus, 
$$
\frac{1}{\lambda}=\max\Big\{\frac{h_i(z)-h_i(x)}{h_i(z)}:\ i=0,\ldots,n\Big\}>0.
$$
Consequently,
$$
\rho(x)=z+\frac{1}{\max\{\frac{h_i(z)-h_i(x)}{h_i(z)}:\ i=0,\ldots,n\}}\vec{zx},
$$
so $\rho:\pol\setminus\{z\}\to\partial\widehat{\sigma}$ is a continuous map such that $\rho|_{\partial\widehat{\sigma}}=\id_{\partial\widehat{\sigma}}$, that is, $\rho$ is a retraction. Observe that
$$
\max\Big\{\frac{h_i(z)-h_i(x)}{h_i(z)}:\ i=0,\ldots,n\Big\}=\frac{h_j(z)-h_j(x)}{h_j(z)}
$$
for some $j=0,\ldots,n$ if and only if $\rho(x)\in\{h_j=0\}$. In addition, if $x\in\pol\setminus(\Int(\widehat{\sigma})\cup\Int_r(\sigma))$, then $h_0(x)\geq0$ and $h_{i_0}(x)\leq 0$ for some $i_0=1,\ldots,n$. Thus, 
$$
\frac{h_0(z)-h_0(x)}{h_0(z)}\leq1\leq\frac{h_{i_0}(z)-h_{i_0}(x)}{h_{i_0}(z)}\leq\max\Big\{\frac{h_i(z)-h_i(x)}{h_i(z)}:\ i=0,\ldots,n\Big\},
$$
so $\rho(x)\not\in\{h_0=0,h_1>0,\ldots,h_n>0\}=\Int_r(\sigma)$. Consequently, $\rho^{-1}(\Int_r(\sigma))\subset\Int(\widehat{\sigma})\cup\Int_r(\sigma)$.

Consider the continuous semialgebraic map $F^*:=\rho\circ F:\Delta_{n-1}\times[0,1]\to\partial\widehat{\sigma}$. Let us prove: \em the restriction map $F^*|_{\partial(\Delta_{n-1}\times[0,1])}:\partial(\Delta_{n-1}\times[0,1])\to\partial\widehat{\sigma}$ has degree $1$ \em (as a continuous map between spheres of dimension $n-1$).

Pick a point $x\in\Int_r(\sigma)$. Then $(F^*)^{-1}(x)=F^{-1}(\rho^{-1}(x))\subset F^{-1}(\Int(\widehat{\sigma}))\cup F^{-1}(\Int_r(\sigma))$. As $F(\partial(\Delta_{n-1}\times[0,1]))\cap\Int(\widehat{\sigma})=\varnothing$, $F(\partial\Delta_{n-1}\times(0,1))\cap\Int_r(\sigma)=\varnothing$ and $F(\Delta_{n-1}\times\{1\})=\{p\}$, we have 
\begin{align*}
&F^{-1}(\Int(\widehat{\sigma}))\cap\partial(\Delta_{n-1}\times[0,1])=\varnothing,\\
&F^{-1}(\Int_r(\sigma))\cap\partial(\Delta_{n-1}\times[0,1])\subset\Delta_{n-1}\times\{0\}.
\end{align*}
Consequently, the preimage 
$$
(F^*)^{-1}(x)\cap\partial(\Delta_{n-1}\times[0,1])=(F^*)^{-1}(x)\cap(\Delta_{n-1}\times\{0\}).
$$ 
As $\alpha_k(0)=v_k$ for each $k$ and $\rho|_\sigma=\id_\sigma$, we deduce $F^*|_{\Delta_{n-1}\times\{0\}}=F|_{\Delta_{n-1}\times\{0\}}:\Delta_{n-1}\times\{0\}\to\sigma$ is a homeomorphism, so $(F^*)^{-1}(x)\cap\partial(\Delta_{n-1}\times[0,1])=(F|_{\Delta_{n-1}\times\{0\}})^{-1}(x)$ is a singleton and the restriction map $F^*|_{\partial(\Delta_{n-1}\times[0,1])}$ has degree $1$.

As $F^*|_{\partial(\Delta_{n-1}\times[0,1])}$ extends continuously to $\Delta_{n-1}\times[0,1]$, we deduce by \cite[Thm.5.1.6(b)]{hir3} that $F^*|_{\partial(\Delta_{n-1}\times[0,1])}$ has degree $0$, which is a contradiction.

Consequently, $\Int(\widehat{\sigma})\subset F(\Delta_{n-1}\times[0,1])$, as required.
\end{proof}

\begin{lem}\label{point}
Let $\pol\subset\R^n$ be an $n$-dimensional convex polyhedron and let $p\in\pol$. Let $v,w\in\R^n$ be linearly independent vectors such that $\Int(\pol)\cap\{p+tv:\ t>0\}\neq\varnothing$. Then, there exists $\veps>0$ such that polynomial arc $\alpha:[-\veps,\veps]\to\pol,\ t\mapsto p+t^2v+t^3w$ satisfies $\alpha([-\veps,\veps]\setminus\{0\})\subset\Int(\pol)$, $\alpha(0)=p$ and $\alpha([-\veps,0))\cap\alpha((0,\veps])=\varnothing$.
\end{lem}
\begin{proof}
Let $h_1,\ldots,h_s$ be equations of degree $1$ of the facets of $\pol$ and denote by $\vec{h}_i$ the linear part of $h_i$. Assume $\pol=\{h_1\geq0,\ldots,h_s\geq0\}$. The polynomial map
$$
\alpha:\R\to\pol,\ t\mapsto p+t^2v+t^3w
$$
satisfies $\alpha((-\infty,0))\cap\alpha((0,+\infty))=\varnothing$, because $v,w$ are linearly independent vectors. We have 
$$
h_i(p+t^2v+t^3w)=h_i(p)+t^2\vec{h}_i(v)+t^3\vec{h}_i(w)
$$
for each $i$. As $p\in\pol$, it holds $h_i(p)\geq0$ for each $i$. As
$$
h_i(p+tv)=h_i(p)+t\vec{h}_i(v)
$$
and $\Int(\pol)\cap\{p+tv:\ t>0\}\neq\varnothing$, we have $\vec{h}_i(v)>0$ if $h_i(p)=0$. Pick $\veps>0$ such that $h_i(p+t^2v+t^3w)>0$ if $0<|t|\leq\veps$ and $i=1,\ldots,s$. Thus, the polynomial arc $\alpha:[-\veps,\veps]\to\pol$ satisfies the required properties if $\veps>0$ is small enough.
\end{proof}

\begin{lem}\label{polyhedron}
Let $\pol\subset\R^n$ be an $n$-dimensional convex polyhedron. There exist polynomial paths $\alpha_1,\ldots,\alpha_n:[0,1]\to\pol$ and values $0:=t_0<t_1<\cdots<t_m=1$ and $s_j\in(t_j,t_{j+1})$ for $j=0,\ldots,m-1$ such that: 
\begin{itemize}
\item[(i)] $\alpha_k(0)=\alpha_k(1)=:p\in\Int(\pol)$ for $k=1,\ldots,n$.
\item[(ii)] If $\beta_1,\ldots,\beta_n:[0,1]\to\pol$ are polynomial paths satisfying $\beta_k$ is close to $\alpha_k$ in the $\Cont^\nu$-topology for $\nu$ large enough and the Taylor expansions of $\beta_k$ and $\alpha_k$ coincide at the values $t_j$ and $s_j$ up to order large enough for each pair $(k,j)$, then $\pol$ is the image of the polynomial map $F_\beta:\Delta_{n-1}\times[0,1]\to\pol,\ (\lambda,t)\mapsto\sum_{k=1}^n\lambda_k\beta_k(t)$.
\end{itemize}
\end{lem}

\begin{proof}
Let $\sigma_1,\ldots,\sigma_\ell$ be the simplices of dimension $n-1$ of a triangulation of $\partial\pol$ and let $p\in\Int(\pol)$. Denote by $v_{i1},\ldots,v_{in}$ the vertices of $\sigma_i$ and let $\widehat{\sigma}_i$ be the $n$-dimensional simplex of vertices $\{p,v_{i1},\ldots,v_{in}\}$. It holds $\pol=\bigcup_{i=1}^\ell\widehat{\sigma}_i$. Let $h_{i1},\ldots,h_{in}\in\R[\x]$ be polynomials of degree $1$ such that $H_{ik}:=\{h_{ik}=0\}$ is the hyperplane generated by the facet of $\widehat{\sigma}_i$ that contains $p$ but does not contain $v_{ik}$. Assume $\widehat{\sigma}_i\subset H_{ik}^+$ for each $k$ and consider the convex polyhedra $\pol_{ik}:=\pol\cap\bigcap_{j\neq k}H_{ij}^-$. As $p\in\Int(\pol)$ and the $H_{ik}$ are independent hyperplanes through $p$, we have $\Int(\pol_{ik})\neq\varnothing$. Observe that $p,v_{ik}\in\partial\pol_{ik}$. We will take advantage of Lemmas \ref{smart}, \ref{clue}(1) and \ref{simplex} to prove the statement. 

Fix values $0<t_1<\cdots<t_{2\ell+1}<1$ and $s_i\in(t_i,t_{i+1})$ for $i=1,\ldots,2\ell$. Fix $k=1,\ldots,n$ and $i=1,\ldots,\ell$. Let us construct: \em a bridge between $\Int(\pol)$ and $\Int(\pol_{ik})$ with base point $v_{ik}$ and a bridge between $\Int(\pol_{ik})$ and $\Int(\pol)$ with base point $p$.\em

\begin{center}
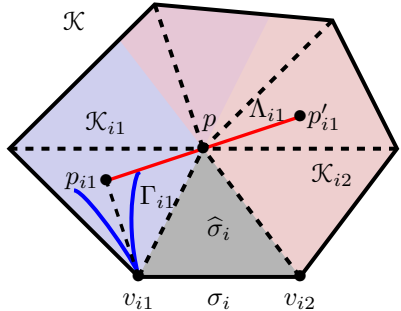
\begin{figure}[ht]
\begin{tikzpicture}[scale=0.85]

\draw[fill=gray!50,opacity=0.5,draw=none] (2,0) -- (4.5,0) -- (6,2) -- (5,4) -- (2.25,4.25) -- (0,2) -- (2,0);
\draw[fill=gray!90,opacity=0.5,draw=none] (2,0) -- (4.5,0) -- (3,2) -- (2,0);

\draw[fill=blue!25,opacity=0.5,draw=none] (2,0) -- (3,2) -- (4.05,4.1) -- (2.25,4.25) -- (0,2) -- (2,0);
\draw[fill=red!25,opacity=0.5,draw=none] (4.5,0) -- (6,2) -- (5,4) -- (2.25,4.25) -- (1.7,3.7) -- (3,2) -- (4.5,0);

\draw[color=blue,line width=1.5pt] (2,0) .. controls (1.75,1) and (2,1.75) .. (2,1.6);
\draw[color=blue,line width=1.5pt] (1,1.35) .. controls (1.2,1.35) and (2.1,0) .. (2,0);

\draw[line width=1.5pt] (2,0) -- (4.5,0) -- (6,2) -- (5,4) -- (2.25,4.25) -- (0,2) -- (2,0);

\draw[line width=1.25pt,dashed] (2,0) -- (1.5,1.5);
\draw[color=red,line width=1.25pt] (1.5,1.5) -- (3,2) -- (4.5,2.5);

\draw[dashed,line width=1.5pt] (2,0) -- (3,2);
\draw[dashed,line width=1.5pt] (4.5,0) -- (3,2);
\draw[dashed,line width=1.5pt] (6,2) -- (3,2);
\draw[dashed,line width=1.5pt] (5,4) -- (3,2);
\draw[dashed,line width=1.5pt] (2.25,4.25) -- (3,2);
\draw[dashed,line width=1.5pt] (0,2) -- (3,2);

\draw (3,2) node{$\bullet$};
\draw (1.5,1.5) node{$\bullet$};
\draw (4.5,2.5) node{$\bullet$};
\draw (2,0) node{$\bullet$};
\draw (4.5,0) node{$\bullet$};
\draw (2.3,1.25) node{\small $\Gamma_{i1}$};
\draw (4,2.6) node{\small $\Lambda_{i1}$};
\draw (1.125,1.5) node{\small $p_{i1}$};
\draw (4.875,2.5) node{\small $p_{i1}'$};
\draw (2,-0.375) node{{\small $v_{i1}$}};
\draw (4.5,-0.375) node{{\small $v_{i2}$}};
\draw (3.25,-0.375) node{{\small $\sigma_i$}};
\draw (3.25,0.75) node{{\small $\widehat{\sigma}_i$}};
\draw (3.1,2.4) node{{\small $p$}};
\draw (1,4) node{{\small $\pol$}};
\draw (1.5,2.4) node{{\small $\pol_{i1}$}};
\draw (5,1.6) node{{\small $\pol_{i2}$}};

\end{tikzpicture}
\caption{Construction of the bridges $\Gamma_{ik}$ and $\Lambda_{ik}$\label{bridges1}}
\end{figure}
\end{center}

Let $p_{ik}\in\Int(\pol_{ik})$ and let $\ell_{ik}$ be the ray with origin $v_{ik}$ that passes through $p_{ik}$. As $\Int(\pol_{ik})\subset\Int(\pol)$, there exists by Lemma \ref{point} a bridge $\Gamma_{ik}$ between $\Int(\pol)$ and $\Int(\pol_{ik})$ with base point $v_{ik}$ (see Figure \ref{bridges1}). Let $\ell_{ik}'$ be the ray with origin $p_{ik}$ that passes through $p$. By \cite[Lem.11.2.4]{ber1} the relative interior of the segment $\ell_{ik}'\cap\pol_{ik}$ is contained in $\Int(\pol_{ik})$. As $p\in\partial\pol_{ik}$, the difference $\ell_{ik}'\setminus[p_{ik},p]\subset\R^n\setminus\pol_{ik}$. As $p\in\Int(\pol)$, we can pick a point $p_{ik}'\in(\ell_{ik}'\setminus[p_{ik},p])\cap\Int(\pol)$. The segment $\Lambda_{ik}:=[p_{ik},p_{ik}']$ provides a bridge between $\Int(\pol_{ik})$ and $\Int(\pol)$ with base point $p$ (see Figure \ref{bridges1}).

Consider the family of connected open semialgebraic sets:
$$
\Ss_a:=\begin{cases}
\Int(\pol)&\text{if $a=2i-1$ for $i=1,\ldots,\ell+1$,}\\
\Int(\pol_{ik})&\text{if $a=2i$ for $i=1,\ldots,\ell$.}
\end{cases}
$$

Define the points
\begin{equation}
\begin{split}
&p_a:=\begin{cases}
p&\text{if $a=1$,}\\
p_{ik}&\text{if $a=2i$ for $i=1,\ldots,\ell$,}\\
p_{ik}'&\text{if $a=2i+1$ for $i=1,\ldots,\ell$.}
\end{cases}\\
&q_a:=\begin{cases}
v_{ik}&\text{if $a=2i-1$ for $i=1,\ldots,\ell$,}\\
p&\text{if $a=2i$ for $i=1,\ldots,\ell$.}
\end{cases}
\end{split}
\end{equation}

Observe that $p_a\in\Ss_a$ for $1\leq a\leq 2\ell+1$ and $q_a\in\cl(\Ss_a)\cap\cl(\Ss_{a+1})$ for $a=1,\ldots,2\ell$. As we have commented above, there exists a bridge between $\Ss_a$ and $\Ss_{a+1}$ with base point $q_a$ for $a=1,\ldots,2\ell$. By Lemma \ref{smart} there exist polynomial paths $\alpha_k:\R\to\R^n$ that satisfies:
\begin{itemize}
\item[(i)] $\alpha_k([0,1])\subset\Int(\pol)\cup\{v_{1k},\ldots,v_{\ell k}\}$.
\item[(ii)] $\alpha_k(t_a)=p_a$ for $a=1,\ldots,2\ell+1$.
\item[(iii)] $\alpha_k((t_a,s_a))\subset\Ss_a$, $\alpha_k((s_a,t_{a+1}))\subset\Ss_{a+1}$ and $\alpha_k(s_a)=q_a$. 
\end{itemize}

\begin{center}
\begin{figure}[ht]
\centerline{
 \resizebox{13cm}{!}{
 \begin{tikzpicture}[
 scale=0.5,
 level/.style={thick},
 virtual/.style={thick,densely dashed},
 trans/.style={thick,<->,shorten >=2pt,shorten <=2pt,>=stealth},
 classical/.style={thin,double,<->,shorten >=4pt,shorten <=4pt,>=stealth}
 ]

 \path[->] (0cm,2cm) node{\includegraphics[width=3cm]{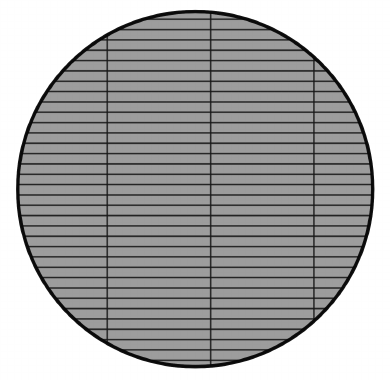}} -- (9cm,2cm) node{\includegraphics[width=2cm]{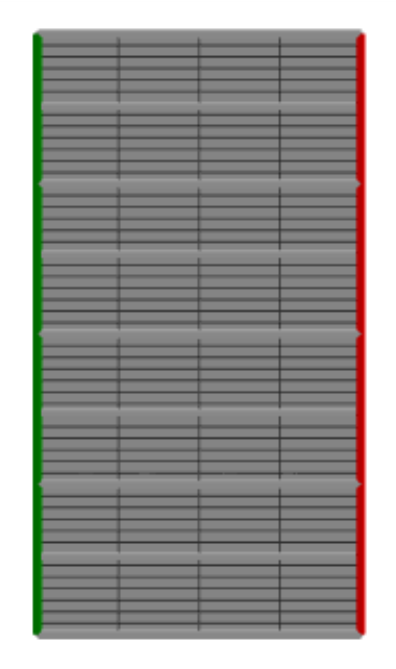}}--(18cm,-2cm) node{\includegraphics[width=3cm]{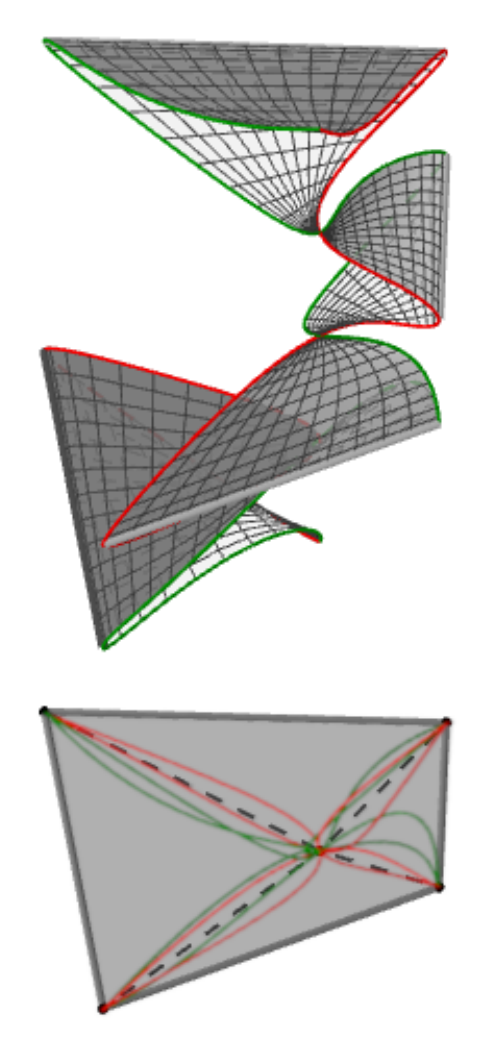}}--(9cm,-8cm) node{\includegraphics[width=3cm]{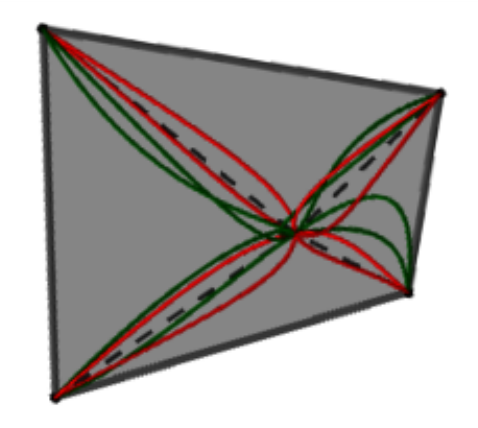}}--(0cm,-8cm) node{\includegraphics[width=3cm]{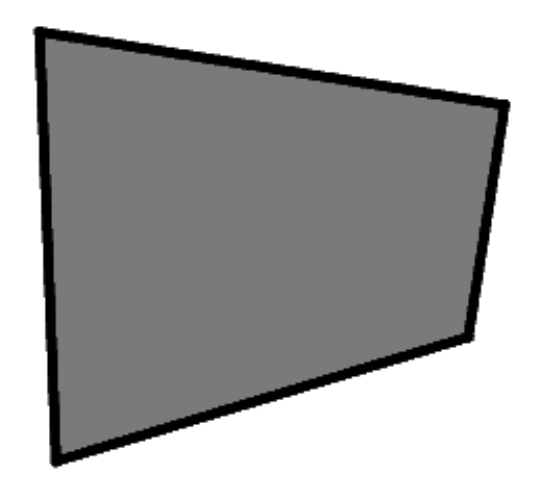}};
 \draw[->,very thick] (4,2)--(6,2);

 \draw[->,very thick] (12,2)--(14,1);
 \draw[->,very thick] (5.5,-8)--(3.5,-8);
 \draw[->,very thick] (14.75,-7)--(12.75,-8);
 \draw[->,very thick] (0,-2.3)--(0,-4.3);
 \draw[->,very thick] (18.5,-2.75)--(18.5,-4);

\draw (5,2.75) node{{\scriptsize Polynomial map}}; 
\draw (13.5,2.75) node{{\scriptsize Polynomial}};
\draw (13.5,2) node{{\scriptsize map}};
\draw (12.8,1) node{{\scriptsize $F_\beta$}};
\draw (19,4.5) node{{\scriptsize Graph($F_\beta$)}};
\draw (20.25,-3.25) node{{\scriptsize Projection}};
\draw (19,-8.25) node{{\scriptsize Image($F_\beta$)}}; 
\draw (4.5,-7.25) node{{\scriptsize Triangulation}};
\end{tikzpicture}
}
}
\caption{Sketch of proof of Lemma \ref{polyhedron}\label{figb}}
\end{figure}
\end{center}

By Lemma \ref{clue}(1) each polynomial path $\beta_k:\R\to\R^n$ close to $\alpha_k$ in the $\Cont^\nu$-topology for $\nu$ large enough and such that the Taylor expansions of $\beta_k$ and $\alpha_k$ coincide at the values $t_j$ and $s_j$ up to order large enough for each pair $(k,j)$ satisfies:
\begin{itemize}
\item[(i)] $\beta_k([0,1])\subset\Int(\pol)\cup\{v_{1k},\ldots,v_{\ell k}\}$.
\item[(ii)] $\beta_k(t_a)=p_a$ for $a=0,\ldots,2\ell$.
\item[(iii)] $\beta_k((t_a,s_a))\subset\Ss_a$, $\beta_k((s_a,t_{a+1}))\subset\Ss_{a+1}$ and $\beta_k(s_a)=q_a$. 
\end{itemize}

As $p_a\in\Ss_a$ for $a=1,\ldots,2a+1$, we have $\beta_k((s_a,s_{a+1}))\subset\Ss_{a+1}$ for $a=1,\ldots,2\ell$. Let us check: \em $\pol$ is the image of the polynomial map
$$
F_\beta:\Delta_{n-1}\times[0,1]\to\pol,\ (\lambda,t)\mapsto\sum_{k=1}^n\lambda_k\beta_k(t).
$$\em
 
As $\pol$ is convex and $\beta_k([0,1])\subset\pol$ for $k=1,\ldots,n$, we deduce $F_\beta(\Delta_{n-1}\times[0,1])\subset\pol$. In addition, $\beta_k(s_{2i-1})=v_{ik}$, $\beta_k(s_{2i})=p$ and $\beta_k((s_{2i-1},s_{2i}))\subset\Ss_{2i}=\Int(\pol_{ik})$ for each pair $(i,k)$. By Lemma \ref{simplex} we have
$$
\widehat{\sigma_i}\subset F_\beta(\Delta_{n-1}\times[s_{2i-1},s_{2i}])\subset\pol
$$
for $i=1,\ldots,\ell$, so $\pol=\bigcup_{i=1}^\ell\widehat{\sigma_i}\subset F_\beta(\Delta_{n-1}\times[0,1])$ and $F_\beta(\Delta_{n-1}\times[0,1])=\pol$, as required.
\end{proof}

We are ready to prove Theorem \ref{main1-i2}. The reader can compare Figures \ref{figa} and \ref{figb} (even if they are not completely faithful with the true constructions) to get an idea of the different complexity between the constructions provided by the proofs of Theorems \ref{main1-i} and \ref{main1-i2}.

\begin{proof}[Proof of Theorem \em\ref{main1-i2}]
(ii) $\Longrightarrow$ (i) By \cite[Thm.1.5]{f2} there exists a Nash map $g:\R^{n}\to\R^{n}$ such that $g(\R^{n})=\ol{\Bb}_{n}$. Consider the Nash map $f\circ g:\R^{n}\to\R^n$ whose image is $\Ss$. By \cite[Main Thm.1.4]{f2} the semialgebraic set $\Ss$ is connected by analytic paths.

(i) $\Longrightarrow$ (ii) If $n=1$, then $\Ss$ is a compact interval, which is affinely equivalent to $[-1,1]=\ol{\Bb}_1$. So we suppose $n\geq2$. By Lemma \ref{graph} we may assume $\Ss=\bigcup_{j=1}^r\pol_i\subset\R^n$, where each $\pol_j$ is an $n$-dimensional convex polyhedron whose interior is denoted by $\Ss_j$ and there exists a bridge $\Gamma_j$ inside $\Ss$ with base point $q_j$ between $\Ss_j$ and $\Ss_{j+1}$ for $j=1,\ldots,r-1$. 

For each $j=1,\ldots,r$ consider the polynomial paths $\alpha_{jk}:[0,1]\to\cl(\Ss_j)=\pol_j$ in the statement of Lemma \ref{polyhedron} and denote $p_j:=\alpha_{jk}(0)=\alpha_{jk}(1)$. By Lemma \ref{smart} there exists a polynomial path $\gamma_j:[-1,1]\to\Ss_j\cup\Ss_{j+1}\cup\{q_j\}$ such that $\gamma_j(-1)=p_j$, $\gamma_j(0)=q_j$, $\gamma_j(1)=p_{j+1}$, $\gamma_j([-1,0))\subset\Ss_j$ and $\gamma_j((0,1])\subset\Ss_{j+1}$ for $j=1,\ldots,r-1$. For each $k=1,\ldots,n$ consider the concatenated (continuous) semialgebraic path 
$$
\alpha_k:=\alpha_{1k}*\gamma_1*\alpha_{2k}*\cdots*\alpha_{r-1,k}*\gamma_{r-1}*\alpha_{rk}:[0,N]\to\R^n
$$
where $N:=3(r-1)+1$. Observe that $\eta(\alpha_k)\subset\{3j+1,3(j+1):\ j=0,\ldots,r-2\}$ and $\alpha_k(\eta(\alpha_k))\subset\{p_1,\ldots,p_r\}\subset\bigcup_{j=1}^r\Ss_j$ for $k=1,\ldots,n$. Consider the semialgebraic map
$$
F_\alpha:\Delta_{n-1}\times[0,N]\to\R^n,\ (\lambda,t)\mapsto\sum_{k=1}^n\lambda_k\alpha_k(t),
$$
which satisfies $F_\alpha(\Delta_{n-1}\times[0,N])=\Ss$. Let us modify $\alpha$ to achieve the statement.

By Lemma \ref{polyhedron} we find $\nu\geq0$ large enough and a finite set of values $0=:t_0<t_1<\cdots<t_m:=N$ with the following property: \em if $\beta_k:[0,N]\to\R^n$ are polynomial paths such that $\alpha_k$ and $\beta_k$ are close in the $\Cont^\nu$-topology of some of the intervals $[t_i,t_{i+1}]$ \em (determined by Lemma \ref{polyhedron}) \em and the Taylor expansions of $\alpha_k$ and $\beta_k$ coincide at some of the values $t_j$ \em (determined by Lemma \ref{polyhedron}) \em up to order large enough, then $\Ss=\bigcup_{j=1}^\ell\pol_j$ is contained in the image of the polynomial map
$$
F_\beta:\Delta_{n-1}\times[0,N]\to\R^n,\ (\lambda,t)\mapsto\sum_{k=1}^n\lambda_k\beta_k(t).
$$
\em

For each $i=0,\ldots,m-1$ we may assume in addition by Lemma \ref{clue}(1) that $\beta_k((t_i,t_{i+1}))\subset\Ss_{j_i}$ for some index $1\leq j_i\leq r-1$ that depends only on $i$ (and not on $k$). As each $\Ss_{j_i}$ is convex, we deduce 
$$
F_\beta(\Delta_{n-1}\times(t_i,t_{i+1}))\subset\Ss_{j_i}\subset\Ss 
$$
for each $i=0,\ldots,m-1$. As $\Ss$ is closed, $F_\beta(\Delta_{n-1}\times[0,N])\subset\Ss$, so $F_\beta(\Delta_{n-1}\times[0,N])=\Ss$. By Lemma \ref{imagebc} we conclude that $\Ss$ is a polynomial image of $\ol{\Bb}_n$, as required.
\end{proof}

\section{Finite unions of $m$-bricks}\label{s5}

Denote by $\psd_d(\R^m,\R^n)$ the space of polynomial maps $f:\R^m\to\R^n$ whose components have degree $\leq d$. This set admits the structure of the affine space $\R^N$ where
$$
N:=\binom{n+d}{d}m.
$$
This $N$ is the number of coefficients necessary to define each polynomial map of $\psd_d(\R^m,\R^n)$. The Euclidean topology of $\R^N$ induces on $\psd_d(\R^m,\R^n)$ the compact-open topology. 

If $\Ss$ is an $m$-brick, there exists a homotopy $H:=(H_1,\ldots,H_n):[0,1]\times\ol{\Bb}_m\to\Ss$ such that:
\begin{itemize}
\item[(i)] $H_i\in{\mathcal C}^0([0,1])[\x_1,\ldots,\x_m]$ for $i=1,\ldots,n$.
\item[(ii)] $H(\{0\}\times\ol{\Bb}_m)=\Ss$ and $H(1,\cdot)$ is a constant map.
\item[(iii)] $H(\{t\}\times\ol{\Bb}_m)\subset\Int(\Ss)$ for each $t\in(0,1)$.
\end{itemize}
When we want to stress the associate homotopy $H$, we write $(\Ss,H)$. Define 
$$
\deg(H):=\max\{\deg(H_i):\ i=1,\ldots,n\}.
$$
If $d:=\deg(H)$, then $H_t:=H(t,\cdot)\in\psd_d(\R^m,\R^n)$ for each $t\in[0,1]$. As $\Ss$ is a polynomial image of $\ol{\Bb}_m$ of dimension $n$, it is pure dimensional of dimension $n$, so $\Ss=\cl(\Int(\Ss))$. If $p\in\R^n$, we denote by $F_p:\R^m\to\R^n$ the constant polynomial map that values $p$.

\begin{lem}\label{basico} 
Let $(\Ss,H)$ be an $m$-brick and let $d\geq\deg(H)$. Consider the semialgebraic set
$$
\Omega_\Ss:=\{F\in\psd_d(\R^m,\R^n):\ F(\ol{\Bb}_m)\subset\Int(\Ss)\}.
$$
We have:
\begin{itemize}
\item[(i)] $\Omega_\Ss$ is an open semialgebraic set in $\psd_d(\R^m,\R^n)$.
\item[(ii)] If $F\in\cl(\Omega_\Ss)$, then $F(\ol{\Bb}_m)\subset\Ss$.
\item[(iii)] If $F_{y_0}\in\psd_d(\R^m,\R^n)$ and $y_0\in\Ss$, then $F_{y_0}\in\cl(\Omega_\Ss)$. In addition, if $y_0\in\Int(\Ss)$, then $F_{y_0}\in\Omega_\Ss$.
\item[(iv)] There exists $F\in\cl(\Omega_\Ss)$ such that $F(\ol{\Bb}_m)=\Ss$.
\item[(v)] Both $\Int(\Ss)$ and $\Omega_\Ss$ are connected.
\end{itemize}
\end{lem}
\begin{proof}
(i) The semialgebraicity of $\Omega_\Ss$ follows because it is described by a formula in first order language. By definition $\Omega_\Ss$ is an open subset of $\psd_d(\R^m,\R^n)$ endowed with the compact-open topology.

(ii) Let $F\in\cl(\Omega_\Ss)$ and $x_0\in\ol{\Bb}_m$. Write $y_0:=F(x_0)$. As $F\in\cl(\Omega_\Ss)$, there exists a sequence $\{F_m\}_m\subset\Omega_\Ss$ that converges to $F$. As $F_m(x_0)\in\Int(\Ss)$ for each $m\geq1$, we deduce $y_0=F(x_0)=\lim_{m\to\infty}F_m(x_0)\in\cl(\Int(\Ss))=\Ss$.

(iii) As $\cl(\Int(\Ss))=\Ss$, there exists a sequence $\{y_k\}_k\subset\Int(\Ss)$ that converges to $y_0$. The sequence of constant polynomial maps $\{F_{y_k}\}_k$ is contained in $\Omega_\Ss$ and converges to $F_{y_0}$, so $F_{y_0}\in\cl(\Omega_\Ss)$. If $F_{y_0}(\ol{\Bb}_m)=\{y_0\}\subset\Int(\Ss)$, it is clear that $F_{y_0}\in\Omega_\Ss$. 

(iv) Let $F:=H(0,\cdot)$, which satisfies $F(\ol{\Bb}_m)=\Ss$, and let $\{t_k\}_k\subset(0,1)$ be a sequence that converges to $0$. Define $F_k:=H(t_k,\cdot)$ and observe that the sequence $\{F_k\}_k\subset\Omega_\Ss$ converges to $F$, so $F\in\cl(\Omega_\Ss)$.

(v) We prove first: \em $\Int(\Ss)$ is connected\em. 

As $\Tt:=H((0,1)\times\ol{\Bb}_m)\subset\Int(\Ss)$ is connected, it must be contained in one of the components of $\Int(\Ss)$. Let $\Cc$ be another connected component of $\Int(\Ss)$, which is an open semialgebraic subset of $\R^n$. The inverse image $H^{-1}(\Cc)$ is an open subset of $[0,1]\times\ol{\Bb}_m$ that does not meet $(0,1)\times\ol{\Bb}_m$, so $H^{-1}(\Cc)\subset\{0,1\}\times\ol{\Bb}_m$, which is a contradiction.

We show next: \em $\Omega_\Ss$ is connected by (continuous) semialgebraic paths\em.

We claim: \em for each $G\in\Omega_\Ss$ there exists a (continuous) semialgebraic path $\phi:[0,1]\to\Omega_\Ss$ that connects the constant polynomial map $F_{G(0)}$ with $G$\em.

Consider the (continuous) semialgebraic map
$$
\phi:[0,1]\to\psd_d(\R^m,\R^n),\ t\mapsto G_t
$$
where $G_t(\x)=G(t\x)\in\psd_d(\R^m,\R^n)$ for each $t\in[0,1]$. Observe that 
$$
G_t(\ol{\Bb}_m)=G(t\ol{\Bb}_m)=G(\ol{\Bb}_m(0,t))\subset G(\ol{\Bb}_m)\subset\Int(\Ss),
$$ 
so $\im(\phi)\subset\Omega_\Ss$. Observe that $G_0$ is a constant polynomial map such that $G_0(\ol{\Bb}_m)=\{G(0)\}\subset\Int(\Ss)$. 

Thus, to prove that $\Omega_\Ss$ is connected by (continuous) semialgebraic paths it is enough to show: \em Given two constant maps $F_p,F_q\in\Omega_\Ss$ there exists a (continuous) semialgebraic path $F_\alpha:[0,1]\to\Omega_S$ such that $F_\alpha(0)=F_p$ and $F_\alpha(1)=q$\em.

As $F_p,F_q\in\Omega_\Ss$ are constant maps, $p,q\in\Int(\Ss)$. As $\Int(\Ss)$ is connected, there exists a (continuous) semialgebraic path $\alpha:[0,1]\to\Int(\Ss)$ such that $\alpha(0)=p$ and $\alpha(1)=q$. The constant polynomial map $F_{\alpha(t)}\in\Omega_\Ss$ for each $t\in[0,1]$, so $F_\alpha:[0,1]\to\Omega_\Ss$ provides a (continuous) semialgebraic path that connects $F_p$ and $F_q$, as required.
\end{proof}

\begin{cor}\label{freedom2}
Let $\Ss_1$, $\Ss_2$ be two $m$-bricks and assume there exists a bridge $\Gamma$ between $\Int(\Ss_1)$ and $\Int(\Ss_2)$ with base point $q$. Then there exists a bridge (of constant polynomial maps) between $\Omega_{\Ss_1}$ and $\Omega_{\Ss_2}$ with base point the constant map $F_q$. 
\end{cor}
\begin{proof}
Let $\alpha:[-1,1]\to\Ss$ be a polynomial arc such that $\alpha([-1,0))\subset\Int(\Ss_1)$, $\alpha(0)=q$ and $\alpha((0,1])\subset\Int(\Ss_2)$. By Lemma \ref{basico}(iii) $F_{\alpha(t)}\in\Omega_{\Ss_1}$ for each $t\in[-1,0)$, whereas $F_{\alpha(t)}\in\Omega_{\Ss_2}$ for each $t\in(0,1]$. Thus, $F_\alpha:[-1,1]\to\psd(\R^m,\R^n),\ t\mapsto F_{\alpha(t)}$ defines a bridge between $\Omega_{\Ss_1}$ and $\Omega_{\Ss_2}$ with base point $F_q$, as required.
\end{proof}

\subsection{Proof of Theorem \ref{main2-i}}

We are ready to prove Theorem \ref{main2-i}.

\begin{proof}[Proof of Theorem \em\ref{main2-i}]
(ii) $\implies$ (i) By \cite[Thm.1.5]{f2} there exists a Nash map $g:\R^{m+1}\to\R^{m+1}$ such that $g(\R^{m+1})=\ol{\Bb}_{m+1}$. Consider the Nash map $F\circ g:\R^{m+1}\to\R^n$ whose image is $\Ss$. By \cite[Main Thm.1.4]{f2} the semialgebraic set $\Ss$ is connected by analytic paths.

(i) $\implies$ (ii) Let $(\Ss_i,H_i)$ be $m$-bricks for $i=1,\ldots,r$ and let $d:=\max\{\deg(H_i):\ i=1,\ldots,r\}$. The restriction map $H_{i,t}:=H_i(t,\cdot)$ belongs to $\psd_d(\R^m,\R^n)$ for each $t\in[0,1]$ and each $i=1,\ldots,r$. By Lemma \ref{graph} we may assume that there exists a bridge $\Gamma_k$ between $\Int(\Ss_k)$ and $\Int(\Ss_{k+1})$ with base point $q_k$ for $1\leq k\leq\ell-1$. 

By Lemma \ref{basico} each $\Omega_{\Ss_i}$ is a connected open semialgebraic subset of $\psd_d(\R^m,\R^n)$. By Corollary \ref{freedom2} there exist bridges $\Delta_k$ in $\psd_d(\R^m,\R^n)$ connecting $\Omega_{\Ss_k}$ and $\Omega_{\Ss_{k+1}}$ with base point the constant map $F_{q_k}$ for $k=1,\dots,\ell-1$. 

For each $k=1,\ldots,\ell$ we pick a polynomial map $F_k\in\cl(\Omega_{\Ss_k})$ such that $F_k(\ol{\Bb}_m)=\Ss_k$ (use Lemma \ref{basico}(iv)). Fix real values $0=t_1<\cdots<t_\ell=1$ and $s_k\in(t_k,t_{k+1})$ for $k=1,\ldots,\ell-1$. By Lemma \ref{smart} there exists a polynomial path $\phi:\R\to\psd_d(\R^m,\R^n)$ that satisfies: 
\begin{itemize}
\item[(i)] $\phi([0,1])\subset\bigcup_{k=1}^\ell\Omega_{\Ss_k}\cup\{F_1,\ldots,F_\ell,F_{q_1},\ldots,F_{q_{\ell-1}}\}$.
\item[(ii)] $\phi(t_k)=F_k$ for $k=1,\ldots,\ell$.
\item[(iii)] $\phi((t_k,s_k))\subset\Omega_{\Ss_k}$, $\phi((s_k,t_{k+1}))\subset\Omega_{\Ss_{k+1}}$ and $\phi(s_k)=F_{q_k}$.
\end{itemize}
Thus, $\phi(t)(\ol{\Bb}_m)\subset\Int(\Ss_k)$ for each $t\in(t_k,s_k)$, whereas $\phi(t)(\ol{\Bb}_m)\subset\Int(\Ss_{k+1})$ if $t\in(s_k,t_{k+1})$ for each $k$. Consider the polynomial map
$$
\Phi:[0,1]\times\ol{\Bb}_m\to\R^n,\ (t,x)\mapsto\phi(t)(x)
$$ 
and observe that 
$$
\Ss=\bigcup_{k=1}^\ell\Ss_k=\bigcup_{k=1}^\ell F_k(\ol{\Bb}_m)\subset\Phi([0,1]\times\ol{\Bb}_m)\subset\bigcup_{k=1}^\ell\Ss_k\cup\bigcup_{k=1}^\ell\Int(\Ss_k)\cup\bigcup_{k=1}^{\ell-1}\{q_k\}=\Ss.
$$
Thus, $\Phi([0,1]\times\ol{\Bb}_m)=\Ss$ and by Lemma \ref{imagebc} $\Ss$ is a polynomial image of $\ol{\Bb}_{m+1}$, as required.
\end{proof}

\appendix
\section{Alternative models}\label{a1}

A regular map on a semialgebraic set $\Ss\subset\R^m$ is the restriction to $\Ss$ of a rational map $f:=(f_1,\ldots,f_n):\R^m\to\R^n$ such that $f_k=\frac{g_k}{h_k}$ where $g_k,h_k\in\R[\x_1,\ldots,\x_m]$ and $h_k$ does not vanish on $\Ss$. We prove here some comments made in the Introduction (\S\ref{am}). The interval $[-1,1]$ is the image of the regular function $f:\R\to\R,\ t\mapsto\frac{2t}{t^2+1}$. By \cite[Prop.1.4]{f1} there exists no regular function $f:\R\to\R$ such that $f(\R)=(-1,1)$. The situation changes for $n\geq2$.

\begin{lem}\label{ocb}
The open and closed unit balls $\Bb_n$ and $\ol{\Bb}_n$ are regular images of $\R^n$ if $n\geq2$.
\end{lem}
\begin{proof}
Consider the polynomial map
$$
F:\R^n\rightarrow\R^n,\ (x_1,\dots,x_n)\mapsto ((x_1x_2-1)^2+x_1^2, x_2(x_1x_2-1),x_3,\dots,x_n),
$$
whose image is $\HH_n:=\{\x_1>0\}$, see \cite[Ex.1.4(iv)]{fg1}. Let ${\tt e}_{n+1}:=(0,\dots,0,1)\in\R^{n+1}$, and let $\rho:\sph^n\setminus\{{\tt e}_{n+1}\}\rightarrow\R^n$ be the stereographic projection from the north pole. Its inverse map $G:=\rho^{-1}$ is the regular map
\begin{equation}\label{sTInv}
G:\R^n\rightarrow\sph^n\setminus\{{\tt e}_{n+1}\}, (x_1,\dots,x_n)\mapsto \Big(\frac{2x_1}{\|x\|^2+1},\ldots,\frac{2x_n}{\|x\|^2+1},\frac{\|x\|^2-1}{\|x\|^2+1}\Big).
\end{equation}
Notice that $G(\HH_n)=\sph^n\cap\HH_{n+1}$, where $\HH_{n+1}:=\{\x_1>0\}$. 

Consider the projection $\pi_1:\R^{n+1}\rightarrow\R^n,\ (x_1,\dots,x_{n+1})\mapsto (x_2,\dots,x_{n+1})$. The composition $f:=\pi_{n+1}\circ G\circ F$ satisfies the equality $f(\R^n)=\Bb_n$.

Next, we proceed with the closed unit ball $\ol{\Bb}_n$. Consider the projection 
$$
\pi:\R^{n+1}\to\R^n,\, (x_1,\dots,x_{n+1})\mapsto (x_1,\dots,x_n).
$$
The composition $\pi\circ G$ satisfies $(\pi\circ G)(\R^n)=\ol{\Bb}_n$, as required.
\end{proof}

We conclude the following result announced in \S\ref{am}.

\begin{cor}\label{images}
The family of regular images of $\ol{\Bb}_n$ is a subfamily of the family of the regular images of $\R^n$.
\end{cor}

We can also compare the family of regular images of the open ball $\Bb_n$ with that of $\R^n$. As we have already commented, if $n=1$ both families are different. For $n\geq2$ they are equal as a consequence of Lemma \ref{ocb} and the following result.

\begin{lem}
The regular map $f:\Bb_n\to\R^n,\ x\mapsto\frac{x}{1-\|x\|^2}$ is surjective.
\end{lem}
\begin{proof}
Pick a point $y\in\R^n$ and let us show that there exists $\lambda\in\R$ such that $f(\lambda y)=y$. If $y=0$, take any $\lambda\in\R$. Assume $y\neq0$ and write
$$
y=f(\lambda y)=\frac{\lambda y}{1-\lambda^2\|y\|^2}\quad\leadsto\quad\|y\|=\frac{\lambda\|y\|}{1-\lambda^2\|y\|^2}\quad\leadsto\quad\lambda^2\|y\|^2+\lambda-1=0.
$$
It is enough to take $\lambda:=\frac{-1+\sqrt{1+4\|y\|^2}}{2\|y\|^2}$.
\end{proof}

The next result (lent by A. Carbone) allows us to compare the regular images of the closed unit ball $\ol{\Bb}_n$ with those of the $n$-sphere $\sph^n$.

\begin{lem}\label{nsph}
The $n$-sphere $\sph^n\subset\R^{n+1}$ is a regular image of $[-1,1]^n$.
\end{lem}
\begin{proof}
We proceed by induction on the dimension $n$. Assume first $n=1$. Consider the inverse 
$$
f_0:\R\to\sph^1\setminus\{(0,1)\},\ t\mapsto\Big(\frac{2t}{t^2+1},\frac{t^2-1}{t^2+1}\Big)
$$
of the stereographic projection. We have $f_0([-1,1])=\sph^1\cap\{\y\leq0\}$. Consider the polynomial map 
$$
f_1:\R^2\equiv\C\to\C\equiv\R^2,\ (x,y)\equiv x+y\sqrt{-1}=:z\mapsto z^2=x^2-y^2+2xy\sqrt{-1}\equiv(x^2-y^2,2xy).
$$
The image of $[-1,1]$ under the regular map $f:=f_1\circ f_0:\R\to\R^2$ is $\sph^1$.

Assume the result is true for dimension $n-1$ and we check that it is true for dimension $n$.

Let $g:\R^{n-1}\to\R^n$ be a regular map such that $g([-1,1]^{n-1})=\sph^{n-1}$ and let $f_1:=(f_{11},f_{12}):\R\to\R^2$ be the regular map described above such that $f_1([-1,1])=\sph^1$. Denote $x':=(x_1,\ldots,x_{n-1})$ and ${\tt e}_{n+1}:=(0,\ldots,0,1)$. Consider the regular map
$$
f:\R^n\to\R^{n+1},\ (x',x_n)\mapsto (g(x'),0)f_{11}(x_n)+{\tt e}_{n+1}f_{12}(x_n).
$$
A straightforward computation shows that $f([-1,1]^n)=\sph^n$.
\end{proof}

By Corollary \ref{bq} and Lemma \ref{nsph} the sphere $\sph^n$ is a regular image of $\ol{\Bb}_n$. Consequently, having in mind that $\ol{\Bb}_n$ is the image of $\sph^n$ under the projection $\pi:\R^{n+1}\to\R^n,\, (x_1,\dots,x_{n+1})\mapsto (x_1,\dots,x_n)$, we deduce the following result announced in \S\ref{am}.

\begin{cor}\label{images2}
The family of regular images of $\ol{\Bb}_n$ and the family of the regular images of $\sph^n$ coincide.
\end{cor}

\section{Parabolic, elliptic and hyperbolic sectors and segments}\label{a2}

We approach in this appendix the postponed proofs of most of the results of \S\ref{quadratic}.

\subsection{Parabolic segments}

We begin by proving Lemma \ref{parab}.

\begin{proof}[Proof of Lemma \em\ref{parab}] 
Let $\ptri{a}:=\{0\leq\y\leq\x,\x\leq\sqrt{a}\}$ be the triangle of vertices $(0,0)$, $(\sqrt{a},0)$ and $(\sqrt{a},\sqrt{a})$. Consider the polynomial map $\eta:\R^2\to\R^2,\ (x,y)\mapsto(xy,y)$. By Lemma \ref{tetra} $\ptri{a}$ is a polynomial image of $\ol{\Bb}_2$. We claim: $\eta(\ptri{a})=\pseg{a}:=\{\x^2-\y\geq0,\sqrt{a}\y-\x\geq0\}$. 

To show the inclusion $\eta(\ptri{a})\subset\pseg{a}$ pick a point $(x_0,y_0)\in\eta(\ptri{a})$. Thus, there exists $(u_0,v_0)\in\ptri{a}$ such that $\eta(u_0,v_0)=(u_0v_0,v_0)=(x_0,y_0)$. Observe that $0\le v_0\le u_0$ and $u_0\leq\sqrt{a}$. Therefore, $x_0-y_0^2=u_0v_0-v_0^2=v_0(u_0-v_0)\ge 0$ and $\sqrt{a}y_0-x_0=\sqrt{a}v_0-u_0v_0=v_0(\sqrt{a}-u_0)\ge 0$. Consequently, $(x_0,y_0)\in\pseg{a}$.

To show the inclusion $\pseg{a}\subset\eta(\ptri{a})$ pick a point $(x_0,y_0)\in\pseg{a}$. Thus, $x_0-y_0^2\ge0$ and $\sqrt{a}y_0-x_0\ge 0$. If $y_0=0$, then $x_0=0$ and the image of $(0,0)\in\ptri{a}$ under $\eta$ is $(0,0)=(x_0,y_0)$. If $y_0\neq 0$, the pair $(u_0,v_0):=(\frac{x_0}{y_0},y_0)\in\ptri{a}$ satisfies $\eta(u_0,v_0)=(x_0,y_0)$, as required.
\end{proof}

\subsection{Elliptic sectors and segments}

The proof of Theorem \ref{circdef} is done after some preliminary work that we develop next. We prove only the $2$-dimensional case. Due to the nature of the involved polynomial maps, the $n$-dimensional case follows from Lemma \ref{revolution} (for $m=k=2$ and $\ell=n-2$) and the $2$-dimensional case.

\begin{lem}\label{circ0}
The polynomial map $\phi_0:\R^2\to\R^2,\ (x_1,x_2)\mapsto\frac{3-(x_1^2+x_2^2)}{2}(x_1,x_2)$ satisfies $\phi_0(\tri{\alpha})=\sector{\alpha}$ whenever $0<\alpha\le\arcsin\Big({\sqrt{\tfrac{2}{3}}}\Big)$.
\end{lem}
\begin{proof}
We show first the inclusion $\phi_0(\tri{\alpha})\subset\sector{\alpha}$. Pick a point $(r_0,\beta_0)\in\phi_0(\tri{\alpha})$ (in polar coordinates). Then there exists $(\rho_0,\theta_0)\in\tri{\alpha}=\{(\rho,\theta):\ -\alpha\leq\theta\leq\alpha,\ 0\leq\rho\leq\frac{1}{\cos(\theta)}\}$ with $\phi_0(\rho_0,\theta_0)=(r_0,\beta_0)$. In polar coordinates we obtain
\begin{multline*}
\phi_0(\rho_0,\theta_0)\equiv\phi_0(\rho_0\cos(\theta_0),\rho_0\sin(\theta_0))\\
=\frac{3-\rho_0^2}{2}(\rho_0\cos(\theta_0),\rho_0\sin(\theta_0))\equiv\Big(\frac{3\rho_0-\rho_0^3}{2},\theta_0\Big)=(r_0,\beta_0).
\end{multline*}
Thus, $\beta_0=\theta_0$, so $-\alpha\leq\beta_0\leq\alpha$. Next, let us consider the continuous function 
\begin{equation}\label{h0}
h:[0,+\infty)\to\R,\ \rho\mapsto\frac{3\rho-\rho^3}{2}.
\end{equation}
The function $h$ is increasing on the interval $[0,1]$ and decreasing on the interval $[1,+\infty)$. In addition, $h(0)=0$ and $h(1)=1$. Besides, $h$ is nonnegative on the interval $[0,\sqrt{3}]$.
As 
$$
1\leq\frac{1}{\cos(\beta_0)}\leq\frac{1}{\cos(\alpha)}\leq\frac{1}{\cos\Big(\arcsin\Big(\sqrt{\frac{2}{3}}\Big)\Big)}=\sqrt{3},
$$
we deduce $h([0,\frac{1}{\cos(\beta_0)}])=[0,1]$, so 
$$
(r_0,\beta_0)=\phi_0(\rho_0,\beta_0)=(h(\rho_0),\beta_0)\in\sector{\alpha}=\{\ 0\leq\rho\le 1,-\alpha\leq\theta\leq\alpha\}.
$$

To prove the inclusion $\sector{\alpha}\subset\phi_0(\tri{\alpha})$ pick a point $(r_0,\beta_0)\in\sector{\alpha}$ and consider the function $h$ defined in \eqref{h0}. As $h([0,1])=[0,1]$, there exists $\rho_0\in[0,1]$ such that $h(\rho_0)=r_0$. Therefore, $(r_0,\beta_0)=\phi_0(\rho_0,\beta_0)\in\phi_0(\sector{\alpha})\subset\phi_0(\tri{\alpha})$, because $\sector{\alpha}\subset\tri{\alpha}$. Thus, $\sector{\alpha}\subset\phi_0(\tri{\alpha})$, as required.
\end{proof}

The previous lemma sets a limit to the amplitude of the elliptic sectors we work with. This limit can be dealt with by means of the following lemma.

\begin{lem}\label{circ1}
The polynomial map $\phi_1:\R^2\to\R^2, (x_1,x_2)\mapsto (x_1^2-x_2^2, 2x_1x_2)$ satisfies $\phi_1(\sector{\alpha})=\sector{2\alpha}$.
\end{lem}
\begin{proof}
We identify $\R^2\equiv\C$ using the standard map $(x_1,x_2)\mapsto x_1+\sqrt{-1}x_2$. We can interpret $\phi_1:\C\to\C,\ z:=x_1+\sqrt{-1}x_2\mapsto z^2=(x_1^2-x_2^2)+\sqrt{-1}(2x_1x_2)$. Using this fact the statement follows readily.
\end{proof}

\begin{lem}\label{circ2}
Fix an angle $0<\alpha<\frac{\pi}{2}$. The polynomial map
$$
\phi_2:\R^2\to\R^2,\ (x_1,x_2)\mapsto (x_1^2-x_2^2+(1-x_1^2-x_2^2)\cos(2\alpha),2x_1x_2)
$$
satisfies $\phi_2(\sector{\alpha})=\seg{2\alpha}$.
\end{lem}
\begin{proof}
We use again polar coordinates. Observe that
\begin{multline*}
\phi_2(\rho,\theta)\equiv\phi_2(\rho\cos(\theta),\rho\sin(\theta))=(\rho^2\cos(2\theta)+(1-\rho^2)\cos(2\alpha),\rho^2\sin(2\theta))\\
=(1-\rho^2)(\cos(2\alpha),0)+\rho^2(\cos(2\theta),\sin(2\theta)).
\end{multline*}
As $(\rho,\theta)\in\sector{\alpha}$, we have $0\leq\rho\leq1$ and $-\alpha\le\theta\le\alpha$. Fix $\theta\in[-\alpha,\alpha]$. The continuous map
$$
\phi_2(\cdot,\theta):[0,1]\to\R^2, \rho\to(1-\rho^2)(\cos(2\alpha),0)+\rho^2(\cos(2\theta),\sin(2\theta)),
$$
transforms $[0,1]$ onto the segment that connects the midpoint $(\cos(2\alpha),0)$ of the chord of the elliptic segment $\seg{2\alpha}$ with the point $(\cos(2\theta),\sin(2\theta))$ on the arc of $\seg{2\alpha}$. As $-\alpha\le\theta\le\alpha$, we conclude $\phi_2(\sector{\alpha})=\seg{2\alpha}$, as required.
\end{proof}

We are ready to prove Theorem \ref{circdef}.

\begin{proof}[Proof of Theorem \em \ref{circdef}]
Observe first that both elliptic sectors and segments are convex. By Lemma \ref{propbrick} we have to prove that both elliptic sectors and segments are polynomial images of $\ol{\Bb}_n$. By the nature of the polynomial maps proposed in Lemmas \ref{circ0}, \ref{circ1}, \ref{circ2} and Lemma \ref{revolution} it is enough to deal with the $2$-dimensional case. By Lemma \ref{tetra} the triangle $\tri{\alpha}$ is a polynomial image of $\ol{\Bb}_2$ for each $0<\alpha<\pi$. Fix an angle $0<\alpha\leq\pi$ and let $\beta:=\frac{\alpha}{4}<\arcsin\Big(\sqrt{\frac{2}{3}}\Big)$. By Lemma \ref{circ0} the elliptic sector $\sector{\beta}$ is the image under a polynomial map $\phi_0:\R^2\to\R^2$ of $\tri{\beta}$. By Lemma \ref{circ1} the elliptic sector $\sector{\alpha}=\sector{4\beta}$ is the image under a polynomial map $\R^2\to\R^2$ of $\sector{\beta}$, so $\sector{\alpha}$ is the image under a polynomial map $\R^2\to\R^2$ of $\ol{\Bb}_2$ for $0<\alpha\leq\pi$. 

By Lemma \ref{circ2} the elliptic segment $\seg{\alpha}=\seg{4\beta}$ is the image under a polynomial map $\phi_2:\R^2\to\R^2$ of $\sector{2\beta}$. Consequently, $\seg{\alpha}$ is the image under a polynomial map $\R^2\to\R^2$ of $\ol{\Bb}_2$ for $0<\alpha\leq\pi$, as required.
\end{proof}

\subsection{Hyperbolic sectors and segments}

The proof of Theorem \ref{hiperdef} is done after some preliminary work that we develop next. We prove only the $2$-dimensional case. Due to the nature of the involved polynomial maps, the $n$-dimensional case follows from Lemma \ref{revolution} (for $m=k=2$ and $\ell=n-2$) and the $2$-dimensional case.

\begin{lem}\label{hiper0}
The polynomial map $\psi_0:\R^2\to\R^2,\ (x_1,x_2)\mapsto\frac{3-(x_1^2-x_2^2)}{2}(x_1,x_2)$ satisfies $\psi_0(\htri{\alpha})=\hsector{\alpha}$ whenever $0<\alpha\leq\arctan(\sqrt{\frac{2}{3}})$.
\end{lem}
\begin{proof}
We show first the inclusion $\psi_2(\htri{\alpha})\subset\hsector{\alpha}$. Pick a point $(r_0,\beta_0)\in\psi_0(\htri{\alpha})$. Then there exists $(\rho_0,\theta_0)$ such that $-\alpha\le\theta_0\le\alpha$, $\rho_0\cos(\theta_0)\leq\frac{\cos(\alpha)}{\sqrt{\cos(2\alpha)}}$ and 
\begin{multline*}
\psi_0(\rho_0,\theta_0)\equiv\psi_0(\rho_0\cos(\theta_0),\rho_0\sin(\theta_0))=\frac{3-\rho_0^2\cos(2\theta_0)}{2}(\rho_0\cos(\theta_0),\rho_0\sin(\theta_0))\\
\equiv\Big(\frac{3\rho_0-\rho_0^3\cos(2\theta_0)}{2},\theta_0\Big)=(r_0,\beta_0).
\end{multline*}
Thus, $\beta_0=\theta_0$, so $-\alpha\leq\beta_0\leq\alpha$. Consider the continuous function
\begin{equation}\label{rb}
h_{\theta_0}:[0,+\infty)\to\R,\ \rho\mapsto\frac{3\rho-\rho^3\cos(2\theta_0)}{2}.
\end{equation}
This function is increasing on the bounded interval $[0,\frac{1}{\sqrt{\cos(2\theta_0)}}]$ and decreasing on the unbounded interval $[\frac{1}{\sqrt{\cos(2\theta_0)}},+\infty)$. In addition, $h_{\theta_0}(0)=0$ and $h_{\theta_0}(\frac{1}{\sqrt{\cos(2\theta_0)}})=\frac{1}{\sqrt{\cos(2\theta_0)}}$. Besides, it is nonnegative on the interval $[0,\sqrt{\frac{3}{\cos(2\theta_0)}}]$. As $0<\alpha\le\arctan(\sqrt{\frac{2}{3}})$ and $1\leq\frac{\cos(\theta_0)}{\sqrt{\cos(2\theta_0)}}$, we have
$$
\rho_0\cos(\theta_0)\le\frac{\cos(\alpha)}{\sqrt{\cos(2\alpha)}}=\frac{1}{\sqrt{1-\tan^2(\alpha)}}\leq\frac{1}{\sqrt{1-\frac{2}{3}}}=\sqrt{3}\leq\frac{\sqrt{3}\cos(\theta_0)}{\sqrt{\cos(2\theta_0)}}.
$$
As $|\theta_0|\leq\alpha\leq\arctan(\sqrt{\frac{2}{3}})$, we deduce
\begin{align*}
&0\leq\rho_0\leq\frac{\cos(\alpha)}{\cos(\theta_0)\sqrt{\cos(2\alpha)}},\\
&\frac{1}{\sqrt{\cos(2\theta_0)}}\leq\frac{\cos(\alpha)}{\cos(\theta_0)\sqrt{\cos(2\alpha)}}\leq\frac{\sqrt{3}}{\sqrt{\cos(2\theta_0)}}.
\end{align*}
Consequently, 
$$
r_0=h_{\theta_0}(\rho_0)\in h_{\theta_0}\Big(\Big[0,\frac{\cos(\alpha)}{\cos(\theta_0)\sqrt{\cos(2\alpha)}}\Big]\Big)=\Big[0, \frac{1}{\sqrt{\cos(2\theta_0)}}\Big]\quad\leadsto\quad r_0^2\cos(2\theta_0)\leq 1.
$$
This means that $(r_0,\beta_0)=\psi_0(\rho_0,\theta_0)\in\hsector{\alpha}$.

For the converse inclusion $\hsector{\alpha}\subset\psi_0(\htri{\alpha})$, pick a point $(r_0,\beta_0)\in\hsector{\alpha}$ and set $\theta_0:=\beta_0$. We have $-\alpha\leq\theta_0\leq\alpha$ and $r_0^2\cos(2\theta_0)$, so $0\leq r_0\leq\frac{1}{\sqrt{\cos(2\theta_0)}}$. We use again the continuous function $h_{\theta_0}$ already introduced in \eqref{rb}. As
$$
h_{\theta_0}\Big(\Big[0,\frac{1}{\sqrt{\cos(2\theta_0)}}\Big]\Big)=\Big[0, \frac{1}{\sqrt{\cos(2\theta_0)}}\Big],
$$
there exists $\rho_0\in[0,\frac{1}{\sqrt{\cos(2\theta_0)}}]$ such that $h_{\theta_0}(\rho_0)=r_0$, so $(\rho_0,\theta_0)\in\hsector{\alpha}$. As $\hsector{\alpha}\subset\htri{\alpha}$, we conclude $(r_0,\beta_0)=(h_{\theta_0}(\rho_0),\beta_0)=\psi_0(\rho_0,\theta_0)\in\psi_0(\hsector{\alpha})\subset\psi_0(\htri{\alpha})$, as required.
\end{proof}

The previous lemma sets a limit to the amplitude of the hyperbolic sectors we can work with. This limit can be dealt with by means of the following lemma.

\begin{lem}\label{hiper1}
The polynomial map $\psi_1:\R^2\to\R^2,\ (x_1,x_2)\mapsto (x_1^2+x_2^2,2x_1x_2)$ satisfies $\psi_1(\hsector{\alpha})=\hsector{\alpha'}$, where $\alpha':=\arctan(\sin(2\alpha))$.
\end{lem}
\begin{proof}
We rewrite the previous map as $\psi_1(\rho,\theta)\equiv\psi_1(\rho\cos(\theta),\rho\sin(\theta))=(\rho^2,\rho^2\sin(2\theta))$. Recall that $\hsector{\alpha}=\{-\alpha\leq\theta\leq\alpha,\ \rho^2\cos(2\theta)\le 1\}$. Thus, $0\leq\rho^2\leq\frac{1}{\cos(2\theta)}$ and $-\alpha\leq\theta\leq\alpha$. The map
$$
\eta:[-\alpha,\alpha]\to\Hh,\ \theta\to\Big(\frac{1}{\cos(2\theta)},\frac{\sin(2\theta)}{\cos(2\theta)}\Big)
$$
provides a parameterization of the arc of hyperbola that defines the hyperbolic sector $\hsector{\alpha'}$ where $\alpha':=\arctan(\sin(2\alpha))$ (because $\tan(\alpha')=\frac{\sin(2\alpha)}{\cos(2\alpha)}/\frac{1}{\cos(2\alpha)}$). Thus, $\psi_1(\hsector{\alpha})=\hsector{\alpha'}$, as required.
\end{proof}

\begin{lem}\label{hiper2}
Fix $0<\alpha<\frac{\pi}{4}$. Then the polynomial map 
$$
\psi_2:\R^2\to\R^2,\ (x_1,x_2)\mapsto\Big(x_1^2+x_2^2+(1-x_1^2+x_2^2)\frac{1}{\cos(2\alpha)},2x_1x_2\Big)
$$
satisfies $\psi_2(\hsector{\alpha})=\hseg{\alpha'}$, where $\alpha':=\arctan(\sin(2\alpha))$.
\end{lem}
\begin{proof}
We use once more polar coordinates. We have 
\begin{multline*}
\psi_2(\rho\cos(\theta),\rho\sin(\theta))=\Big(\rho^2+\frac{1}{\cos(2\alpha)}-\rho^2\frac{\cos(2\theta)}{\cos(2\alpha)},\rho^2\sin(2\theta)\Big)\\
=\Big(\frac{1}{\cos(2\alpha)},0\Big)(1-(\rho\sqrt{\cos(2\theta)})^2)+(\rho\sqrt{\cos(2\theta)})^2\Big(\frac{1}{\cos(2\theta)},\frac{\sin(2\theta)}{\cos(2\theta)}\Big).
\end{multline*}

As $(\rho,\theta)\in\hsector{\alpha}$, it holds $-\alpha\leq\theta\leq\alpha,\ \rho^2\cos(2\theta)\le 1$. Fix $\theta\in[-\alpha,\alpha]$. The interval $[0,\frac{1}{\sqrt{\cos(2\theta)}}]$ for the fixed $\theta$ provides (in polar coordinates) the segment $\Ee_\theta$ that connects the origin $(0,0)$ with the point $(\frac{\cos(\theta)}{\sqrt{\cos(2\theta)}},\frac{\sin(\theta)}{\sqrt{\cos(2\theta)}})\in\Hh:=\{\x^2-\y^2=1\}$. We have $\hsector{\alpha}=\bigcup_{\theta\in[-\alpha,\alpha]}\Ee_\theta$. Consider the continuous map
\begin{multline*}
\psi_2(\cdot,\theta):\Big[0,\frac{1}{\sqrt{\cos(2\theta)}}\Big]\to\R^2,\\ 
\rho\mapsto\Big(\frac{1}{\cos(2\alpha)},0\Big)(1-(\rho\sqrt{\cos(2\theta)})^2)+(\rho\sqrt{\cos(2\theta)})^2\Big(\frac{1}{\cos(2\theta)},\frac{\sin(2\theta)}{\cos(2\theta)}\Big),
\end{multline*} 
which maps the interval $[0,\frac{1}{\sqrt{\cos(2\theta)}}]$ onto the segment $\Ll_\theta$ that connects the point $(\frac{1}{\cos(2\alpha)},0)$ with the point $(\frac{1}{\cos(2\theta)},\frac{\sin(2\theta)}{\cos(2\theta)})\in\Hh$. Consequently, $\psi_2(\Ee_\theta)=\Ll_\theta$. The map
$$
\eta:[-\alpha,\alpha]\to\Hh,\ \theta\to\Big(\frac{1}{\cos(2\theta)},\frac{\sin(2\theta)}{\cos(2\theta)}\Big)
$$
provides a parameterization of the arc of hyperbola that defines the hyperbolic segment $\hseg{\alpha'}$ where $\alpha':=\arctan(\sin(2\alpha))$ (because $\tan(\alpha')=\frac{\sin(2\alpha)}{\cos(2\alpha)}/\frac{1}{\cos(2\alpha)}$). In addition, $(\frac{1}{\cos(2\alpha)},0)$ is the midpoint of the chord of the hyperbolic segment $\hseg{\alpha'}$. We deduce $\hseg{\alpha'}=\bigcup_{\theta\in[-\alpha,\alpha]}\Ll_\theta$ and 
$$
\psi_2(\hsector{\alpha})=\psi_2\Big(\bigcup_{\theta\in[-\alpha,\alpha]}\Ee_\theta\Big)=\bigcup_{\theta\in[-\alpha,\alpha]}\psi_2(\Ee_\theta)=\bigcup_{\theta\in[-\alpha,\alpha]}\Ll_\theta=\hseg{\alpha'},
$$
as required.
\end{proof}

We are ready to prove Theorem \ref{hiperdef}.

\begin{proof}[Proof of Theorem \em \ref{hiperdef}]
Observe first that both hyperbolic sectors and segments are strictly radially convex. By Lemma \ref{propbrick} we have to prove that both hyperbolic sectors and segments are polynomial images of $\ol{\Bb}_n$. By the nature of the polynomial maps proposed in Lemmas \ref{hiper0}, \ref{hiper1}, \ref{hiper2} and Lemma \ref{revolution} it is enough to deal with the $2$-dimensional case. By Lemma \ref{tetra} the triangle $\htri{\alpha}$ is a polynomial image of $\ol{\Bb}_2$ for each $0<\alpha<\frac{\pi}{4}$. By Lemma \ref{hiper0} the hyperbolic sector $\hsector{\alpha}$ is the image under a polynomial map $\psi_0:\R^2\to\R^2$ of $\htri{\alpha}$ if $0<\alpha\leq\arctan(\sqrt{\frac{2}{3}})$.

Suppose $\arctan(\sqrt{\frac{2}{3}})\leq\alpha<\frac{\pi}{4}$. Consider the continuous function $f:[0,\frac{\pi}{4}]\to[0,\frac{\pi}{4}],\ x\mapsto\arctan(\sin(2x))$, which is strictly increasing and satisfies $x<f(x)$ for each $x\in(0,\frac{\pi}{4})$. The restriction $f|_{[\frac{1}{2},\frac{\pi}{4}]}:[\frac{1}{2},\frac{\pi}{4}]\to[\frac{1}{2},\frac{\pi}{4}]$ is contractive. Thus, $f|_{[\frac{1}{2},\frac{\pi}{4}]}$ has a unique fixed point, which is $\frac{\pi}{4}$. In addition, if $x_0\in[\frac{1}{2},\frac{\pi}{4}]$ and $x_m:=f(x_{m-1})$ for each $m\geq1$, then $\{x_m\}_m$ is an increasing sequence converging to $\frac{\pi}{4}$. If we denote $f^m:=f\circ\overset{m}{\cdots}\circ f$, then $x_m=f^m(x_0)$ for each $m\geq1$.

Denote $x_0:=\arctan(\sqrt{\frac{2}{3}})$ and $x_m:=f(x_{m-1})=f^m(x_0)$ for each $m\geq1$. Let $m\geq1$ be such that $\alpha<x_m$. As $f^m$ is a strictly increasing function, there exists $0<\beta<\arctan(\sqrt{\frac{2}{3}})$ such that $f^m(\beta)=\alpha$. By Lemma \ref{hiper1} there exists a polynomial map $\psi_1:\R^2\to\R^2$ such that $\hsector{\alpha}=\hsector{f^m(\beta)}=(\psi_1\circ\overset{m}{\cdots}\circ\psi_1)(\hsector{\beta})$. By Lemma \ref{hiper0} the hyperbolic sector $\hsector{\beta}$ is the image under a polynomial map $\psi_0:\R^2\to\R^2$ of $\htri{\beta}$. Thus, $\hsector{\alpha}$ is the image of $\ol{\Bb}_2$ under a polynomial map $\R^2\to\R^2$ for each $0<\alpha<\frac{\pi}{4}$.

By Lemma \ref{hiper2} the hyperbolic segment $\hseg{\alpha}$ is the image under a polynomial map $\psi_2:\R^2\to\R^2$ of $\hsector{\beta}$ where $0<\beta:=\frac{\arcsin(\tan(\alpha))}{2}<\alpha<\frac{\pi}{4}$. Consequently, $\hseg{\alpha}$ is the image of $\ol{\Bb}_2$ under a polynomial map $\R^2\to\R^2$ for each $0<\alpha<\frac{\pi}{4}$, as required.
\end{proof}

\section{A convex hexagon as a polynomial image of $\ol{\Bb}_3$}\label{a3}

Let us construct explicitly (Problem \ref{probl2}) a polynomial map $F:\R^3\to\R^2$ such that the image of $\ol{\Bb}_3$ under $F$ is the convex hexagon $\Hh$ of vertices $(0,0),(1,0),(2,1),(2,2),(1,2),(0,1)$. This example is inspired by one proposed by Sara Abentin de Gregorio in her Bachelor's Thesis (supervised by the first author). We have found a polynomial parameterization $\alpha:\R\to\R^2,\ t\mapsto(h(t),h(-t))$ of degree $34$ (the explicit expression of $h$ is presented below) such that
\begin{multline*}
\alpha(-3)=\alpha(3)=(0,0),\quad\alpha(-2)=(1,0),\quad\alpha(-1)=(2,1),\\\alpha(0)=(2,2),\quad\alpha(1)=(1,2),\quad\alpha(2)=(0,1),
\end{multline*}
and $\alpha([-3,3])\subset\Hh$ (Figure \ref{hexagon}).

\begin{figure}[ht]
\centering
\includegraphics[scale=0.35]{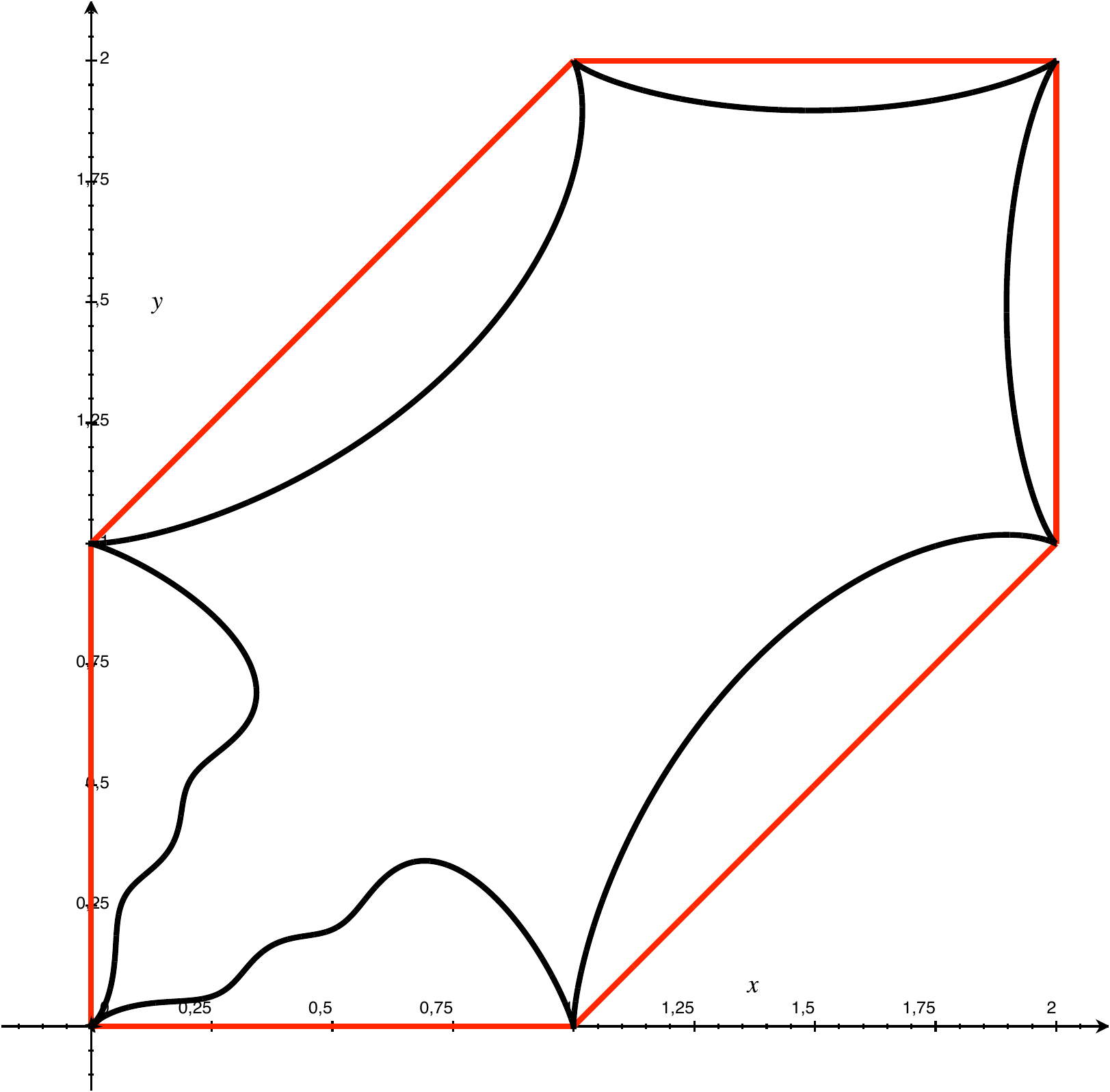}
\caption{Boundary of the convex hexagon $\Hh$ and auxiliary polynomial path $\alpha:[-3,3]\to\R^2$.\label{hexagon}}
\end{figure}

The polynomial $h\in\R[\t]$ is the following:

{\Small\begin{align*} 
h&:=\frac{76664779821250669077010607272790474060504126133999431}{104530224145652815761417086083845114789055289966288790958080000}\t^{34}\\
&-\frac{78717893577241614088318159777360793613982855123}{9304806924249799602963648728748783545276701634902800000}\t^{33}\\
&-\frac{912613527484440059374721830348026508601541166858620900719}{12976165756012073680727638272477324594503415306159987843072000000}\t^{32}\\
&+\frac{83892235220371692608393905913301582885887545485043}{128342164472411029006395154879293566141747608757280000000}\t^{31}\\
&+\frac{22685040173892809529211504901998267843669377577565665759314889}{7526176138487002734822030198036848264811980877572792948981760000000}\t^{30}\\
&-\frac{341651515368750302761136090666887873089921143003186509}{14887691078799679364741837965998053672442722615844480000000}\t^{29}\\
&-\frac{2312987243313117187847123345648030529341762836741812156033242153}{30104704553948010939288120792147393059247923510291171795927040000000}\t^{28}\\
&+\frac{739549879853169825818932859020391314048855175653852243}{1526942674748685063050444919589543966404381806753280000000}\t^{27}\\
&+\frac{2276284563512474920119374134981072455831492763452604504854047611}{1745200263996986141408007002443327133869444841176299814256640000000}\t^{26}\\
&-\frac{1844015505082748933043548683479792565506942192402271751}{269155996904852960266519104470021309332297810003968000000}\t^{25}\\
&-\frac{617678727783923372994726045176856180266061575815960662776038981}{39546409923084415026979469020883274954677075218773296283648000000}\t^{24}\\
&+\frac{191534650791625808889042148881672737990833829426795743}{2793843036134117638234452351113873548663893523968000000}\t^{23}\\
&+\frac{1092968886400506955836901798243093407953206769838766664682946289821}{8027921214386136250476832211239304815799446269410979145580544000000}\t^{22}\\
&-\frac{26872808849517631782827161438579675458460680303908554289}{53831199380970592053303820894004261866459562000793600000}\t^{21}\\
&-\frac{108429710987186801524931157295746273363317657983915070157800658683}{123506480221325173084258957095989304858453019529399679162777600000}\t^{20}\\
&+\frac{21278895451981505723409875956831063512932611500243467814443}{7940101908693162327862313581865628625302785395117056000000}\t^{19}\\
&+\frac{3073703987676894324210366665805746792617837392955601774090905999873}{729811019489648750043348382839936801436313297219179922325504000000}\t^{18}\\ 
&-\frac{15343314371199209768200722187083411564089297923735661841193}{1443654892489665877793147923975568840964142799112192000000}\t^{17}\\ 
&-\frac{10936887264822748056498590600656508762604715166403711626815259269679}{729811019489648750043348382839936801436313297219179922325504000000}\t^{16}\\ 
&+\frac{76038476443389770380096278588852356368626040742215775927}{2460775384925566837147411234049265069825243407577600000}\t^{15}\\ 
&+\frac{6496934190786413999058459578468782136581293765475784303435117586589}{166094921676954543113313769887709754809643715918847844391321600000}\t^{14}\\ 
&-\frac{106279014558443243149270242653074552792653941625109030932479}{1642779705246861171281857982454957646614369392093184000000}\t^{13}\\ 
&-\frac{29977661421428095114360408317401194506069501391456164973924272457861}{408199383782345911041194858198608719447429471325981990453248000000}\t^{12}\\
&+\frac{489624276273750826348881432522588958622036832178716706790257}{5178327331756410213823247988173236059980077431598080000000}\t^{11}\\
&+\frac{5444857799367724341614972148198281098432336302759938935926034370597}{56587790514939870186631806000277054622646472763705210142720000000}\t^{10}\\ 
&-\frac{14424739984071336997741067336445906355369656570801663620093}{157541704537562744600442729798921202882991773712640000000}\t^9\\ 
&-\frac{11695393745625246285991189606549798699384856308791698561139404628187}{139373632194203754348556114778460153052073719955051721277440000000}\t^8\\ 
&+\frac{1631979839386110496201193087227698207675223532405371249967}{30633109215637200338974975238679122782803955999680000000}\t^7\\ 
&+\frac{1135847093359630638538403956445142918501566447380671798346253777}{25470327520870569142645488811853098145481308471317931520000000}\t^6\\ 
&-\frac{7362117782018690715541516858507252889735560321011407953}{510551820260620005649582920644652046380065933328000000}\t^5\\ 
&-\frac{132022313339010089934748996284139958625614088083058661615643}{11558411055895884489273367067925573721788800978176819200000}\t^4-\frac{1}{6}\t^3-\frac{1}{8}\t^2+2.
\end{align*}} 

Consider the six triangules $\Tt_1,\ldots,\Tt_6$ of vertices $(1,1)$ and two consecutive vertices of $\Hh$. Define 
$$
\alpha_1:[-1,1]\to \R^2,\ t\mapsto\alpha\Big(\frac{5}{2}t-\frac{1}{2}\Big)\quad\text{and}\quad
\alpha_2:[-1,1]\to \R^2,\ t\mapsto\alpha\Big(\frac{5}{2}t+\frac{1}{2}\Big), 
$$
see Figures \ref{cyan} and \ref{pink}.

\begin{center}
\begin{figure}[ht]
\begin{minipage}{0.49\textwidth}
\begin{center}
\includegraphics[width=0.75\textwidth]{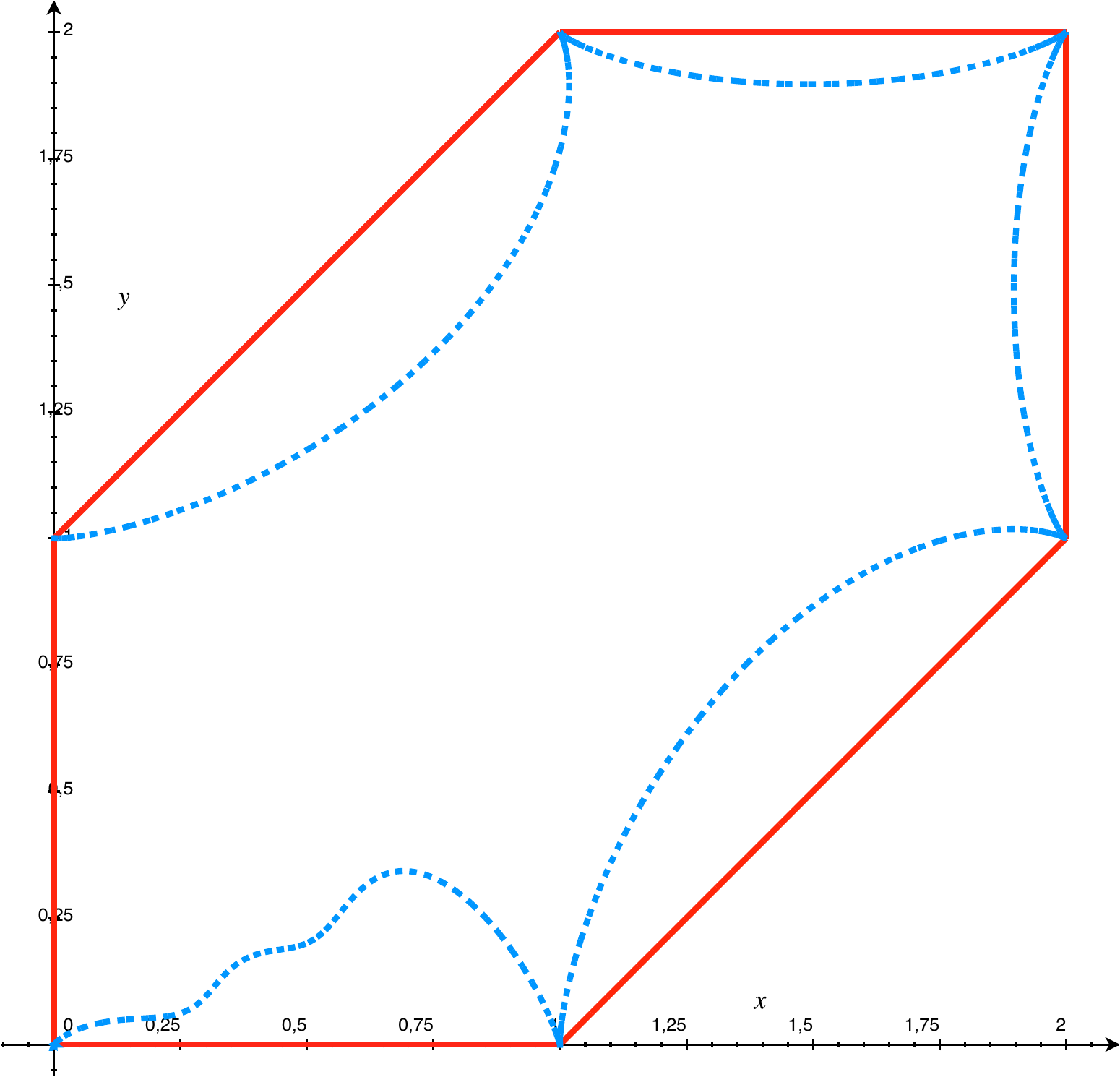}
\caption{Parameterization of $\alpha_1:[-1,1]\to \R^2$\label{cyan}}
\end{center}
\end{minipage}
\hfil
\begin{minipage}{0.49 \textwidth}
\begin{center}
\hspace*{-10mm}\includegraphics[width=0.75\textwidth]{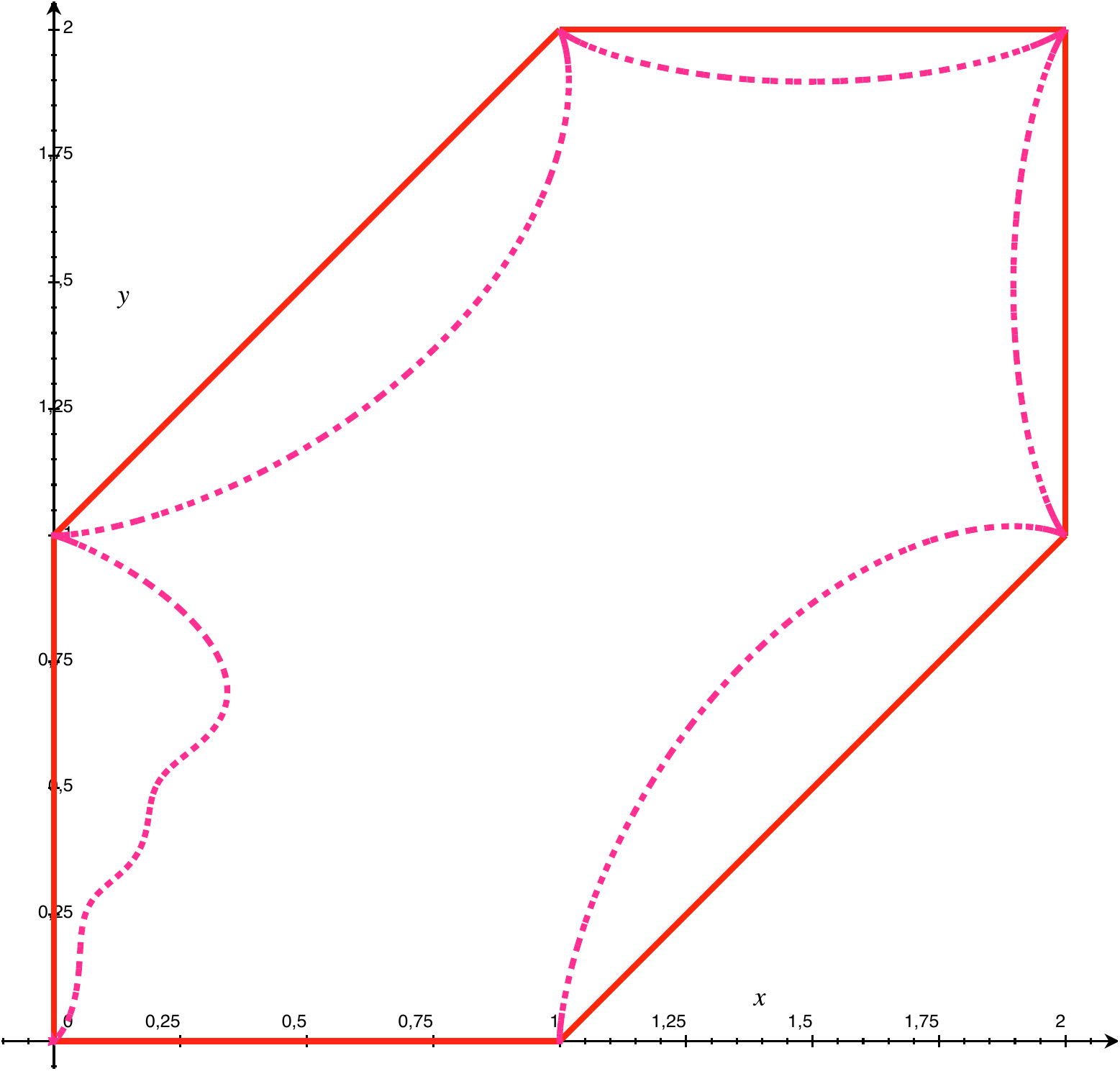}\\
\caption{Parameterization of $\alpha_2:[-1,1]\to \R^2$\label{pink}}
\end{center}
\end{minipage}
\end{figure}
\end{center}

The parameterization $\alpha_1$ begins on vertex $(0,0)$ and goes all over the remaining vertices in counterclockwise direction (Figure \ref{cyan}). The parameterization $\alpha_2$ begins on vertex $(1,0)$ and goes all over the remaining vertices in counterclockwise direction (Figure \ref{pink}). Denote $\Delta_2:=\{(x,y)\in\R^2:\ x\geq0,y\geq0,1-x-y\geq0\}$. We construct the polynomial map
\begin{equation*}
G:\R^3\to \R^2,\ (\lambda,\mu,t)\mapsto\lambda\alpha_1(t)+\mu\alpha_2(t)+(1-\lambda-\mu)(1,1)
\end{equation*}
that maps the triangular prism $\Delta_2\times[-1,1]$ into the convex hexagon $\Hh$. Observe that $\Hh=\bigcup_{j=1}^6\Tt_j=\bigcup_{k=1}^6G(\Delta_2\times\{t_k\})$ where $t_k\in\{-1,-\frac{3}{5},-\frac{1}{5},\frac{1}{5},\frac{3}{5},1\}$ (because $G(\Delta_2\times{t_k})$ is one of the six triangules $\Tt_j$ considered above for each $k=1,\ldots,6$), so $G(\Delta_2\times[-1,1])=\Hh$. It only remains to compose the previous polynomial map with the polynomial map provided in Corollary \ref{tetra1} 
$$
H:\R^3\to \R^3,\ (x,y,z)\mapsto\Big(3\Big(1-\frac{4}{9}(x^2+y^2)\Big)^2x^2,3\Big(1-\frac{4}{9}(x^2+y^2)\Big)^2y^2,3z-4z^3\Big)
$$
that transforms $\overline{\Bb}_{3}$ onto $\Delta_2\times[-1,1]$. Observe that the polynomial paths $\alpha_1$, $\alpha_2$ and $\alpha_3$ are not the ones suggested in the proof of Theorem \ref{main1-i}, but there we analyzed a general case and here we take advantage of the convexity and symmetry of $\Hh$. 

\bibliographystyle{amsalpha}

\end{document}